\documentclass[11pt,reqno]{amsart}
\usepackage{graphicx,enumerate,amssymb,bbm}
\usepackage[left=3.0cm,right=3.0cm,top=2.5cm,bottom=2.5cm,includeheadfoot]{geometry} 
\usepackage{multirow}

\numberwithin{equation}{section}
\newcommand{\etalchar}[1]{$^{#1}$}
\providecommand{\bysame}{\leavevmode\hbox to3em{\hrulefill}\thinspace}
\providecommand{\MR}{\relax\ifhmode\unskip\space\fi MR }

\providecommand{\href}[2]{#2}
\theoremstyle{plain}
\newtheorem{theorem}{Theorem}[section]
\newtheorem{lemma}[theorem]{Lemma}
\newtheorem{proposition}[theorem]{Proposition}
\newtheorem{corollary}[theorem]{Corollary}

\theoremstyle{definition}
\newtheorem{definition}{Definition}[section]

\newtheorem{remark}{Remark}[section]

\newcommand{\bitem}{\begin{itemize}}
\newcommand{\eitem}{\end{itemize}}
\newcommand{\mc}[1]{\mathcal{#1}}

\newcommand{\N}{\mathbb{N}}
\newcommand{\R}{\mathbb{R}}

\newcommand{\Rmn}{\mathbb{R}^{m\times n}}

\newcommand{\Z}{\mathbb{Z}}

\newcommand{\bbd}{\mathbbm{d}}
\newcommand{\bpm}{\begin{pmatrix}}
\newcommand{\epm}{\end{pmatrix}}
\newcommand{\bsm}{\left(\begin{smallmatrix}}
\newcommand{\esm}{\end{smallmatrix}\right)}
\newcommand{\T}{\top}

\newcommand{\ol}[1]{\overline{#1}}
\newcommand{\la}{\langle}
\newcommand{\ra}{\rangle}

\newcommand{\mrm}[1]{\mathrm{#1}}

\newcommand{\veps}{\varepsilon}

\newcommand{\gdw}{\Leftrightarrow}

\newcommand{\eins}{\mathbbm{1}}

\DeclareMathOperator{\rank}{rank}
\DeclareMathOperator{\spark}{spark}

\DeclareMathOperator{\sign}{sign}

\DeclareMathOperator{\TV}{TV}

\title[TV-Based Tomographic Reconstruction]{Phase Transitions and Cosparse Tomographic Recovery of Compound Solid Bodies from Few Projections}
\author[A.~Deni\c{t}iu, S.~Petra, Cl.~Schn\"{o}rr, Ch.~Schn\"{o}rr]{Andreea Deni\c{t}iu, Stefania Petra, Claudius Schn\"{o}rr,
Christoph Schn\"{o}rr}
\address[A.~Deni\c{t}iu, S.~Petra, Ch.~Schn\"{o}rr]{Image and Pattern Analysis Group, University of Heidelberg, Speyerer Str.~6, 69115 Heidelberg, Germany}
\email{\{denitiu,petra,schnoerr\}@math.uni-heidelberg.de}
\urladdr{iwr.ipa.uni-heidelberg.de}
\address[A.~Deni\c{t}iu, Cl.~Schn\"{o}rr]{University of Applied Sciences, Lothstr.~64, 80335 Munich, Germany}
\email{\{denitiu,schnoerr\}@cs.hm.edu}
\urladdr{http://www.cs.hm.edu/die\_fakultaet/ansprechpartner/professoren/schnoerr/}
\date{} 


\begin{document}

\begin{abstract}
We study unique recovery of cosparse signals from limited-angle tomographic measurements of two- and three-dimensional domains. Admissible signals belong to the union of subspaces defined by all cosupports of maximal cardinality $\ell$ with respect to the discrete gradient operator. We relate $\ell$ both to the number of measurements and to a nullspace condition with respect to the measurement matrix, so as to achieve unique recovery by linear programming. These results are supported by comprehensive numerical experiments that show a high correlation of performance in practice and theoretical predictions. Despite poor properties of the measurement matrix from the viewpoint of compressed sensing, the class of uniquely recoverable signals basically seems large enough to cover practical applications, like contactless quality inspection of compound solid bodies composed of few materials.
\end{abstract}

\maketitle

\section{Introduction} \label{sec:Introduction}
\subsection{Overview, Motivation} 
\label{sec:overview}
\emph{Discrete tomography} \cite{Herman1999} is concerned with the recovery of functions from few tomographic projections. Feasibility of this severely ill-posed problem rests upon assumptions that restrict the degrees of freedom of functions to be reconstructed. The canonical assumption is that functions only attain values from a finite set. Discrete tomography has shown potential for large-scale applications in various areas \cite{Petra-et-al-09a,Batenburg-TomoTV-12} which also stimulates theoretical research.

As advocated in \cite{Petra2009}, considering the problem of discrete tomography from the broader viewpoint of \emph{compressive sensing} \cite{CompressiveSampling-Notes-07,CompressedSensingIntro-08} enables to consider more general scenarios and to employ additional methods for investigating theoretical and practical aspects of discrete tomography. While the set of measurements (tomographic projections) is still ``discrete'' as opposed to the ``continuous'' theory of established tomographic settings \cite{Natterer2001}, functions to be reconstructed are required to be compressible: a \emph{sparse} representation exists such that few measurements of any function of some admissible class capture the degrees of freedom and enable recovery. Establishing corresponding sampling rates in connection with a given sparse representation and a model of the imaging sensor constitutes the main problem of mathematical research.

\begin{figure}
\begin{center}
\includegraphics[width=0.25\textwidth]{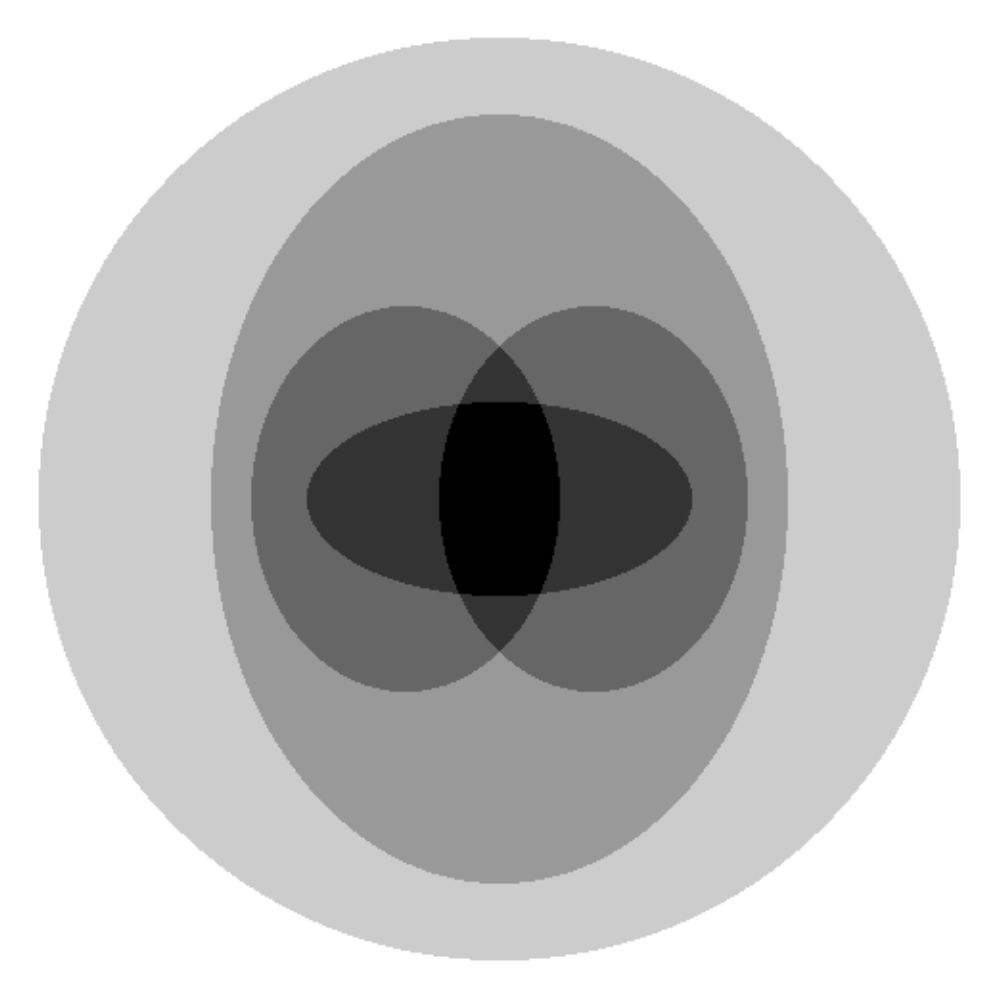}
\includegraphics[width=0.205\textwidth]{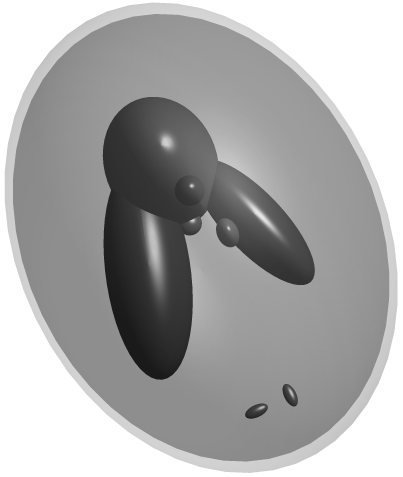}
\end{center}
\caption{The left figure sketches the class of compound solid bodies considered in this paper for reconstruction from few tomographic projections. These objects are similar to the 3D Shepp-Logan phantom (right) and are composed of different materials in a homogeneous way but with unknown geometry. The gradient of the piecewise constant intensity function is sparse. Recovery conditions depending on this property and the number of measurements are studied in this paper. For the present example, the cosparsity $\ell$ (defined by\eqref{eq:def-cosupport}) of the 3D Shepp-Logan phantom with $128^3$ voxels equals $\ell=3(d-1)d^2-109930=6132374$ with $d=128$. As we will show in connection with eqn.~\eqref{eq:m-lb-unknown}, the $\approx 2\cdot 10^{6}$ voxel values of the 3D Shepp-Logan phantom can be recovered \emph{exactly} from 
tomographic projections along $4\times (2d-1)d=130560$ parallel rays via the projecting matrix from four directions, see Section \ref{sec:setup_3D}, Fig.~\ref{fig:4C3D}.}
\label{fig:set-up}
\end{figure}

\begin{figure}
\begin{center}
\includegraphics[width=0.25\textwidth]{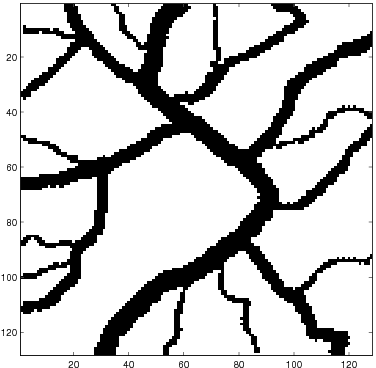}
\includegraphics[width=0.25\textwidth]{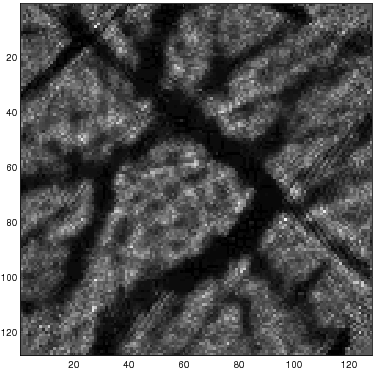}
\end{center}
\caption{Two experimental results are shown that demonstrate the effect of regularization -- total variation minimization recovery (left) versus $\ell_1$-minimization recovery (right). Our recovery analysis considered in the present paper applies also to the $128\times 128$ binary image on the left. This image has cosparsity $\ell=2(d-1)d-2251=30261$, while the sparsity of the gradient equals $2251$. As a consequence, the left image can be reconstructed \emph{exactly} via \eqref{eq:TVmin_pos} from 18 projections. 
Uniqueness of $\ell$-cosparse solution is provided for at least 2263 measurements according to \eqref{eq:m-lb-unknown}. On the other hand, taking into account that the image by itself (rather than its gradient) is $12648$-sparse, we can reconstruct it exactly by \eqref{eq:l1min}, but from 63 projections, following the analysis from \cite{Petra-Schnoerr-LAA,Petra2013b}. Thus about $3.5\times$ more measurements are needed than in the previous case.
Using the same number of 18 projections that suffice for exact reconstruction via \eqref{eq:TVmin_pos}, the reconstruction via \eqref{eq:l1min}, yields the poor result shown on the right.}
\label{fig:TV_vs_ell1}
\end{figure}

In this paper, we consider the problem of reconstructing compound solid bodies in dimensions $\bbd=2$ or $\bbd=3$, as illustrated by Figure \ref{fig:set-up}. These functions are represented by vectors $u \in \R^{n}$ in a high-dimensional Euclidean space, with components $u(v_{i})$ indexed by vertices $v_{i} \in V$ of a regular grid graph $G=(V,E)$ corresponding to pixels in 2D ($\bbd=2$) and to voxels in 3D ($\bbd=3$). The key assumption is that \emph{gradients} of functions to be reconstructed are sufficiently \emph{sparse}. 

As a consequence, if the linear system $A u = b$ represents the tomographic imaging set-up with given measurements $b \in \R^{m}$, then the standard $\ell_{1}$-minimization approach 
\begin{equation}\label{eq:l1min}
\min \|u\|_1 \qquad {\rm s.t.}  \quad Au=b, 
\end{equation}
does not apply, because $u$ itself is \emph{not} sparse. We consider instead, the \emph{total variation} criterion
\begin{equation}\label{eq:TVmin}
\min_{u} \TV(u) \qquad {\rm s.t.}  \quad Au=b, 
\end{equation}
and its \emph{nonnegative} counterpart
\begin{equation}\label{eq:TVmin_pos}
\min_{u} \TV(u) \qquad {\rm s.t.}  \quad Au=b, \quad u\ge 0, 
\end{equation}
that in the continuous case returns the $(\bbd-1)$-dimensional Hausdorff measure of discontinuities of indicator functions \cite{Ziemer-89}, with numerous applications in mathematical imaging \cite{Scherzer2011}. This provides the natural sparse representation of the class of functions considered in this paper (cf.~Fig.~\ref{fig:set-up} \& Fig.~\ref{fig:TV_vs_ell1}). 
Our objective in this paper is to establish sampling rates that enable the recovery of $u$ as solution to the optimization problem \eqref{eq:TVmin}
or \eqref{eq:TVmin_pos}.

For industrial applications additionally motivating our work, we refer to e.g.~\cite{Carmignato2012,Grunzweig2013}. In this context, scenarios of \emph{limited-angle tomography} are relevant to our work as they enable minimization of acquisition time and related errors, affecting the quality of projection measurements and in turn object reconstruction.

\subsection{Related Work and Contribution}\label{sec:related_work}
Theoretical recovery guarantees, expressed as thresholds on the critical parameters of problem \eqref{eq:l1min},  
relate the solution sparsity to the solution degrees of freedom and to the number of measurements. Recent work illustrates that the focus of corresponding research in \emph{compressed sensing (CS)} is shifting - in contrast to discrete tomography \cite{Herman1999} - from a \emph{worst-case analysis}
\cite{Gardner-Gritzmann-97, Petra2009} towards an \emph{average-case analysis}
\cite{LimS13, LimS13a, JafarpourDC12}. 
As for many other difficult combinatorial problems, the probabilistic approach is a plausible, and often the only possible way, 
to make well-founded statements that go beyond idealistic mathematical assumptions and that are also relevant for real-world applications.

In discrete tomography, images to be reconstructed are sampled along lines. Thus, sampling patterns are quite different from
random and non-adaptive measurements that are favourable from the viewpoint of compressed sensing.
In \cite{Petra2009}, we showed that structured sampling patterns as used in commercial \emph{computed tomography (CT)} scanners do not satisfy the CS conditions, like the nullspace property and
the \emph{restricted isometry property (RIP)}, that guarantee accurate recovery of
sparse (or compressible) signals. In fact, these recovery conditions predict a
quite poor worst-case performance of tomographic measurements, due to the high nullspace sparsity
of a tomographic projection matrix $A$. Moreover, the gap
between available \emph{worst-case} recovery results of CS \cite{FacesRandomPolytopes-09} and 
\emph{worst-case}
results from tomographic projections in \cite{Petra2009} is dramatic.

In  \cite{Petra2013b,Petra-Schnoerr-LAA}, we presented an average-case relation between image sparsity
and sufficient number of measurements for recovery, and we showed
that the transition from non-recovery to recovery is sharp for specific sparse
images. The analysis is based on the non-negativity of the coefficient matrix
and of the signal itself and utilizes new mathematical tools from CS via expander graphs.

However, due to the unrestricted sign patterns of the sparse vector $\nabla u$
and of the corresponding coefficient matrix, compare Section \ref{sec:minTV}, we cannot
transfer the recovery results established in \cite{Petra2013b} to the problem \eqref{eq:TVmin} and \eqref{eq:TVmin_pos}.

We overcome this difficulty by adopting the recently introduced \emph{cosparse analysis model} from \cite{Nam2013}, that provides an alternative viewpoint to the classical \emph{synthesis model}
and is more suitable to the problem class considered in this paper. Our present work applies and extends the results
from \cite{Nam2013} to the 3D recovery problem from few tomographic projections of three-dimensional images consisting of few homogeneous regions.
We give a  theoretical relation between the image \emph{cosparsity} 
and sufficient sampling, validate it empirically and conclude that TV-reconstructions of a
class of synthetic phantoms exhibit a well-defined recovery curve similar to
the  study in \cite{Petra-Schnoerr-LAA,Petra2013b}.

Empirical evidence for the recovery of piecewise constant functions
from few tomographic measurements was already  observed in \cite{Sidky_TV_POCS,Herman_InvProbl,Jorgensen_IEEE}.
The first theoretical guarantees that have been obtained for recovery  from  noiseless samples of images with exactly sparse gradients via total variation minimization, date back to the beginnings of CS \cite{Candes_2006_TV,Candes_UP_2006}. 
However, the measurements considered were incomplete Fourier samples, and images were not sampled
along lines in the spatial domain, but along few radial lines in the frequency domain.
Such measurements ensembles are known to have good CS properties as opposed to the
CT setup, and are almost isometric on sparse signals for a sufficient number of samples.
As a result, recovery is stable in such scenarios. Stable recovery of the image gradient
from incomplete Fourier samples was shown in \cite{PatelMGC12}, while Needell \cite{NeedellStableTV}
showed that stable image reconstruction via total variation minimization is possible
also beyond the Fourier setup, provided the measurement ensemble satisfies the RIP condition.

\subsection{Organization} 
Section \ref{sec:TomoProp} collects basic definitions from compressed sensing and characterizes accordingly the imaging scenarios considered in this paper. 
We work out in more detail in Section \ref{sec:SparseTV} that the required assumptions in \cite{NeedellStableTV} do not imply relevant recovery guarantees for the discrete tomography set-ups considered here.
In Section \ref{sec:CoSparse}, we adopt the cosparse analysis model \cite{Nam2013} and generalize corresponding results to the practically relevant three-dimensional case.
Aspects of the linear programming formulation used to solve problem \eqref{eq:TVmin_pos}, are examined in Section \ref{sec:minTV}. A comprehensive numerical study underpinning our results is reported in Section \ref{sec:Experiments}.
We conclude in Section \ref{sec:Conclusions}.

\subsection{Basic Notation}\label{sec:Notation}
For $n \in \N$, we use the shorthands $[n] = \{1,2,\dotsc,n\}$ and $[n]_{0} = \{0,1,\dotsc,n-1\}$. For a subset $\Gamma \subset [n]$, the complement is denoted by $\Gamma^{c} = [n] \setminus \Gamma$. For some matrix $A$ and a vector $z$, $A_{\Gamma}$ denotes the submatrix of rows indexed by $\Gamma$, and $z_{\Gamma}$ the corresponding subvector. Thus, $A_{\Gamma} z_{\Gamma} = (A z)_{\Gamma}$. $\mc{N}(A)$ denotes the nullspace of $A$. Vectors are columns vectors and indexed by superscripts. $z^{\T}$ denotes the transposed vector $z$ and $\la z^{1}, z^{2} \ra$ the Euclidean inner product. To save space, however, we will sometimes simply write e.g.~$z = (z^{1},z^{2})$ instead of correctly denoting $z = \big( (z^{1})^{\T}, (z^{2})^{\T}\big)^{\T}$, for $z = \bsm z^{1} \\ z^{2} \esm$.
$\eins = (1,1,\dotsc,1)^{\T}$ denotes the one-vector whose dimension will always be clear from the context.
The dimension of a vector $z$ we denote by $\dim(z)$.

We consider signals $u(x),\,x \in \Omega \subset \R^{\bbd},\, \bbd \in \{2,3\}$ discretized as follows. $\Omega$ is assumed to be a rectangular cuboid covered by a \emph{regular grid graph} $G = (V,E)$ of size $|V| = n$. Accordingly, we identify $V = \prod_{i \in [d]} [n_{i}]_{0} \subset \Z^{\bbd}$, $n_{i} \in \N$. Thus, vertices $v \in V$ are indexed by $(i,j)^{\T} \in \Z^{2}$ and $(i,j,k)^{\T} \in \Z^{3}$ in the case $\bbd=2$ and $\bbd=3$, respectively, with ranges $i \in [n_{1}]_{0}, j \in [n_{2}]_{0}, k \in [n_{3}]_{0}$, and 
\begin{equation}\label{eq:def-n}
n = n_{1} n_{2} n_{3}.
\end{equation}
As a result, discretization of $u(x),\,x \in \Omega$, yields the vector $u \in \R^{n}$, where we keep the symbol $u$ for simplicity.

Two vertices $v_{1}, v_{2} \in V$ are adjacent, i.e.~form an edge $e = (v_{1},v_{2}) \in E$, if $\|v_{1}-v_{2}\|_{1}=1$. We also denote this by $v_{1} \sim v_{2}$.
\begin{remark}
Informally speaking, $G$ corresponds to the \emph{regular pixel or voxel grid} in 2D and 3D, respectively, and should not be confused with the general notion of a \emph{regular graph}, defined by equal valency $\big|\{v' \in V \colon v' \sim v\}\big|$ for every $v \in V$. In this sense, the regular grid graphs $G$ considered here are \emph{not} regular graphs.
\end{remark}

Consider the one-dimensional discrete derivative operator 
\begin{equation} \label{eq:def-partial}
\partial \colon \R^{m} \to \R^{m-1},
\qquad
\partial_{i,j} = \begin{cases}
-1, & i=j, \\
+1, & j=i+1, \\
0, & \text{otherwise}.
\end{cases}
\end{equation}
Forming corresponding operators $\partial_{1}, \partial_{2}, \partial_{3}$ for each coordinate, conforming to the ranges of $i,j,k$ such that $(i,j,k) \in V$, we obtain the discrete gradient operator 
\begin{equation}\label{eq:def-nabla}
\nabla = \bpm 
\partial_{1} \otimes I_{2} \otimes I_{3} \\
I_{1} \otimes \partial_{2} \otimes I_{3} \\
I_{1} \otimes I_{2} \otimes \partial_{3}
\epm 
\in \R^{p \times n},
\end{equation} 
where $\otimes$ denotes the Kronecker product and $I_{i},\,i=1,2,3$, are identity matrices with appropriate dimensions. The \emph{anisotropic} discretized TV-measure is given by
\begin{equation}\label{eq:def-TV}
\TV(u) := \|\nabla u \|_{1}.
\end{equation}

\section{Properties of Tomographic Sensing Matrices}\label{sec:TomoProp}

Depending on the application, different scanning geometries are used in CT imaging.
In the present study, we adopt a simple discretized model based on an image
$u(x),\,x \in \Omega \subset \R^{\bbd},\, \bbd \in \{2,3\}$, that represents the inhomogeneity of $\Omega$ and 
consists of an array of unknown densities $u_j, j\in[n]$ as defined in Section \ref{sec:Notation}. The model comprises algebraic
equations for these unknowns in terms of measured projection data.
To set up these equations, the sensing device measures line integrals of the object attenuation
coefficient along X-rays $L_i, i\in[m]$, along some known orientations. 
The $i$-th corresponding measurement obeys
\begin{equation}\label{eq:discretization}
b_{i}:\approx\int_{L_i} u(x)dx\approx  \sum_{j=1}^n
u_j\int_{L_i}\mathcal{B}_j(x)dx = \sum_{j=1}^n
u_j A_{ij}.
\end{equation}
The values $A_{ij}$ form the \emph{measurement or projecton matrix} $A$ depend on the choice of the basis function.
We assume $\mathcal{B}_j$ are cube- or square-shaped uniform
basis functions, the classical \emph{voxel} in 3D or \emph{pixel} in 2D.

The main task studied in this paper concerns  estimation of the weights $u_j$ from the recorded measurements $b_{i}$ and solving
the noiseless setting $Au = b$. The matrix $A$ has dimensions $(\# \;\text{rays}=:m)
\times (\# \;\text{voxel/pixel}=:n)$, where $m\ll n$! Since the projection matrix encodes the incident 
relation between rays and voxels/pixels, the projection matrix $A$ will be sparse. Based on additional assumptions on $u$, we will devise in this paper conditions for exact recovery of $u$ from the \emph{underdetermined} linear system $A u = b$.

\subsection{Imaging Set-Up}
For simplicity, we will assume that $\Omega$ is a cube in 3D or a square in 2D and that $\Omega = [0,d]^{3}$
is discretized into $d^3$ voxels, while $\Omega = [0,d]^{2}$ is discretized into $d^2$ pixels.
We consider a parallel ray geometry and choose the projection angles such that the intersection of each 
line with all adjacent cells is constant, thus yielding binary projection matrices after scaling. This 
simplification is merely made in order to obtain a structure in the projection matrix which allows to compute 
relevant combinatorial measures. We stress however that other discretization 
choices are possible and lead to similar results.

\subsubsection{2D Case: $3, \dotsc, 8$ Projection Directions}\label{sec:setup_2D}

We set  $\Omega = [0,d]^{2}$ and obtain the binary projection matrices according to \eqref{eq:discretization} from few projecting directions (three to eight), compare Fig.~\ref{fig:2D_projections}. We summarize the used parameters in Table \ref{tab:1}.

\begin{figure}
\centerline{
\includegraphics[width=0.95\textwidth]{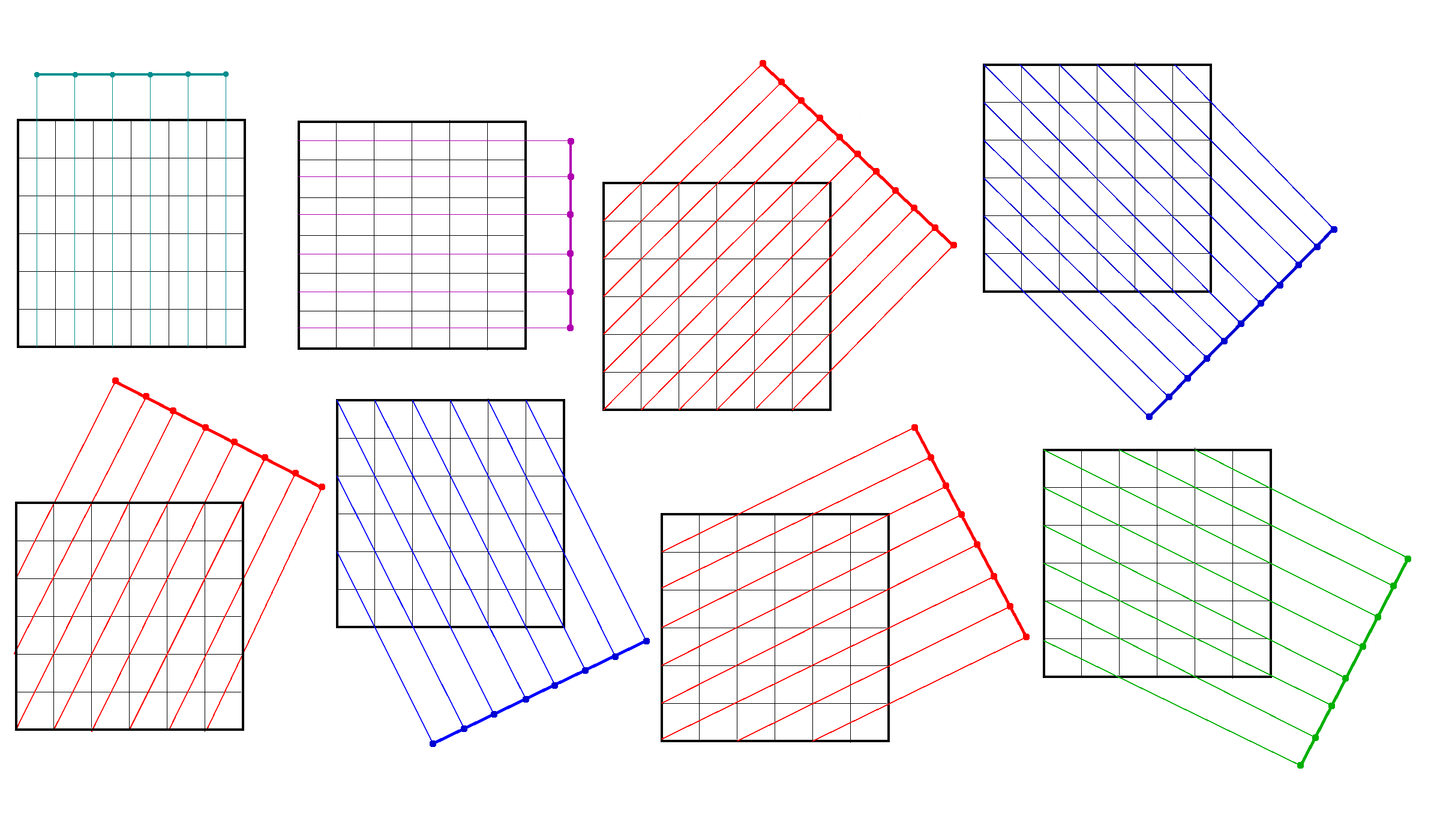}
}
\caption{Eight different projecting directions along with projecting rays for $90^\circ$, $0^\circ$, $\mp 45^\circ$, $\mp \arctan(2)$, 
and $\mp \arctan(0.5)$ (from left to right, top to bottom). Note that the intersection segments for each 
projection ray with all adjacent pixel are equal. As a consequence, we obtain after appropriate scaling, binary projection matrices.
Each sensor resolution varies with the projection angle, however. The illustration above depicts $\Omega = [0,d]^{2}$ with $d=6$.
}
\label{fig:2D_projections}
\end{figure}

\begin{table}%
\begin{tabular}{|c|c|c|r|}
	\hline
\# proj. dir. &  $m$    &   $n$  &  projection angles  \\
\hline
3   & $4d-1$ & $d^2$ & $0^\circ, 90^\circ, 45^\circ$\\
4   & $6d-2$ & $d^2$ & $0^\circ, 90^\circ, \mp 45^\circ$\\
5  & $7d + \lfloor \frac{d}{2}\rfloor-2$ & $d^2$ & $0^\circ, 90^\circ, \mp 45^\circ, \arctan(2)$\\
6  & $8d + 2\lfloor \frac{d}{2}\rfloor-2$ & $d^2$ & $0^\circ, 90^\circ, \mp 45^\circ, \mp \arctan(2)$\\
7  & $9d + 3\lfloor \frac{d}{2}\rfloor-2$& $d^2$ & $0^\circ, 90^\circ, \mp 45^\circ, \mp \arctan(2), \arctan(0.5)$\\
8  & $10d + 4\lfloor \frac{d}{2}\rfloor-2$& $d^2$ & $0^\circ, 90^\circ, \mp 45^\circ, \mp \arctan(2), \mp \arctan(0.5)$\\
\hline
\end{tabular}
\vspace{2mm}
\caption{Dimensions of projection matrices in 2D.}
\label{tab:1}
\end{table} 

\subsubsection{3D Case: $3$ or $4$ Projection Directions}\label{sec:setup_3D}

We consider the imaging set-up depicted by Fig.~\ref{fig:3C3D} and Fig.~\ref{fig:4C3D}.
The projection angles were chosen again such that the intersection of each ray with all adjacent voxels is constant. After
appropriate scaling the resulting measurement matrices are binary as well.

\begin{figure}
\centerline{
\includegraphics[width=0.26\textwidth]{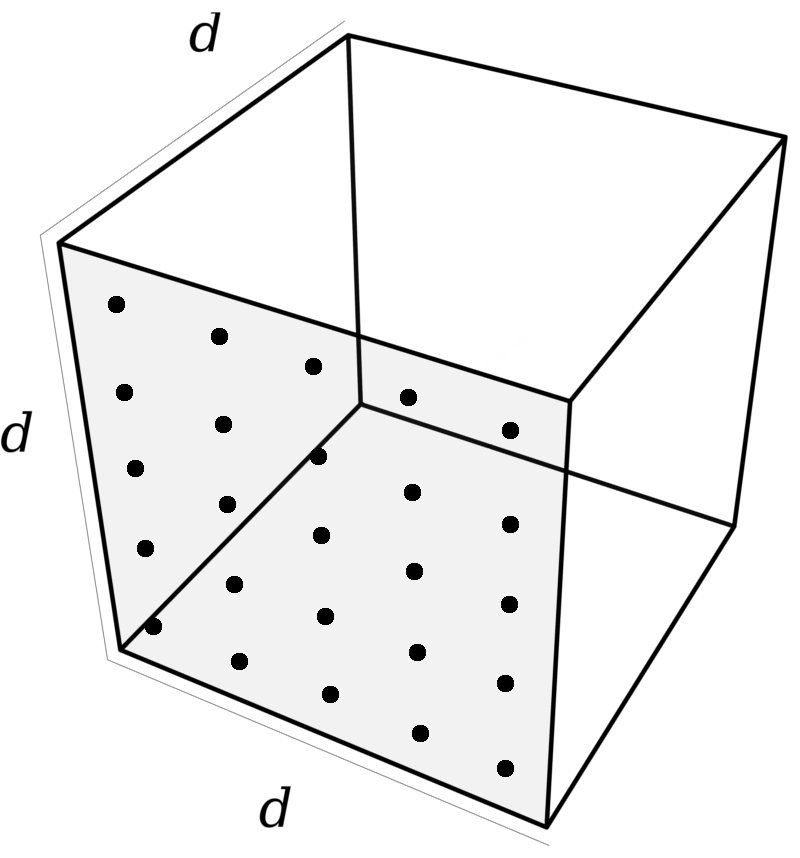}
\hspace{0.025\textwidth}
\includegraphics[width=0.26\textwidth]{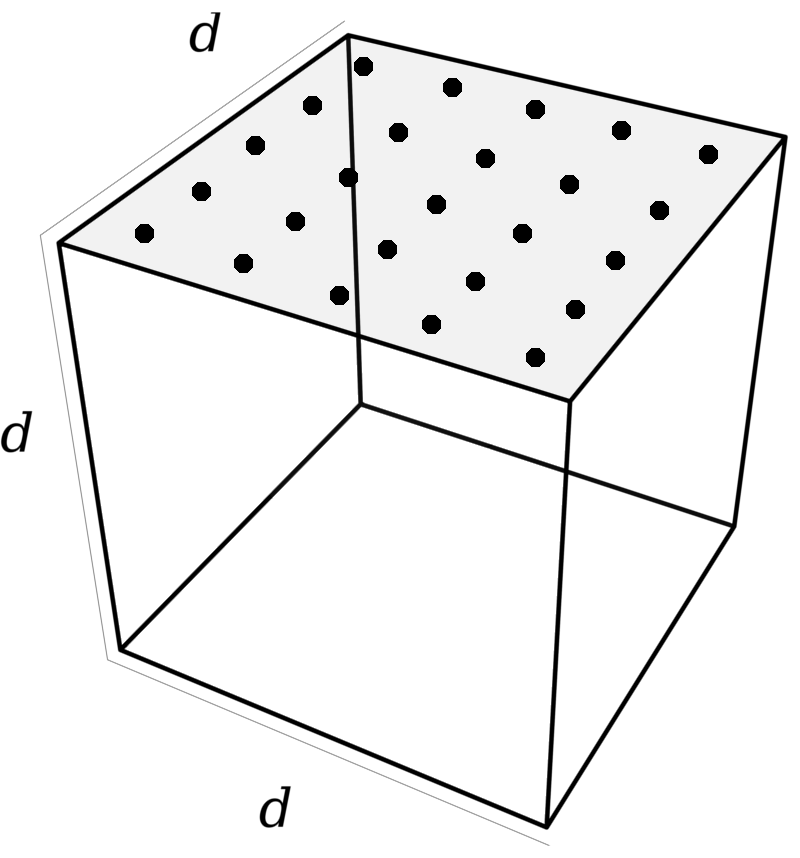}
\hspace{0.025\textwidth}
\includegraphics[width=0.26\textwidth]{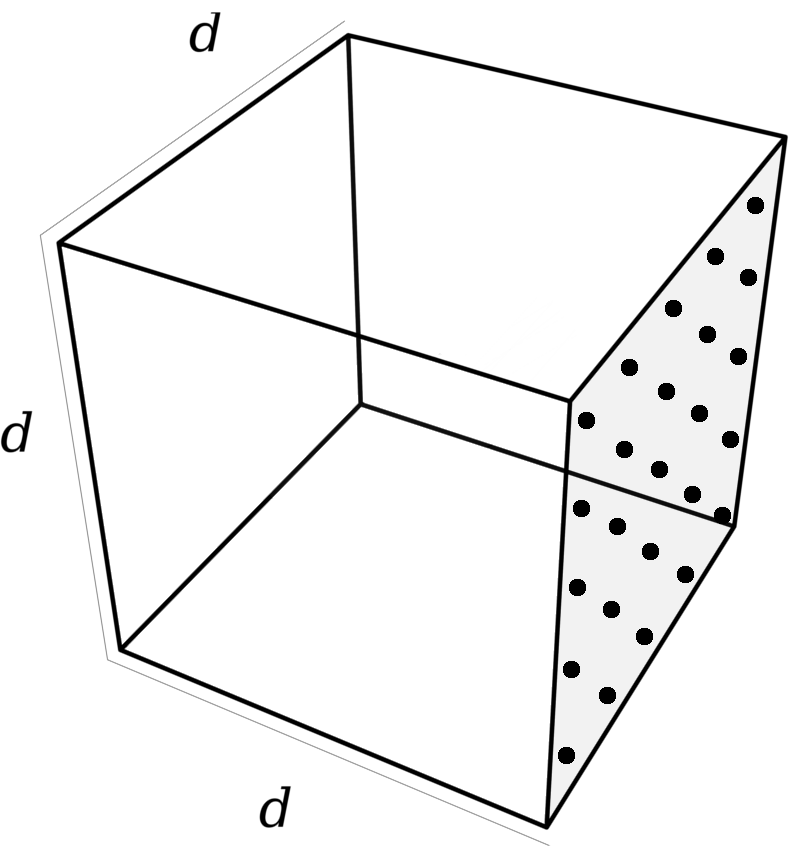}
}
\caption{
Imaging set-up for three orthogonal projections corresponding to each shaded plane of the cube. 
From left to right: Cell centers projected along each direction are shown as dots for the case $d=5$. The cube $\Omega = [0,d]^{3}$ is discretized into $d^{3}$ cells and projected along $3 \cdot d^2$ rays.
}
\label{fig:3C3D}
\end{figure}

\begin{figure}
\centerline{
\includegraphics[width=0.3\textwidth]{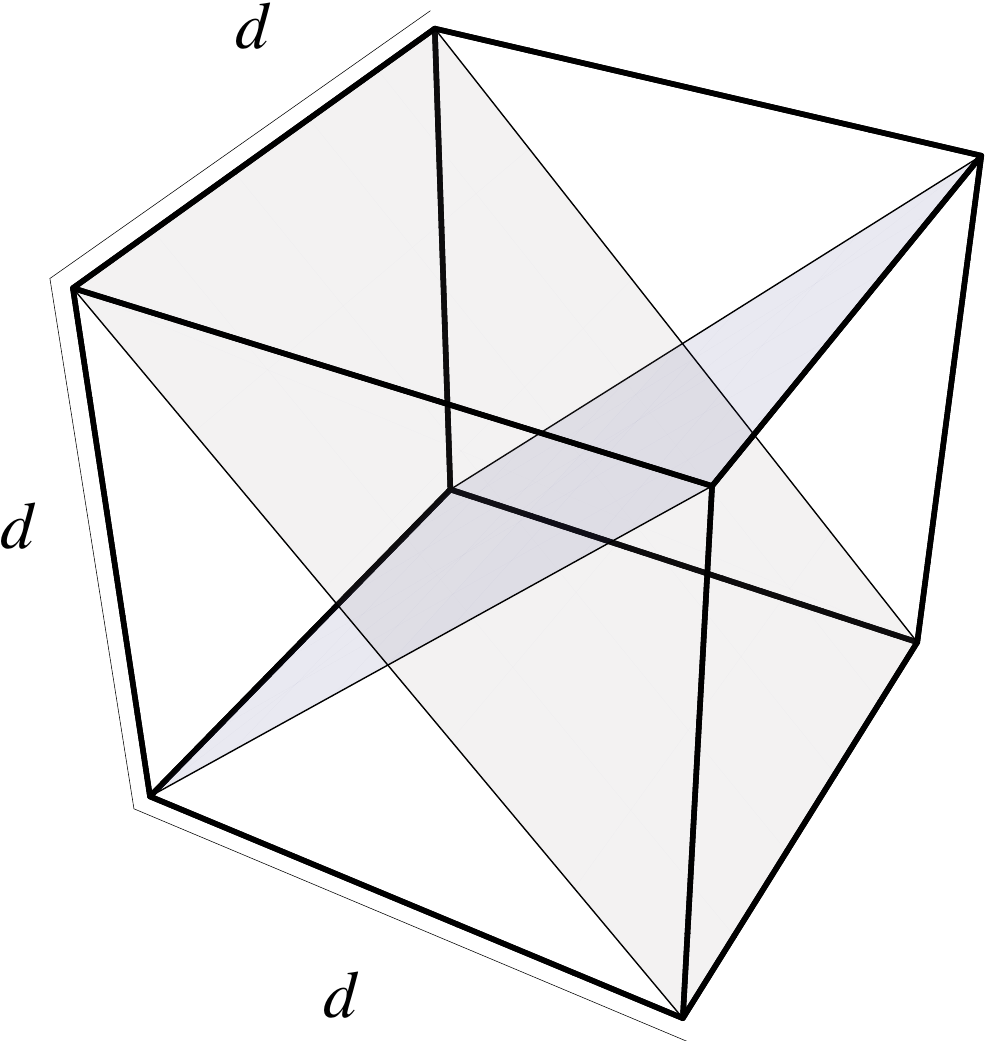}
\hspace{0.025\textwidth}
\includegraphics[width=0.3\textwidth]{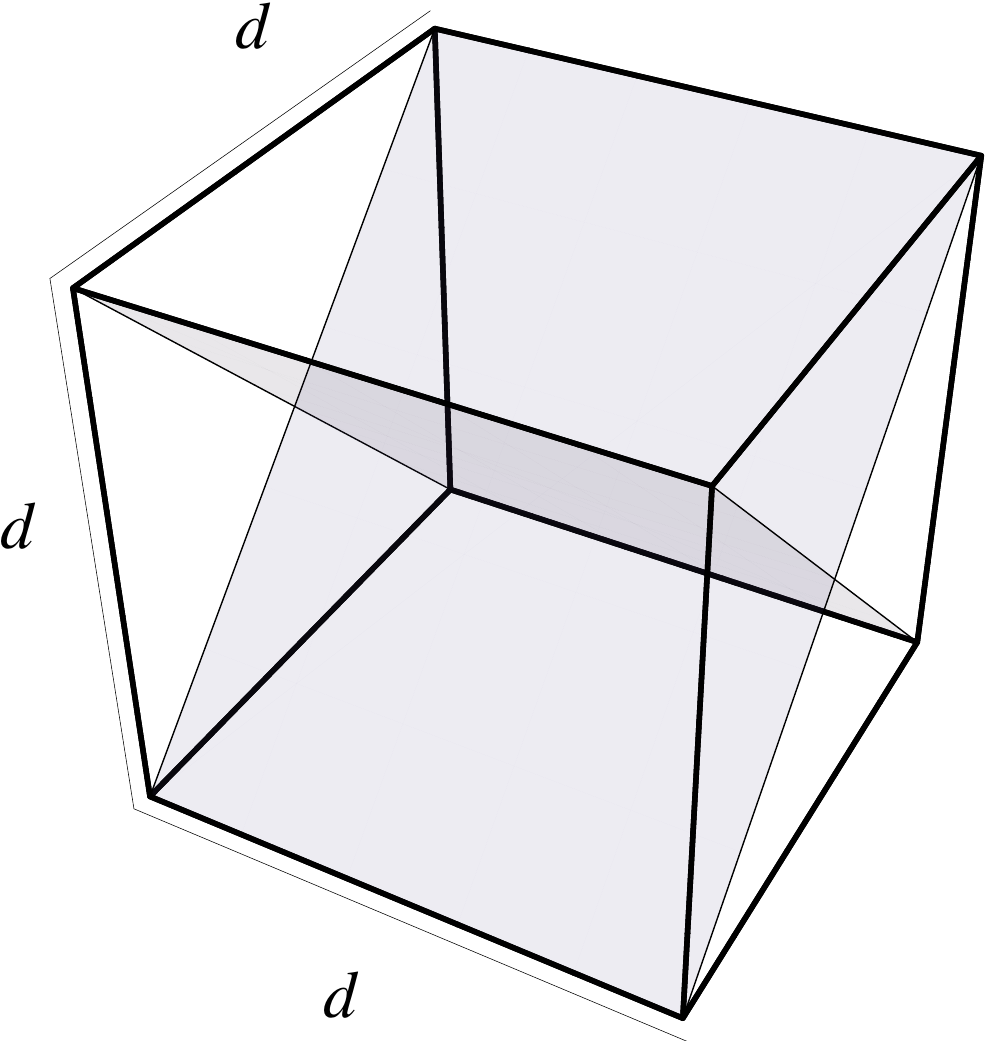}
\hspace{0.025\textwidth}
\includegraphics[width=0.25\textwidth]{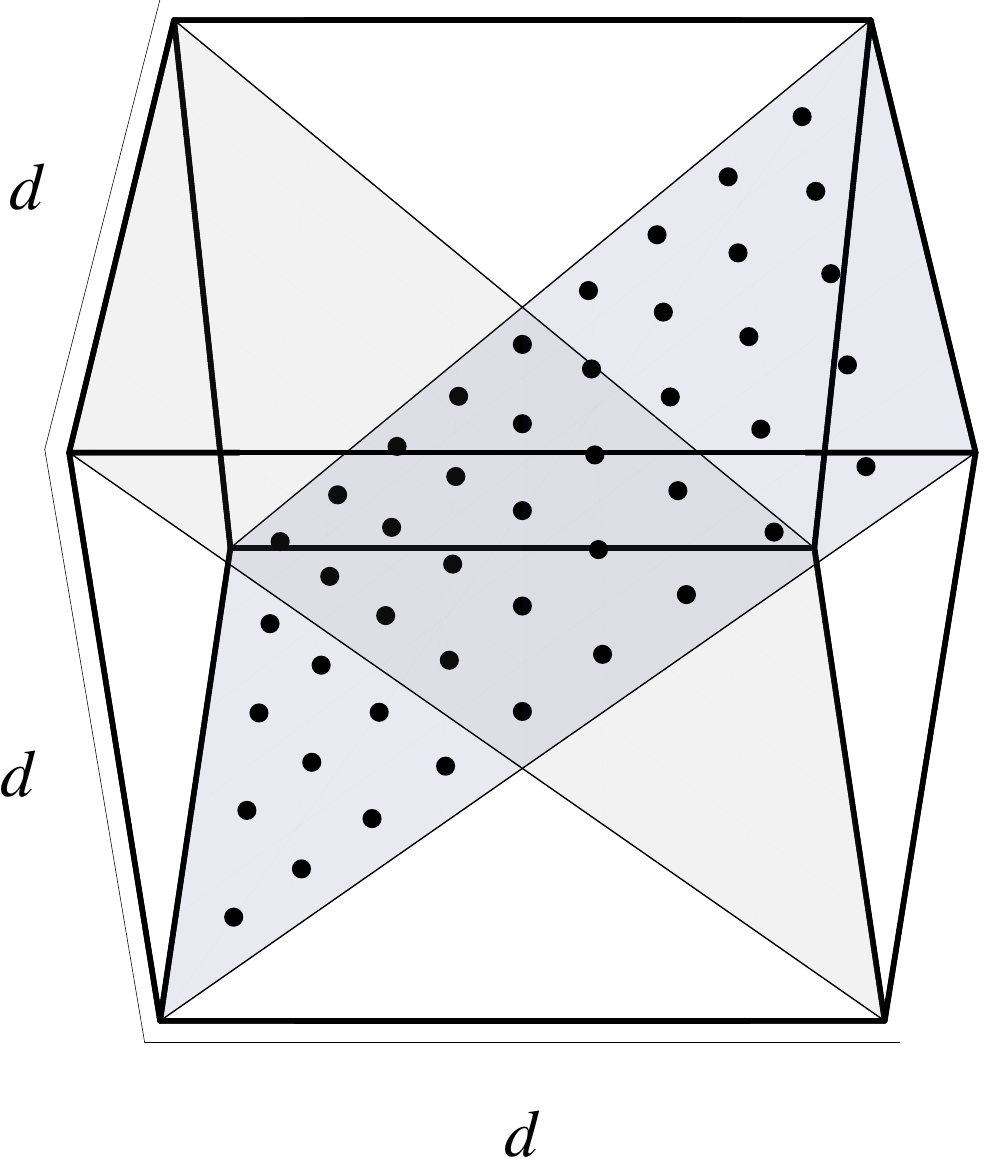}
}
\caption{
Imaging set-up for four projecting directions corresponding to the image planes shown as two pairs in the left and center panel respectively. Right panel: Voxel centers projected onto the first image plane are shown as dots for the case $d=5$. The cube $\Omega = [0,d]^{3}$ is discretized into $d^{3}$ voxel and projected along $4 \cdot d (2 d-1)$ rays.
}
\label{fig:4C3D}
\end{figure}

\subsection{Complete Rank and RIP}
\label{sec:RIP}

For recovery of $k$-sparse signals by \emph{compressed censing (CS)} both necessary
and sufficient conditions have been provided, which not only depend on the sparsity $k$ of the
original signal, but also on the conditions of the sensing matrix $A$. In particular,
compressed sensing aims for a matrix which has high spark, also known as complete rank, high nullspace property order, and a small RIP
constant, as detailed next.

\begin{definition}[\cite{Elad06}]\label{def:spark} 
Let $A\in\R^{m\times n}$ be an arbitrary matrix. Then the \emph{spark} of $A$ denoted by $\spark(A)$ is the
minimal number of linearly dependent columns of $A$.
\end{definition}

Any $k$-sparse solution $\ol{u}$ of a linear system $Au=b$ is unique if $\|\ol{u}\|_0=k<\spark(A)/2$. 

Due to the fact that $A$ is underdetermined, the nullspace of $A$ also plays a
particular role in the analysis of uniqueness of the minimization problem \eqref{eq:l1min}.  
The related so-called \emph{nullspace property (NSP)} is defined as follows.

\begin{definition}\label{def:NSP} 
Let $A\in\R^{m\times n}$ be an arbitrary matrix. Then $A$ has the \emph{nullspace property (NSP) of
order $k$} if, for all $v \in \mc{N}(A) \setminus \{0\}$ and for all index sets $|S| \le k$,
$\|v_S\|_1 < \frac{1}{2}\|v\|_1$.
\end{definition}

Any $k$-sparse solution $\ol{u}$ of a linear system $Au=b$ is the unique solution of
\eqref{eq:l1min}, if $A$ satisfies the nullspace property of order $k$. For nonnegative signals,
the NSP can be characterized in terms of the minimal number of negative components in the
sparsest nullspace vector.

\begin{proposition}\label{char:NP} 
Every $k$-sparse nonnegative vector $\ol{u}$ is the unique positive solution of $Au = A\ol{u}$ 
iff every nonzero nullspace vector has at least $k + 1$ negative (and positive) entries.
\end{proposition}

\begin{table}%
\begin{tabular}{|c|c|c|c|c|c|c|}
	\hline
dim. &\# proj. dir. &  $m$    &   $n$  &  $\rank(A)$ & $\spark(A)$& NSP\\
\hline
\multirow {5}{*}{$\bbd=2$} & 3   & $4d-1$ & $d^2$ &$4d-4$ &6 &2\\
&4   & $6d-2$ & $d^2$ & $6d-9$& 8 & 3\\
&5  & $7d + \lfloor \frac{d}{2}\rfloor-2$ & $d^2$ & $7d + \lfloor \frac{d}{2}\rfloor-13$&12&5\\
&6  & $8d + 2\lfloor \frac{d}{2}\rfloor-2$ & $d^2$ & $8d + 2\lfloor \frac{d}{2}\rfloor-19$&16&7 \\
&7  & $9d + 3\lfloor \frac{d}{2}\rfloor-2$& $d^2$ & $9d + 3\lfloor \frac{d}{2}\rfloor-23$&16&7\\
&8  & $10d + 4\lfloor \frac{d}{2}\rfloor-2$& $d^2$ & $10d + 4\lfloor \frac{d}{2}\rfloor-29$&16&7\\
\hline
\multirow {2}{*}{$\bbd=3$} & 3   & $3d^2$ & $d^3$ &$3d^2 - 3d + 1$& 8&3\\
& 4   & $8d^2-4d$ & $d^3$ &$8d^2 - 20d + 16$ &15&6\\
\hline
\end{tabular}
\vspace{2mm}
\caption{Properties of projection matrices in 2D and 3D.}
\label{tab:spark}
\end{table} 


The \emph{restricted isometry property (RIP)}, defined next, characterizes matrices which are well conditioned when operating on sparse vectors. 
This is probably the most popular CS condition since it also enables \emph{stable} recovery.
 
\begin{definition}\label{def:RIP-2} 
A matrix $A$ is said to have the \emph{Restricted Isometry Property} 
$RIP_{\delta,k}$ if, for any $k$-sparse vector $u$, the relation 
\begin{equation}\label{eq:RIP-2}
(1-\delta) \|u\|^2 \leq \|A u\|^2 \leq (1+\delta) \|u\|^2 \;,\quad
\delta \in (0,1) 
\end{equation}
holds.
\end{definition}

This property implies that every submatrix $(A^{i_{1}},\dotsc,A^{i_{k}})$ formed by keeping
at most $k$-columns of $A$ has nonzero singular values
bounded from above by $1+\delta$ and from below by $1-\delta$.
In particular, \eqref{eq:RIP-2} implies
that a matrix $A$ cannot satisfy $RIP_{\delta,k}$ if
$k\ge \spark(A)$.

Cand{\`e}s has shown \cite[Thm. 1.1]{Can08} that if 
$A\in RIP_{\delta,2k}$ with $\delta<\sqrt{2}-1$, then all $k$-sparse solutions $\ol{u}$ of \eqref{eq:l1min} are unique.
Moreover,  when measurements are corrupted with noise
\[
b = A\ol{u} + \nu ,
\]
where $\nu$ is an unknown noise term, recovery of $\ol{u}$ is usually performed by
\begin{equation}\label{eq:l1min_noise}
\min \| u \|_1 \qquad {\rm s.t.} \quad \|Au - b\|_2 \le \veps
\end{equation}
where $\veps$ is an upper bound on the size of $\nu$.
As shown in \cite[Thm. 1.2]{Can08}, the RIP condition also implies \emph{stable} recovery even in case
of observation errors.
Provided that the observation error is small enough, $\|\nu\|_2\le \veps$, and
$A\in RIP_{\delta,2k}$ with $\delta<\sqrt{2}-1$, then
the solution $u$ of \eqref{eq:l1min_noise} obeys
\begin{equation*}
 \| u -\ol{u}\|_2 \le C_0k^{-\frac{1}{2}}\| \ol{u} - (\ol{u})_k\|_2 + C_1\veps,
\end{equation*}
where $C_0$ and $C_1$ are explicit constants depending on $\delta$, and $(\ol{u})_k$
is  the vector $\ol{u}$ with all but the $k$-largest entries set to zero.

It has been shown in \cite{Binary-NegRIP-08} that binary matrices cannot satisfy
$RIP_{\delta,k}$ unless the numbers of rows is $\Omega(k^2)$. 

\begin{theorem}\cite[Thm. 1]{Binary-NegRIP-08} Let $A\in\Rmn$ be any $0/1$-matrix
that satisfies $RIP_{\delta,k}$. Then
$$
m\ge\min\left\{\left(\frac{1-\delta}{1+\delta}\right)^2k^2,\frac{1-\delta}{1+\delta}\ n\right\}\ .
$$
\end{theorem}

Taking into account that there exists $\spark(A)$-columns in $A$ which are linearly dependent, we obtain, together with
$m\ll n$, the following result.

\begin{corollary} Let $\delta \in (0,1)$. 
A necessary condition for $A$ to satisfy
the $RIP_{\delta,k}$
for all $k$-sparse vectors is that
$$
k\le \min\left\{\frac{1+\delta}{1-\delta}m^{\frac{1}{2}} , \spark(A)-1\right\}\ .
$$
\end{corollary}

\textbf{Application to our scenarios.} 
For our particular matrices $A$ defined in Sections \ref{sec:setup_2D} and \ref{sec:setup_3D}
we obtain, along the lines of \cite[Prop.~3.2]{Petra2009}, that $\spark(A)$ is a constant for all dimensions $d$ with $m<n$, while
the number of measurements obeys $O(d^{\bbd-1})$, $\bbd\in\{2,3\}$, see Table \ref{tab:spark}. Compare also Fig.~\ref{fig:nullspace_vec}, left.
However, we cannot be sure that $A$ possesses the $RIP_{\sqrt{2}-1,\sigma}$, with $\sigma=\spark(A)-1$,
unless we compute the singular values of all submatrices containing $\sigma$ or less columns of $A$.

\begin{figure}
\centerline{
\includegraphics[height=0.2\textheight]{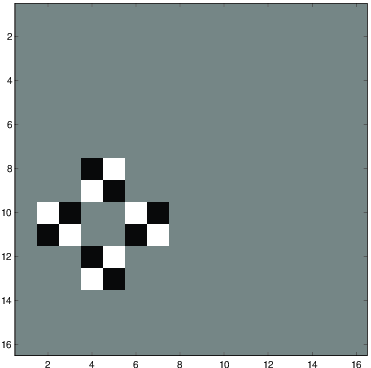}
\includegraphics[height=0.2\textheight]{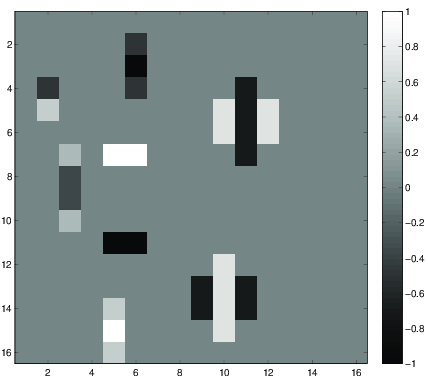}
}
\caption{The \emph{sparsest} nullspace vector $u_\mc{N}\in\mc{N}(A)\setminus\{0\}$ is shown on the left as a $16\times 16$ image, for matrices $A$ from Section \ref{sec:setup_2D} with 6,7 or 8 projecting directions, where $d=16$, compare Table~\ref{tab:1}.
Gray indicates components with value 0, white the value 1 and black the value $-1$. Projections along all rays depicted in
Fig.~\ref{fig:2D_projections} sum up to zero. This shows that $\spark(A)=16$ and the matrix has a NSP of order 7.
These numbers do not change with the problem's size for any $d\ge 16$. The image on the right depicts the bivariate Haar-transformed
nullspace basis vector $Hu_\mc{N}$, which has 32 nonzero elements.
}
\label{fig:nullspace_vec}
\end{figure}

The previous results show that the poor properties  of $A$, from the viewpoint of CS, rest upon the small spark of $A$. 
In order to increase the maximal number  of columns such that all column collections of size $k$ (or less)  
are linearly independent, we can add to the entries of $A$ small random numbers. Due to the fact that 
$\rank(A)$ almost equals $m$ in all considered situations, the probability that $k$-arbitrary columns 
are linearly independent slowly decreases from 1, when $k < \spark(A)$, to 0, when $k> \rank(A)$. 
The perturbed matrix $\tilde A$ is computed by uniformly perturbing the non-zero entries $A_{ij} > 0$
to obtain $\tilde A_{ij} \in [A_{ij} - \epsilon,A_{ij} + \epsilon]$, and by normalizing subsequently all column vectors 
of $\tilde A$.
In practice, such perturbations can be implemented by discretizing the image by different basis
functions or choose their locations on an irregular grid.

As argued in Section \ref{sec:overview} and illustrated by Figures \ref{fig:set-up} and \ref{fig:TV_vs_ell1}, considering functions $u(x)$ with sparse gradients and the corresponding optimization criterion $\TV(u)$ for recovery \eqref{eq:TVmin}, we may boost recovery performance in severely 
undersampled tomographic scenarios, despite the poor properties of measurement matrices $A$.

\section{Sparsity and TV-Based Reconstruction}\label{sec:SparseTV}
As discussed in Section \ref{sec:related_work}, it has been well known empirically that solving the problem
\begin{equation}\label{eq:TVnoise}
 \min \|\nabla u\|_1 \qquad {\rm s.t.} \quad \|Au - b\|_2 \le \varepsilon
\end{equation}
can provide high-quality and stable image recovery.
Until the recent work \cite{NeedellStableTV}, however, it had been an open problem to provide provable \emph{theoretical}
guarantees, beyond incomplete Fourier measurements \cite{Candes_2006_TV,Candes_UP_2006}.  
The gradient operator $\nabla$ is not an orthonormal basis or a tight frame, 
thus neither the standard theory of CS nor the theoretical extensions in \cite{Needell_Redundant}
concerning the \emph{analysis model} apply to \eqref{eq:TVnoise},
even for images with truly sparse gradient $\nabla u$.

The recent work \cite{NeedellStableTV,Needell_multidim} \emph{proves} that stable recovery is possible via
the convex program \eqref{eq:TVnoise} and
considers a general matrix $A$ which is
incoherent with the multidimensional Haar wavelet transform and satisfies a RIP condition.
The Haar wavelet transform provides a sparsifying basis for 2D and 3D images and is closely 
related to the discrete gradient operator. In the remainder of this section, we denote the  discrete multidimensional Haar wavelet transform
by $H$ and refer the reader to the definition of \cite[p.~6]{Needell_multidim}.
The following theorem summarizes the main results of \cite[Thm.~5, Thm.~6]{NeedellStableTV}, and \cite{Needell_multidim}
and specializes them to the case of \emph{anisotropic} TV \eqref{eq:def-TV} as considered in the present paper, see also the Remarks following \cite[Thm.~6]{NeedellStableTV}
and \cite[Main Thm.]{Needell_multidim}.
 
\begin{theorem}\label{thm:TVstable}
Let $d=2^N$ be a power of two and $n=d^\bbd$, $\bbd\in\{2,3\}$. Further, let $H$ be the discrete 
multidimensional Haar transform, and let $A\in\R^{m\times n}$
such that $AH^{-1}$ satisfies $RIP_{\delta,5k}$ with $\delta<\frac{1}{3}$.
Then for any $\ol{u}\in\R^{n}$ with $b=A\ol{u}+\nu$ and $\|\nu\|_2\le \veps$,
the solution $u$ of \eqref{eq:TVnoise}
satisfies the \emph{gradient error bound}
\begin{equation*}
 \|\nabla u - \nabla \ol{u} \|_1 \le \|\nabla \ol{u} - (\nabla \ol{u})_k\|_1+\sqrt{k} \veps ,
\end{equation*}
and the \emph{signal error bound}
\begin{equation*}
 \|u-\ol{u}\|_2 \le \log \left(\frac{n}{k}\right)\frac{\|\nabla \ol{u} - (\nabla \ol{u})_k\|_1}{\sqrt{k}}+ \veps .
\end{equation*}
\end{theorem}
Note that recovery is exact when $\nabla \ol{u}$ is exactly $k$-sparse and $\|\nu\|_2=0$.

The RIP assumption on $AH^{-1}=AH^\top$ implies that $\mc{N}(A)$ cannot
contain any signals admitting a $k$-sparse wavelet expansion, apart from the zero vector, since
$\|Av\|_2 = \|A H^{-1}  H v \|_2 \approx (1 \pm \delta) \|Hv\|_2 = (1 \pm \delta)\|v\|_2$,
with the last equality holding because $H^\top H=I$. 
 

There exist sensing matrices $A\in\Rmn$ which satisfy the above conditions,
e.g.~$RIP_{1/3,k}$, where $k$ can be as large
as $O(m/\log(m/n))$. This class includes matrices with 
i.i.d. standard Gaussian or $\pm 1$ entries, random submatrices of the Fourier
transform and other orthogonal matrices.

In our scenario, however, due to the low RIP order of the tomographic projection matrix $A$,
for any image dimension $d$, the RIP order of $AH^{-1}$ does not improve significantly.
To illustrate this point, let us consider further the bivariate discrete Haar transform $H$ and a sparse nullspace vector
$u_\mc{N}$ with $\|u_\mc{N}\|_0=16$, depicted in Fig.~\ref{fig:nullspace_vec}, left panel, of the 2D
projection matrix $A$ from 6, 7 or 8 projections. Then
$0=\|A u_\mc{N}\|_2 = \|A H^{-1}  H u_\mc{N}\|_2$  holds with $\|H u_\mc{N}\|_0=32$,
see Fig.~\ref{fig:nullspace_vec}, right. Thus, the matrix $A H^{-1}$ cannot satisfy RIP of an order
larger than $32-1$, and this holds for any $d\ge 16$. Consequently, suppose it does accordingly satisfy $RIP_{1/3,31}$, then
Thm.~\ref{thm:TVstable} would imply exact recovery of any image with a $6$-sparse image gradient.
Unfortunately, such an extremely low sparsity is of limited use for practical applications.

\section{Co-Sparsity and TV-Based Reconstruction}\label{sec:CoSparse}We introduce some basic definitions related to the cosparsity of a given analysis operator $B$ in Section \ref{sec:co-definitions} followed by uniqueness results from \cite{Lu2008,Nam2013} in Section \ref{sec:co-sparsity-uniqueness}. The corresponding conditions imply bounds for the number $m$ of measurements, depending on the cosparsity of the vector $u$, that should be reconstructed, with respect to $B$. We apply these results in Section \ref{sec:co-TV} to the discrete gradient operator $B = \nabla$ given by
\eqref{eq:def-nabla}. This requires to estimate the dimension of the subspace of $\ell$-cosparse vectors. We relate this problem to the isoperimetric problem on grid graphs studied by \cite{Bollobas1991}. In this more general way, we reproduce the estimate proved differently in \cite{Nam2013} for the 2D case and additionally provide an estimate for the 3D case.

\subsection{Definitions} \label{sec:co-definitions}
Let $B$ be any given analysis operator.
\begin{definition}[cosparsity, cosupport]
The \emph{cosparsity} of $u \in \R^{n}$ \emph{with respect to} $B \in \R^{p \times n}$ is
\begin{equation}\label{eq:def-cosparsity}
\ell := p - \|B u\|_{0},
\end{equation}
and the \emph{cosupport} of $u$ \emph{with respect to} $B$ is
\begin{equation}\label{eq:def-cosupport}
\Lambda := \{ r \in [p] \colon (B u)_{r} = 0 \},\qquad
|\Lambda| = \ell.
\end{equation}
\end{definition}
We denote by $B_{r}$ the $r$-th row of the matrix $B$ and by $B_{\Lambda}$ the submatrix of $B$ formed by the \emph{rows} indexed by $\Lambda \subset [p]$. Thus, a $\ell$-cosparse vector $u$ satifies $B_{\Lambda} u = 0$, hence is contained in the subspace
\begin{equation} \label{eq:def-W-Lambda}
\mc{W}_{\Lambda} := \mc{N}(B_{\Lambda}).
\end{equation}
In this connection, we define the basic function
\begin{equation}\label{eq:def-kappa_B}
\kappa_{B}(\ell) := \max_{|\Lambda| \geq \ell} \dim \mc{W}_{\Lambda}.
\end{equation}

\subsection{Basic Uniqueness results} 
\label{sec:co-sparsity-uniqueness}

This section collects some results from \cite{Nam2013} that were derived based on \cite{Lu2008}.
\begin{proposition}[Uniqueness with known cosupport]\label{prop:unique-cs-known}
Let $A \in \R^{m \times n}$ and $B \in \R^{p \times n}$ be given measurement and analysis operators and assume the rows of the matrix $\bsm A \\ B \esm$
are linearly independent. Then, if the cosupport $\Lambda \subset [p]$ of the $\ell$-cosparse vector $u \in \R^{n}$ is known, the condition
\begin{equation}\label{eq:cond-unique-cs-known}
\kappa_{B}(\ell) \leq m
\end{equation}
is necessary and sufficient for recovery of every such vector from the measurements $b = A u$.
\end{proposition}
Proposition \ref{prop:unique-cs-known} says that if the dimension of the subspace $\mc{W}_{\Lambda}$ increases, then more measurements are needed for recovery of $\ell$-cosparse vectors $u \in \mc{W}_{\Lambda}$. The dimension $\dim \mc{W}_{\Lambda}$ increases for decreasing $\ell$.

\begin{proposition}[Uniqueness with unknown cosupport]\label{prop:unique-cs-unknown}
Let $A \in \R^{m \times n}$ and $B \in \R^{p \times n}$ be given measurement and analysis operators, and assume the rows of the matrix $\bsm A \\ B \esm$
are linearly independent. Then a necessary condition for uniqueness of a $\ell$-cosparse solution $u$ to the measurement equations $A u = b$ is
\begin{equation}
\tilde\kappa_{B}(\ell) \leq m,\qquad
\tilde\kappa_{B}(\ell) := \max\big\{ \dim(\mc{W}_{\Lambda_{1}} + \mc{W}_{\Lambda_{2}}) \colon |\Lambda_{i}| \geq \ell,\, i=1,2 
\big\},
\end{equation}
whereas a sufficient such condition is
\begin{equation}
\kappa_{B}(\ell) \leq \frac{m}{2},
\end{equation}
with $\kappa_{B}$ from \eqref{eq:def-kappa_B}.
\end{proposition}
Roughly speaking, lack of knowledge of $\Lambda$ implies the need of twice the number of measurements for unique recovery.
\begin{remark}\label{rem:dependency-A-B}
Both propositions assume the rows of $A$ and $B$ are independent. This is neither the case for typical sensor matrices $A$ used in discrete tomography nor in the specific case $B = \nabla$ considered next.

Our experimental results will show, however, that the estimates of $\kappa_{B}(\ell) = \kappa_{\nabla}(\ell)$ derived in Section \ref{sec:co-TV} correctly display the relationship between the basic parameters involved, up to some scale factor discussed in Section \ref{sec:Experiments}.
\end{remark}

\subsection{Application to the Analysis Operator $\nabla$}
\label{sec:co-TV}

In order to apply the results of Section \ref{sec:co-sparsity-uniqueness}, the function \eqref{eq:def-kappa_B} has to be evaluated, or estimated, in the case $B=\nabla$. 

For a given cosupport $\Lambda \subset E$, define the set of vertices covered by $\Lambda$,
\begin{equation}
V(\Lambda) = \{ v \in V \colon v \in e 
\;\text{for some}\; e \in \Lambda \},
\end{equation}
and denote the number of connected components of $V(\Lambda)$ by $|V(\Lambda)|_{\sim}$. Due to definition \eqref{eq:def-nabla} of the analysis operator $\nabla$ and \eqref{eq:def-cosupport}, 
each component $(\nabla_{\Lambda} u)_{i}$ corresponds to an edge $e = (v_{1},v_{2})$ with $u(v_{1})=u(v_{2})$. Therefore, following the reasoning in \cite{Nam2013}, $u \in \mc{W}_{\Lambda} = \mc{N}(\nabla_{\Lambda})$ if and only if $u$ is constant on each connected component of $V(\Lambda)$. Hence $\dim\mc{W}_{\Lambda}$ equals the size of the remaining vertices $|V \setminus V(\Lambda)|$ plus the degree of freedom for each connected component,
\begin{equation}
\dim\mc{W}_{\Lambda} = |V| - |V(\Lambda)| + |V(\Lambda)|_{\sim}.
\end{equation}
Now, in view of \eqref{eq:def-kappa_B}, consider some $\Lambda$ with $|\Lambda| = \ell$ and the problem
\begin{equation} \label{eq:kappa-subproblem}
\max_{\Lambda \colon |\Lambda|=\ell} \dim\mc{W}_{\Lambda} 
= |V| - \min_{\Lambda \colon |\Lambda|=\ell} 
(|V(\Lambda)| - |V(\Lambda)|_{\sim}).
\end{equation}
Clearly, the minimal value of the last term is $|V(\Lambda)|_{\sim} = 1$. It will turn out below that this value is attained for extremal sets $\Lambda$ and that the maximum in \eqref{eq:def-kappa_B} is achieved for $|\Lambda| = \ell$.

We therefore temporarily ignore the last term and focus on the second term. The problem is to \emph{minimize} over all subsets $\Lambda \subset E$ of cardinality $|\Lambda|=\ell$ the number $|V(\Lambda)|$ of vertices covered by $\Lambda$. We establish this relationship by considering instead the problem of \emph{maximizing} the set $\Lambda$ over all sets $S := V(\Lambda) \subseteq V$ of fixed cardinality $s = |S(\Lambda)|$. This problem was studied in \cite{Bollobas1991} for regular grid graphs $G = (V,E)$ with vertex set $V = [q]_{0}^{d}$ with equal dimension along each coordinate, in terms of the problem
\begin{equation}
\max_{S \colon |S|=s} |\mrm{Int}_{e}(S)|,
\end{equation}
where $\mrm{Int}_{e}(S)$ denotes the \emph{edge interior} of a set $S \subset V(G)$,
\begin{equation}
\mrm{Int}_{e}(S) := \{(v_{1},v_{2}) \in E \colon v_{1},v_{2} \in S\},
\end{equation}
which equals $\mrm{Int}_{e}(S) = \Lambda$ for our definition $S = V(\Lambda)$.
\begin{figure}[t]
\begin{center}
\includegraphics[width=0.6\textwidth]{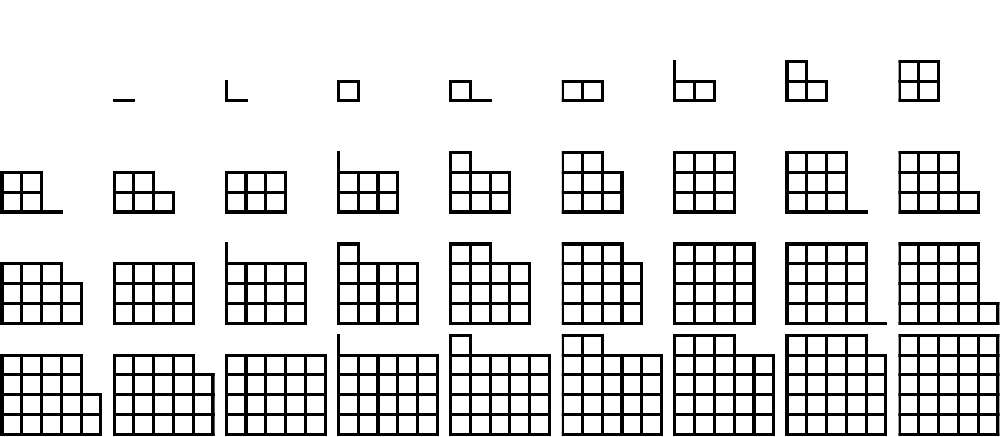}
\caption{From left to right, top to bottom: 
Edge sets $\Lambda$ corresponding to the subsets $S = V(\Lambda) \subseteq V = [q]_{0}^{d},\, q=5,\,d=2$, of cube-ordered vertices of cardinalities $s = |S|=1,2,\dotsc,|V|$. According to Thm.~\ref{thm:maximal-sets}, these sets belong to the maximizers of $|\Lambda|$ among all subsets $V(\Lambda) \subseteq V$ with fixed $s = |V(\Lambda)|$.
}
\label{fig:CubeOrder}
\end{center}
\end{figure}
\begin{theorem}[{\cite[Thm.~13]{Bollobas1991}}]
Let $S$ be a subset of $[q]_{0}^{d}$, with $s = |S|$, $q \geq 3,\, d \geq 2$. Then
\begin{equation} \label{eq:edgeinterior-bound-Bollobas}
|\mrm{Int}_{e}(S)| \leq \max\Big\{
d s (1-s^{-1/d}), d q^{d} (1-1/q) \big(1-(1-s/q^{d})^{1-1/d}\big)
\Big\}.
\end{equation}
\end{theorem}
Some sets $S = V(\Lambda) \subseteq V$ corresponding to maximal sets $\Lambda = \mrm{Int}_{e}(S)$ are also determined in \cite{Bollobas1991}. The following corresponding assertion is based on the \emph{cube order} or vertices $v \in V = [q]_{0}^{d}$ (identified with grid vectors $v = (v_{1},\dotsc,v_{d})^{\T}$, cf.~Section 
\ref{sec:Notation}): $v \prec v' \;\gdw\; w(v) < w(v')$, where $w(v) = \sum_{i \in [d]} 2^{i + d v_{i}}$. See Figure \ref{fig:CubeOrder} for an illustration.
\begin{theorem}[{\cite[Thm.~15]{Bollobas1991}}]\label{thm:maximal-sets}
Let $S' \subset V = [q]^{d}$, and let $S$ be the set of the first $s = |S'|$ vertices in the cube order on $V$. Then $|\mrm{Int}_{e}(S')| \leq |\mrm{Int}_{e}(S)|$.
\end{theorem}
Thm.~\ref{thm:maximal-sets} says (cf.~Fig.~\ref{fig:CubeOrder}) that singly connected mimimal sets $V(\Lambda)$ in \eqref{eq:kappa-subproblem} are achieved, that is $|V(\Lambda)|_{\sim} = 1$. Furthermore, these sets $\{\Lambda\}_{|\Lambda| \geq \ell}$ are nested. Hence the maximum in \eqref{eq:def-kappa_B} is achieved for $|\Lambda|=\ell$.

A closer inspection of the two terms defining the upper bound \eqref{eq:edgeinterior-bound-Bollobas} shows that the first term of the r.h.s.~is larger if $s \geq d^d=4$ in the 2D case $d=2$, respectively, if $s \geq d^d=27$ in the 3D case $d=3$. The corresponding values of the bound are $|\mrm{Int}_{e}(S)| \leq 4$ and $|\mrm{Int}_{e}(S)| \leq 54$, respectively. As a consequence, we consider the practically relevant first term. Setting $\ell = |\mrm{Int}_{e}(S)| = |V(\Lambda)|$ and solving the equality for $s$ (due to Thm.~\ref{thm:maximal-sets}) yields
\begin{subequations} \label{eq:s-kappa-bounds}
\begin{align}
s &= \frac{1}{2}(1+\ell+\sqrt{1+2 \ell}) & & (d=2) \\
s &= \frac{1}{3} \Big( 2^{1/3} \frac{1+2 \ell}{t(\ell)} + 
\big(1+\ell+ \frac{1}{2^{1/3}} t(\ell)\big) \Big) & & (d=3) \\
&\qquad
t(\ell) = \Big(2 + 6 \ell + 3 \ell^{2} + \sqrt{(4 + 9 \ell) \ell^{3}}
\Big)^{1/3} \\
&\geq \frac{1}{3} \Big(1 + \ell + (3 \ell^{2})^{1/3}
+ 2 (\ell/3)^{1/3} \Big) + O(\ell^{-1/3}).
\end{align}
\end{subequations}
Inserting $s$, or simpler terms lower bounding $s$, for $|V(\Lambda)|$ in \eqref{eq:kappa-subproblem} and putting all conclusions together, yields for \eqref{eq:def-kappa_B} and $B = \nabla$:
\begin{lemma}
Let $G=(V,E)$ be a regular grid graph with $V = [q]_{0}^{d}$, $n = |V| = q^{3}$. Then
\begin{subequations} \label{eq:kappa-ub}
\begin{align}
\forall \ell > 4,\qquad
\kappa_{\nabla}(\ell) &\leq n - \frac{1}{2}(\ell + \sqrt{1 + 2 \ell}) + \frac{1}{2}, & & (d=2), 
\label{eq:kappa-ub2D} \\ \label{eq:kappa-ub3D}
\forall \ell > 54,\qquad
\kappa_{\nabla}(\ell) &\leq n 
- \frac{1}{3} \Big(\ell + \sqrt[3]{3 \ell^{2}} 
+ 2 \sqrt[3]{\frac{\ell}{3}} \Big) + \frac{2}{3}, & & (d=3).
\end{align}
\end{subequations}
\end{lemma}
Figure \ref{fig:kappa-nabla} illustrates these bounds.
\begin{figure}
\centerline{
\includegraphics[width=0.45\textwidth]{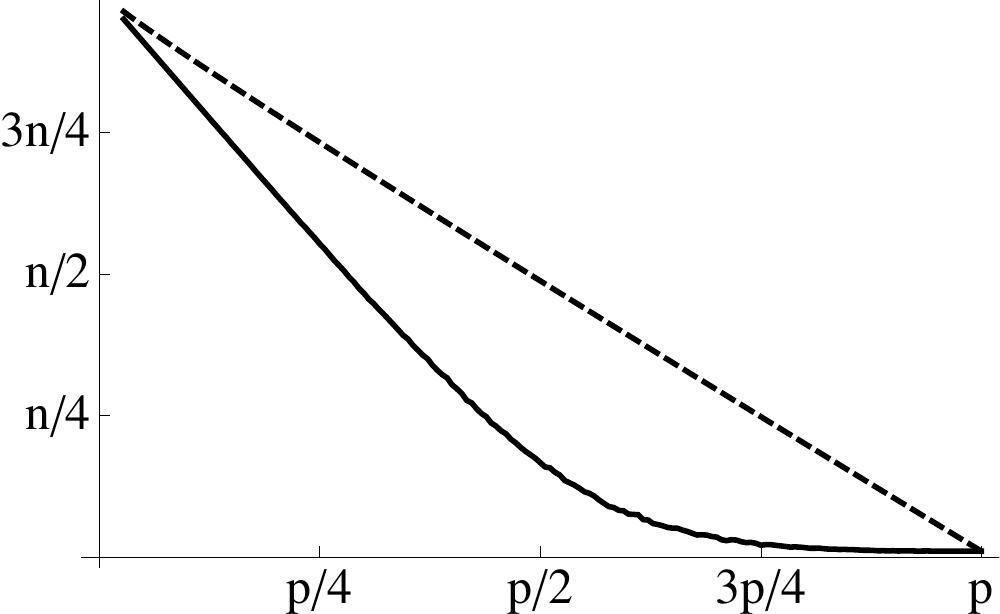}
\hspace{0.02\textwidth}
\includegraphics[width=0.45\textwidth]{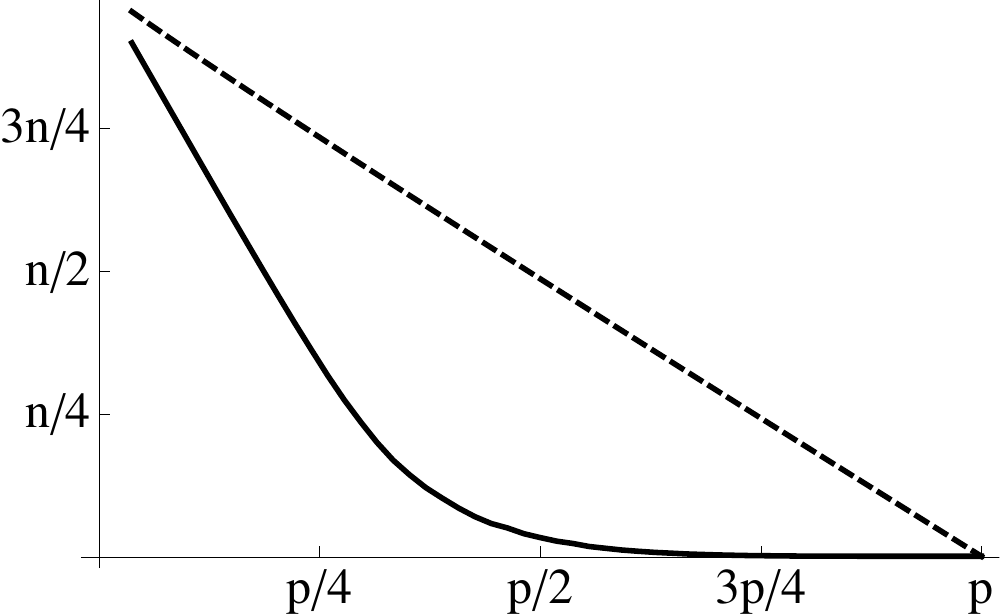}
}
\caption{The bounds \eqref{eq:kappa-ub2D} and \eqref{eq:kappa-ub3D} shown in the left and right panel respectively, as dashed lines, for $d=10$, as a function of $\ell$. The solid lines show empirical expected values of $\kappa_{\nabla}(\ell)$ computed by averaging over $100$ analysis matrices $B_{\Lambda}$ for each value $\ell=|\Lambda|$. The gap to the upper bound simply shows that the extremal sets discussed in connection with Theorem \ref{thm:maximal-sets} are not observed in random experiments. For example, it is quite unlikely that random cosupports $\Lambda$ are singly connected.
}
\label{fig:kappa-nabla}
\end{figure}

\begin{remark}
Up to the constant $1$ under the square root, the bound \eqref{eq:kappa-ub2D} for the 2D case equals the bound derived in a different way in \cite{Nam2013}. We provided the bounds \eqref{eq:kappa-ub} based on general results in \cite{Bollobas1991} that apply to grid graphs in any dimension $d \geq 2$.
\end{remark}
We conclude this section by applying Propositions \ref{prop:unique-cs-known} and \ref{prop:unique-cs-unknown}.
\begin{corollary}\label{cosparsity_curves}
Under the assumptions of Propositions \ref{prop:unique-cs-known} and \ref{prop:unique-cs-unknown}, a $\ell$-cosparse solution $u$ to the measurement equations $A u = b$ will be unique if the number of measurements satisfies
\begin{subequations}\label{eq:m-lb-known}
\begin{align}
m &\geq n - \frac{1}{2} \Big(\ell + \sqrt{2 \ell + 1} - 1 \Big)
& & (d=2), 
\label{eq:m-lb-known-2D} \\ \label{eq:m-lb-known-3D}
m &\geq n - \frac{1}{3}\Big(\ell + \sqrt[3]{3 \ell^2} 
+ 2 \sqrt[3]{\frac{\ell}{3}} - 2 \Big)
& & (d=3)
\end{align}
\end{subequations}
in case the cosupport $\Lambda$ is known, and 
\begin{subequations}\label{eq:m-lb-unknown}
\begin{align}
m &\geq 2 n - (\ell + \sqrt{2 \ell + 1} - 1)
& & (d=2), 
\label{eq:m-lb-unknown-2D} \\ \label{eq:m-lb-unknown-3D}
m &\geq 2 n - \frac{2}{3}\Big(\ell + \sqrt[3]{3 \ell^2} 
+ 2 \sqrt[3]{\frac{\ell}{3}} - 2 \Big)
& & (d=3)
\end{align}
\end{subequations}
in case the cosupport $\Lambda$ is unknown.
\end{corollary}
The above derived bounds on the required image cosparsity 
guarantees uniqueness in case of known or unknown cosupport, and imply that
recovery can be carried out via
\begin{equation}\label{eq:SNTV_known}
\min_{u} \|B u\|_{0} \quad\text{s.t.}\quad A u =b, B_\Lambda u=0,
\end{equation}
when the cosupport is known, or via
\begin{equation}\label{eq:SNTV_unknown}
\min_{u} \|B u\|_{0} \quad\text{s.t.}\quad A u =b,
\end{equation}
when the cosupport is unknown.
In Section \ref{sec:Experiments}, we compare these relationships to numerical results
involving convex relaxations of \eqref{eq:SNTV_known} and \eqref{eq:SNTV_known}, studied in Section \ref{sec:minTV}.

\section{Recovery by Linear Programming}\label{sec:minTV}

In this section, uniqueness of the optimum $\ol{u}$ solving problem \eqref{eq:TVmin_pos} is studied. The resulting condition is necessary and sufficient for unique recovery $\ol{u} = u^{\ast}$ of any $\ell$-cosparse vector $u^{\ast}$ that satisfies $A u^{\ast} = b$ and has cosupport $\Lambda,\, |\Lambda|=\ell$, with respect to the analysis operator $B=\nabla$.

We turn problem \eqref{eq:TVmin_pos} into a standard linear programming formulation. 
Defining
\begin{equation} \label{eq:def-Mq}
M := \bpm B & -I & I \\ A & 0 & 0 \epm,\quad
q := \bpm 0 \\ b \epm
\end{equation}
and the polyhedral set
\begin{equation}
\mc{P} := \{ w \in \R^{n+2 p} \colon M w = q,\; w \geq 0 \},\quad
w := \bpm u \\ v \epm = \bpm u \\ v^{1} \\ v^{2} \epm,
\end{equation}
problem \eqref{eq:TVmin_pos} equals the linear program (LP)
\begin{equation} \label{eq:primal-LP}
\min_{w \in \mc{P}} \la c, w \ra 
= \min_{(u, v^{1}, v^{2}) \in \mc{P}} \la \eins, v^{1} + v^{2} \ra,
\quad
c = \bpm 0 \\ \eins \\ \eins \epm.
\end{equation}
Let $\ol{w}=(\ol{u},\ol{v}) = (\ol{u},\ol{v}^{1},\ol{v}^{2})$ solve \eqref{eq:primal-LP}. We assume throughout 
\begin{equation} \label{eq:u-nonzero}
\ol{u}_{i} > 0,\; i \in [n]
\end{equation}
which is not restrictive with respect to applications ($u$ may e.g.~represent strictly positive material densities).
Based on $\ol{w}$, we define the corresponding index sets
\begin{equation} \label{eq:def-Jv}
J := \{ i \in [\dim(w)] \colon \ol{v}_{i}=0\},\quad
\ol{J} := \{ i \in [\dim(v)] \colon \ol{v}_{i}=0\},\qquad
w_{J} = v_{\ol{J}},\; \forall w = \bpm u \\ v \epm.
\end{equation}
\begin{theorem}
[{\cite[Thm.~2(iii)]{Mangasarian1979}}]
\label{thm:M-uniqueness}
Let $\ol{w}$ be a solution of the linear program \eqref{eq:primal-LP}. The following statements are equivalent:
\begin{enumerate}[(i)]
\item $\ol{w}$ is unique.
\item There exists no $w$ satisfying
\begin{equation} \label{eq:cond-uniqueness-M79}
M w = 0,\quad
w_{J} \geq 0,\quad
\la c, w \ra \leq 0,\quad
w \neq 0.
\end{equation}
\end{enumerate}
\end{theorem}
We turn Theorem \eqref{thm:M-uniqueness} into a \emph{nullspace condition} w.r.t.~the sensor matrix $A$, for the unique solvability of problems \eqref{eq:primal-LP} and \eqref{eq:TVmin_pos}. This condition is stated as Corollary \ref{cor:uniqueness} below, after a preparatory Lemma.

\begin{lemma} \label{lem:olJ}
Let $\ol{w}$ be a solution of the LP \eqref{eq:primal-LP}. Then the cardinality of the index set $\ol{J}$ defined by \eqref{eq:def-Jv} is 
\begin{equation}
|\ol{J}| = 2 \ell + k = p + \ell,\qquad 
|\ol{J}^{c}| = 2 p - |\ol{J}| = k,\qquad
k := p-\ell.
\end{equation}
\end{lemma}
\begin{proof}
The minimal objective function value \eqref{eq:primal-LP} is $\sum_{i \in [p]} \ol{v}^{1}_{i} + \ol{v}^{2}_{i}$ with all summands being non-negative. Since $B \ol{u} = \ol{v}^{1}-\ol{v}^{2}$, $(B \ol{u})_{\Lambda}=0$ and optimality of $\ol{v}$ imply $\ol{v}^{1}_{\Lambda}=\ol{v}^{2}_{\Lambda} = 0$, which contributes $2|\Lambda|=2 \ell$ indices to $\ol{J}$. Furthermore, if $(B \ol{u})_{i} = \ol{v}^{1}_{i} - \ol{v}^{2}_{i} < 0$, then optimality of $\ol{v}$ implies $\ol{v}^{1}_{i}=0,\, \ol{v}^{2}_{i} > 0$ and vice versa if $(B \ol{u})_{i} > 0$. Hence $\Lambda^{c}$ supports $|\Lambda^{c}| = p - \ell = k$ vanishing components of $\ol{v}$.
\end{proof}

\begin{corollary} \label{cor:uniqueness}
Let $\ol{w} = (\ol{u},\ol{v}^{1},\ol{v}^{2})$ be a solution of the linear program \eqref{eq:primal-LP} with corresponding index sets $J, \ol{J}$ given by \eqref{eq:def-Jv}, and with component $\ol{u}$ that solves problem \eqref{eq:TVmin_pos} and has cosupport $\Lambda$ with respect to $B$. Then $\ol{w}$ resp.~$\ol{u}$ are unique if and only if 
\begin{equation} \label{eq:NM}
\forall w = \bpm u \\ v \epm,\; v = \bpm v^{1} \\ v^{2} \epm \quad{\rm s.t.}\quad
u \in \mc{N}(A) \setminus \{0\} 
\quad\text{and}\quad 
B u = v^{1}-v^{2}
\end{equation}
the condition
\begin{equation} \label{eq:cor-condition}
\|(B u)_{\Lambda}\|_{1} > 
\big\la (B u)_{\Lambda^{c}}, \sign(B \ol{u})_{\Lambda^{c}} \big\ra
\end{equation}
holds. Furthermore, any unknown $\ell$-cosparse vector $u^{\ast}$ with $A u^{\ast} = b$ can be uniquely recovered as solution $\ol{u}=u^{\ast}$ to \eqref{eq:TVmin_pos} if and only if, for all vectors $u$ conforming to \eqref{eq:NM}, the condition
\begin{equation} \label{eq:cor-condition-all}
\|(B u)_{\Lambda}\|_{1} > 
\sup_{\Lambda \subset [p] \colon |\Lambda|=\ell} 
\;
\sup_{\ol{u} \in \mc{W}_{\Lambda}} 
\big\la (B u)_{\Lambda^{c}}, \sign(B \ol{u})_{\Lambda^{c}} \big\ra
\end{equation}
holds.
\end{corollary}
\begin{remark}
Condition \eqref{eq:cor-condition} corresponds up to a magnitude $|\cdot|$ operation applied to the right-hand side to the statement of \cite[Thm.~7]{Nam2013}. The authors do not present an explicit proof, but mention in \cite[App.~A]{Nam2013} that the result follows by combining a strictly local minimum condition with convexity of the optimization problem for recovery.

Our subsequent explicit proof elaborates basic LP-theory due to \cite{Mangasarian1979} and Thm.~\ref{thm:M-uniqueness}.
\end{remark}
\begin{proof}[Proof of Corollary \ref{cor:uniqueness}]
Theorem \eqref{thm:M-uniqueness} asserts that $\ol{w}$ is unique iff for every $w \in \mc{N}(M) \setminus \{0\}$ with $w_{J} \geq 0$ the condition $\la c, w \ra > 0$ holds. In view of the definition \eqref{eq:def-Mq} of $M$, vectors $w \in \mc{N}(M) \setminus \{0\}$ are determined by \eqref{eq:NM}. Condition \eqref{eq:NM} excludes vectors $0 \neq w = (0,v^{1},v^{2}) \in \mc{N}(M)$ because then $v^{1}=v^{2}$ and $w_{J} \geq 0$ implies exclusion of those $w$ by $\la c,w \ra \leq 0$ in \eqref{eq:cond-uniqueness-M79}.

It remains to turn the condition \eqref{eq:cond-uniqueness-M79} into a condition for vectors $u$ given by vectors $w = (u,v^{1},v^{2})$ satisfying \eqref{eq:NM}. To this end, we focus on such vectors $w$ with $w_{J} \geq 0$ that minimize $\la c, w \ra$. We have $w_{J}=v_{\ol{J}}$ by \eqref{eq:def-Jv}, and the proof of Lemma \ref{lem:olJ} shows that $v_{\ol{J}} \geq 0$ decomposes into
\begin{itemize}
\item
$2 \ell$ conditions $v^{1}_{\Lambda}, v^{2}_{\Lambda} \geq 0$ leading to the choice
\begin{equation} \label{eq:v-choice-1}
\begin{cases}
v^{1}_{i} = (B u)_{i} \geq 0,\quad v^{2}_{i}=0, 
& \text{if}\; (B u)_{i} \geq 0, \\ 
v^{1}_{i} = 0,\quad v^{2}_{i}=-(B u)_{i} \geq 0,
& \text{if}\; (B u)_{i} \leq 0,
\end{cases}
\qquad i \in \Lambda
\end{equation}
minimizing $\la c, w \ra$;
\item
$k$ conditions supported by $\Lambda^{c}$ of the form: either $v^{1}_{i} \geq 0$ or $v^{2}_{i} \geq 0$ depending on $(B \ol{u})_{i} > 0$ or $(B \ol{u})_{i} < 0,\, i \in \Lambda^{c}$. In order to minimize $\la c, w \ra$, this leads to the choice
\begin{equation} \label{eq:v-choice-2}
\begin{cases}
v^{1}_{i} = 0,\quad
v^{2}_{i} = -(B u)_{i} \leq 0,
& \text{if}\; (B u)_{i} \geq 0,\, (B \ol{u})_{i} > 0, \\
v^{1}_{i} = 0,\quad
v^{2}_{i} = (B u)_{i} \geq 0,
& \text{if}\; (B u)_{i} \leq 0,\, (B \ol{u})_{i} > 0, \\
v^{1}_{i} = (B u)_{i} \geq 0,\quad
v^{2}_{i} = 0,
& \text{if}\; (B u)_{i} \geq 0,\, (B \ol{u})_{i} < 0, \\
v^{1}_{i} = (B u)_{i} \leq 0,\quad
v^{2}_{i} = 0,
& \text{if}\; (B u)_{i} \leq 0,\, (B \ol{u})_{i} < 0, 
\end{cases}
\qquad i \in \Lambda^{c}.
\end{equation}
\end{itemize}
By \eqref{eq:primal-LP}, $\la c, w \ra = \la \eins, v^{1} + v^{2} \ra = \la \eins, (v^{1} + v^{2})_{\Lambda} \ra + \la \eins, (v^{1} + v^{2})_{\Lambda^{c}} \ra$, and \eqref{eq:v-choice-1} shows that $\la \eins, (v^{1} + v^{2})_{\Lambda} \ra = \|(B u)_{\Lambda}\|_{1}$ whereas \eqref{eq:v-choice-2} shows that $\la \eins, (v^{1} + v^{2})_{\Lambda^{c}} \ra = \la (B u)_{\Lambda^{c}}, -\sign(B \ol{u})_{\Lambda^{c}} \ra$. Thus $\la c, w \ra \leq 0 \quad\gdw\quad \|(B u)_{\Lambda}\|_{1} - \la (B u)_{\Lambda^{c}}, \sign(B \ol{u})_{\Lambda^{c}} \ra \leq 0$, and non-existence of such $w$ means $\la c, w \ra > 0$ for every such $w$, which equals \eqref{eq:NM} and \eqref{eq:cor-condition}.

Finally, generalizing condition \eqref{eq:cor-condition} to all vectors $u^{\ast} \in \mc{W}_{\Lambda}$ and all possible cosupports $\Lambda$ leads to \eqref{eq:cor-condition-all}.
\end{proof}

Conditions \eqref{eq:cor-condition} and \eqref{eq:cor-condition-all} clearly indicate the direct influence of cosparsity on the recovery performance: If $\ell=|\Lambda|$ increases, then these conditions will more likely hold.

On the other hand, these results are mainly theoretical since numerically checking \eqref{eq:cor-condition-all} is infeasible. This motivates the comprehensive experimental assessment of recovery properties reported in Section \ref{sec:Experiments}.

\section{Numerical Experiments}\label{sec:Experiments}

In this section, we relate the previously derived bounds on the required image cosparsity 
that guarantees uniqueness in case of known or unknown cosupport $\Lambda$
to numerical experiments.

\subsection{Set-Up}\label{sec:num_setup}
This section describes how we generate 2D or 3D images for a given cosparsity $\ell$
and how we acquire measurements.
\subsubsection{Test Images}\label{TestImages}
Recall from Section \ref{sec:co-definitions} that the sparsity of the image gradient is denoted by $k$ and the cosparsity by $\ell$,
\begin{subequations}
	\begin{align}
		k& = \| B u \|_0 = \vert \text{supp}(B u) \vert,\quad B \in \mathbb{R}^{p \times n},\\
		\ell & =  p - \|B u \|_0 = p-k,
	\end{align}
\end{subequations}
$\Lambda := \{r \in [p] : (Bu)_r = 0\}$ denotes the cosupport of the input image $u$ with respect to the analysis operator $B$, and $\Lambda^c = [p]\setminus \Lambda$ denotes the complement of the set $\Lambda$.

Using the parametrization
\begin{align}
\label{eq:def-rho}
\rho &:= \frac{k}{n}
\intertext{with}
\label{eq:def-knp}
k := p - \ell
\qquad\text{and}\qquad
n &= \begin{cases}
d^{2} & \text{in 2D} \\
d^{3} & \text{in 3D}
\end{cases}
\qquad,\qquad
p = \begin{cases}
2 d (d-1) & \text{in 2D} \\
3 d^{2} (d-1) & \text{in 3D}
\end{cases},
\end{align}
we generated random 2D and 3D images composed of randomly located ellipsoids with random radii along the coordinate axes. Figures \ref{fig:ellipsoids-2D} and \ref{fig:ellipsoids-3D} depict a small sample of these images for illustration and provide the parameter ranges.

\begin{figure}[htbp]
\begin{center}
\includegraphics[width=\textwidth]{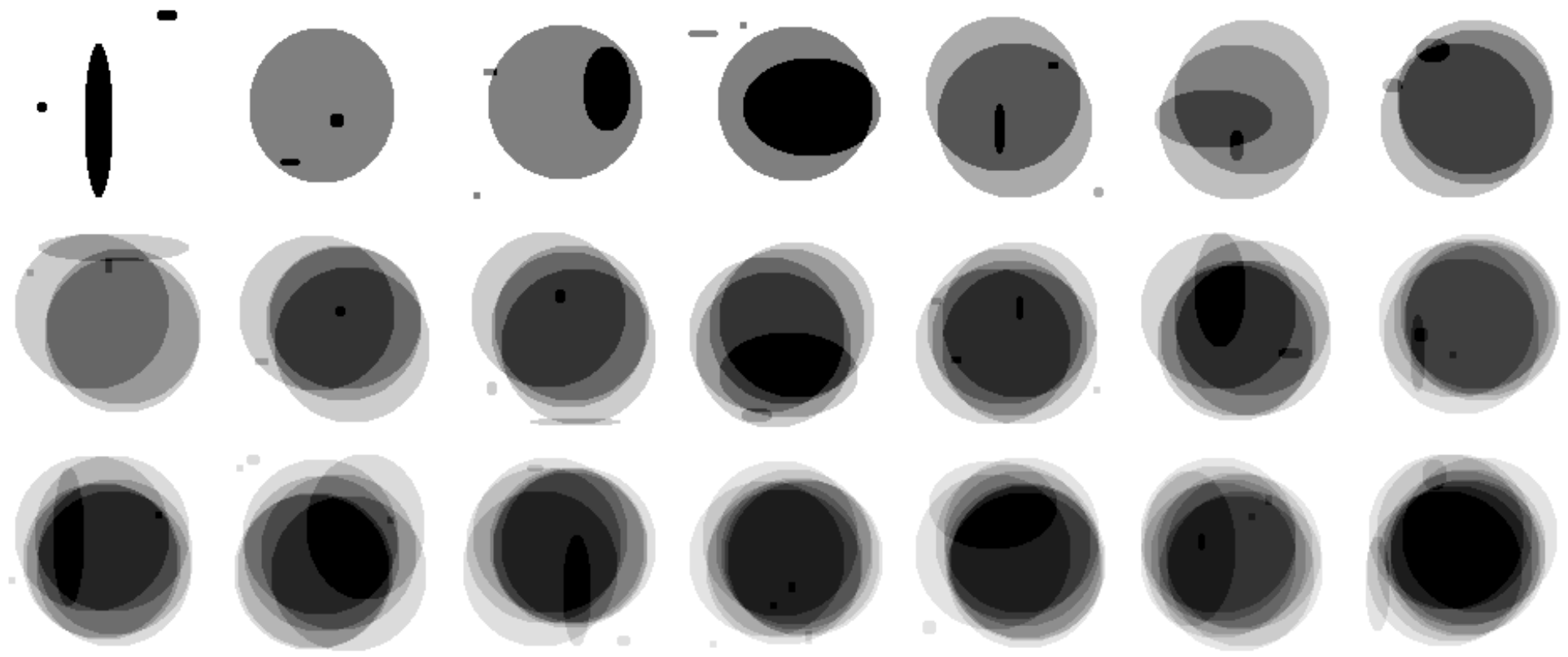}
\caption{Random images with varying cosparsity $\ell=p-k$, parametrized by $\rho$ \eqref{eq:def-rho}. For each dimension $d = 80 \dotsb 180$, random images were generated for $\rho = 0.005 \dotsb 0.22$. The figure shows a sample image for a subset of increasing values of $\rho$ and $d=120$.
}
\label{fig:ellipsoids-2D}
\end{center}
\end{figure}
\begin{figure}
\centerline{
\includegraphics[width=0.15\textwidth]{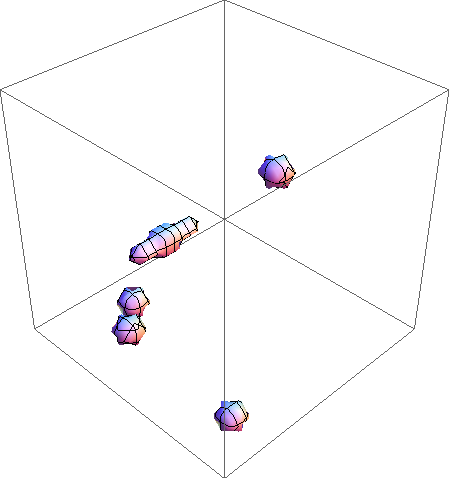}
\includegraphics[width=0.15\textwidth]{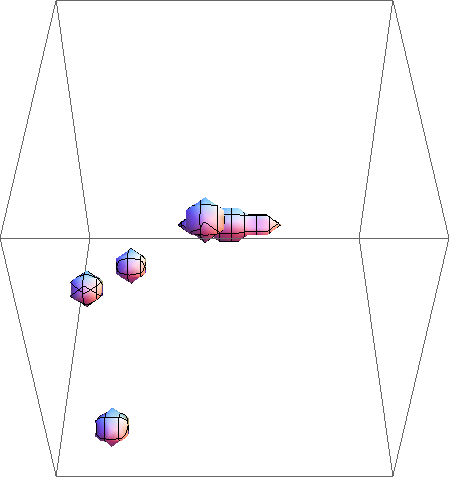}
\includegraphics[width=0.15\textwidth]{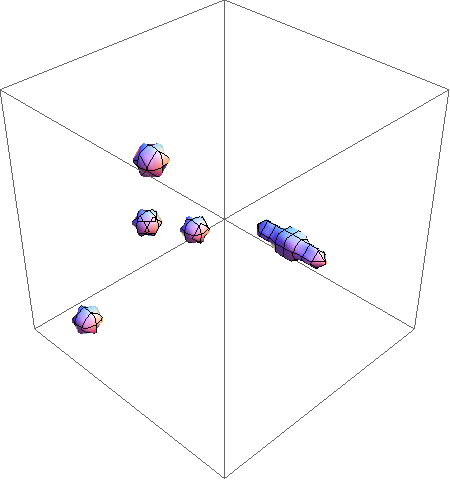}
}

\vspace{0.5cm}
\centerline{
\includegraphics[width=0.15\textwidth]{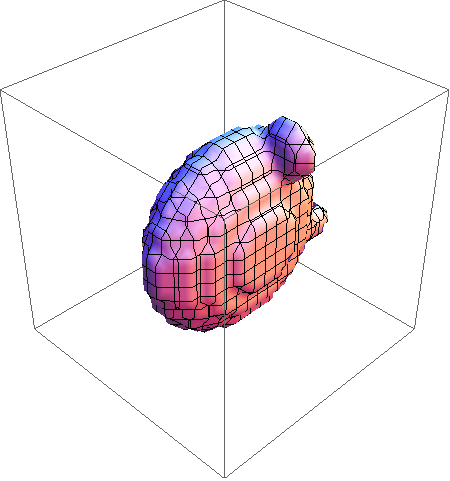}
\includegraphics[width=0.15\textwidth]{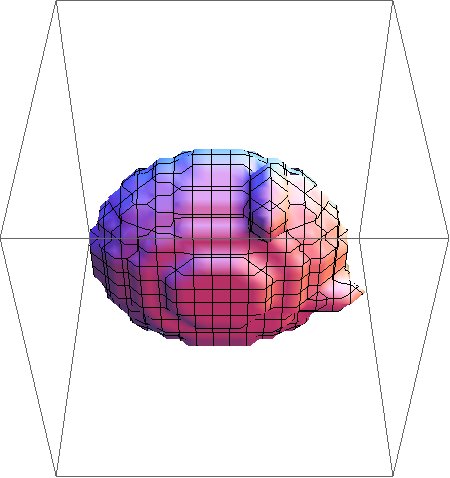}
\includegraphics[width=0.15\textwidth]{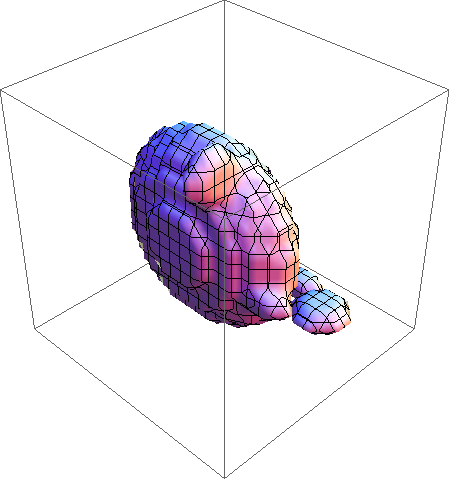}
\hfill
\includegraphics[width=0.15\textwidth]{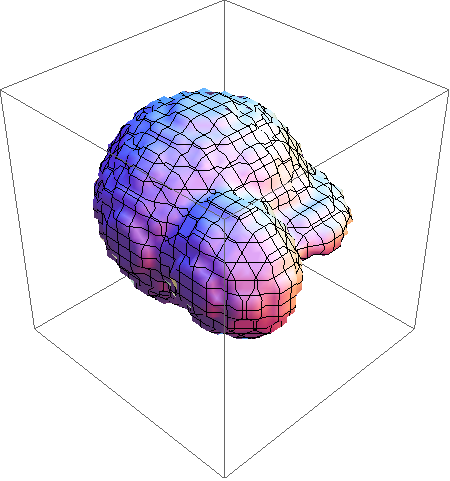}
\includegraphics[width=0.15\textwidth]{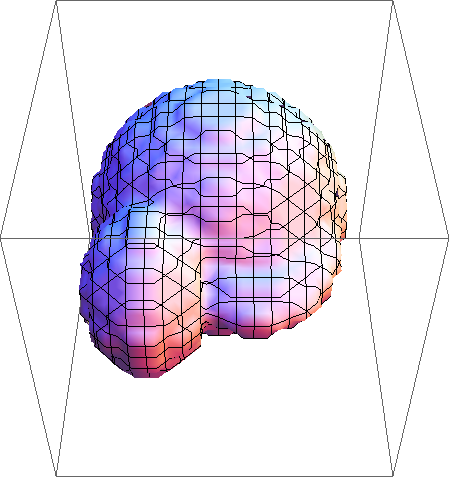}
\includegraphics[width=0.15\textwidth]{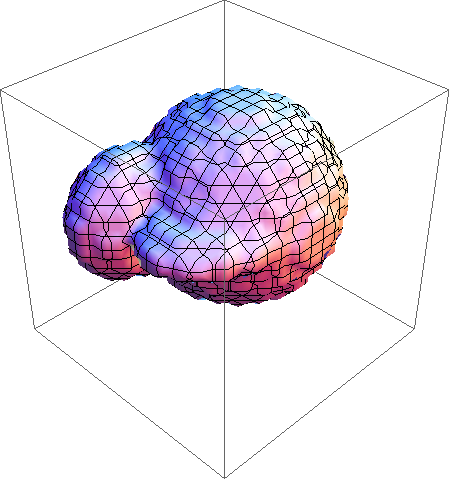}
}
\centerline{
\includegraphics[width=0.15\textwidth]{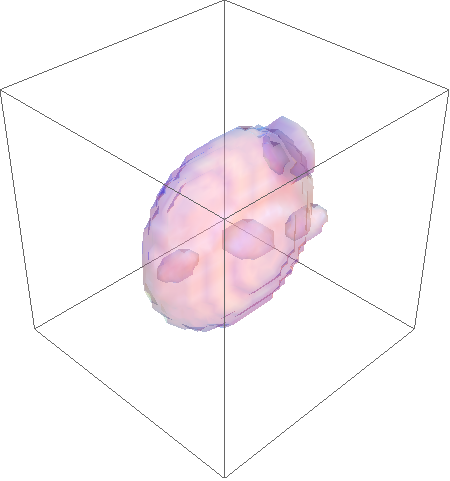}
\includegraphics[width=0.15\textwidth]{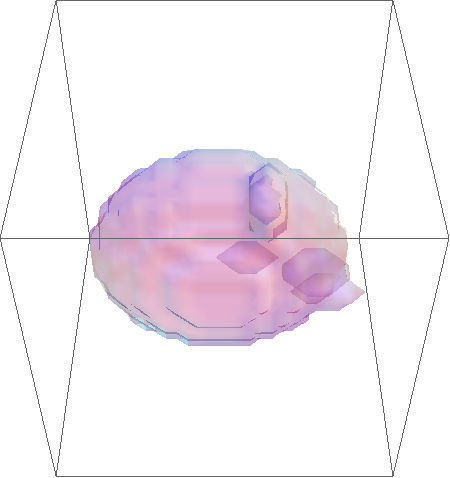}
\includegraphics[width=0.15\textwidth]{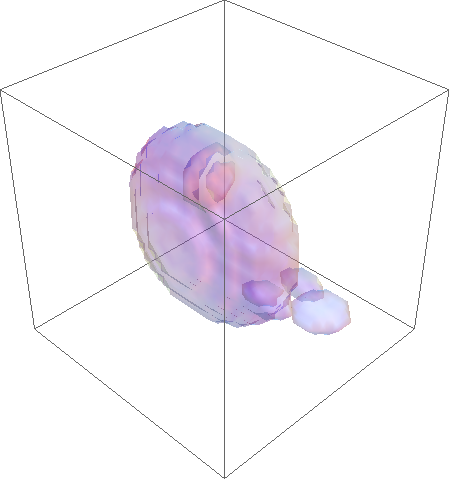}
\hfill
\includegraphics[width=0.15\textwidth]{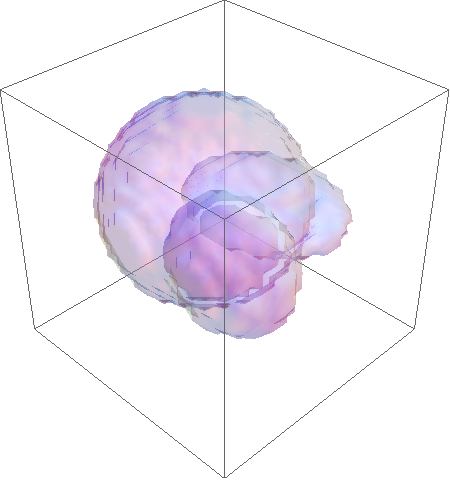}
\includegraphics[width=0.15\textwidth]{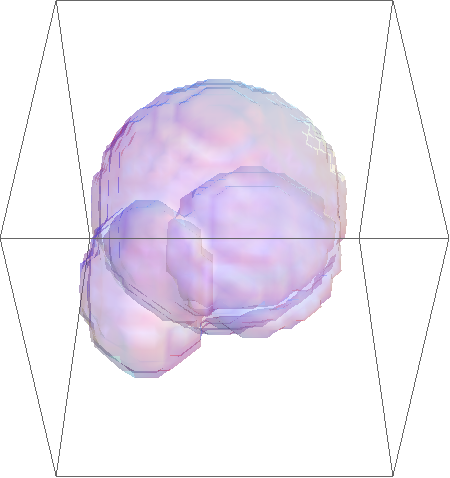}
\includegraphics[width=0.15\textwidth]{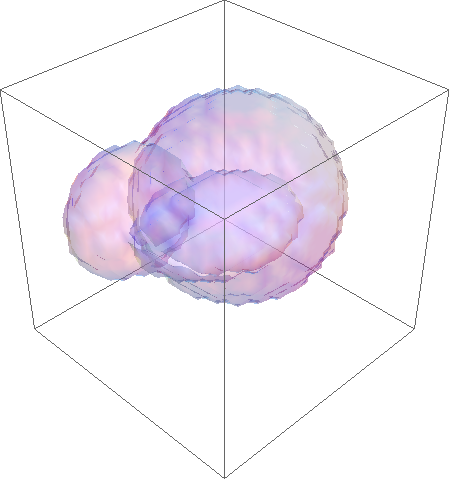}
}

\vspace{0.5cm}
\centerline{
\includegraphics[width=0.15\textwidth]{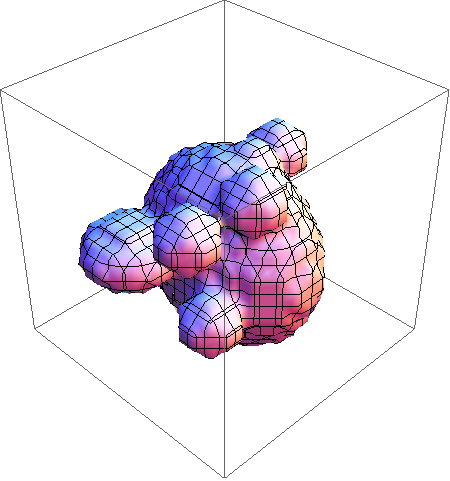}
\includegraphics[width=0.15\textwidth]{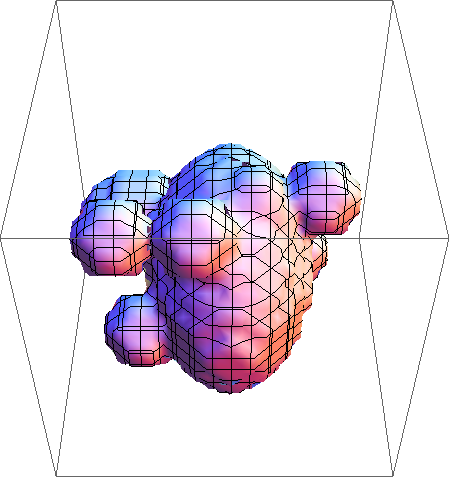}
\includegraphics[width=0.15\textwidth]{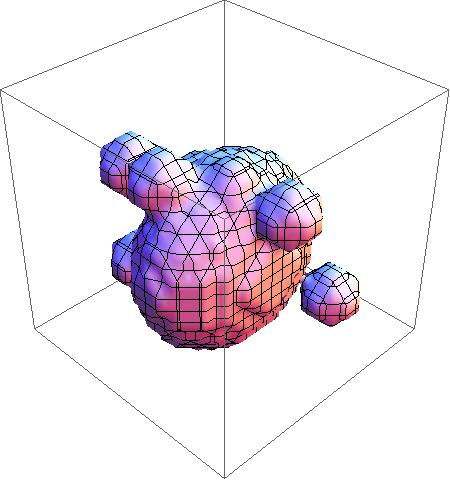}
\hfill
\includegraphics[width=0.15\textwidth]{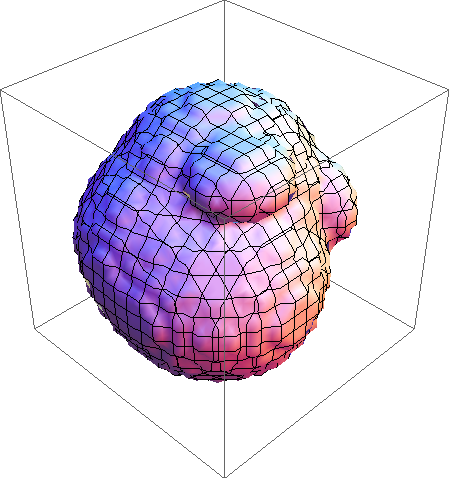}
\includegraphics[width=0.15\textwidth]{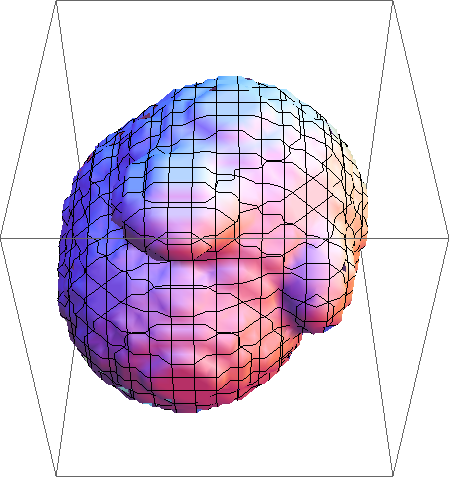}
\includegraphics[width=0.15\textwidth]{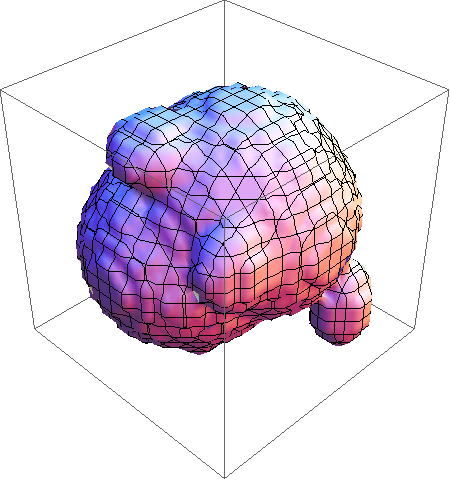}
}
\centerline{
\includegraphics[width=0.15\textwidth]{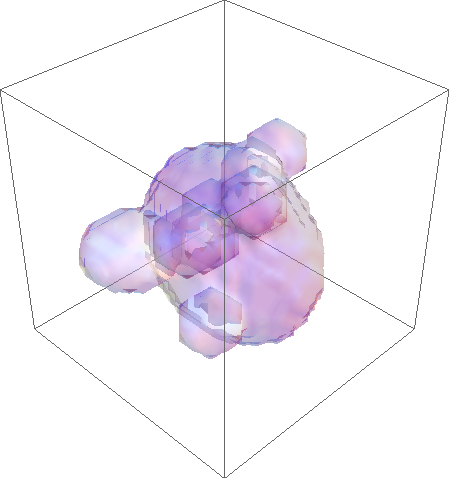}
\includegraphics[width=0.15\textwidth]{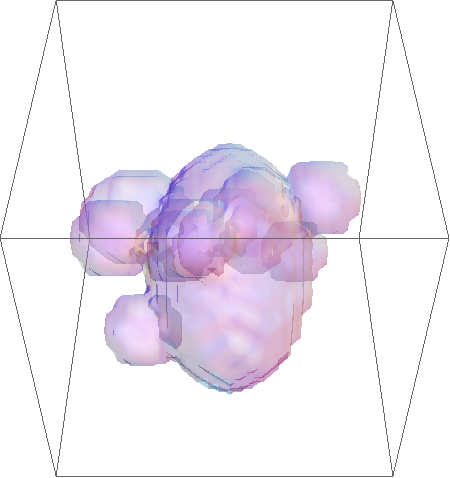}
\includegraphics[width=0.15\textwidth]{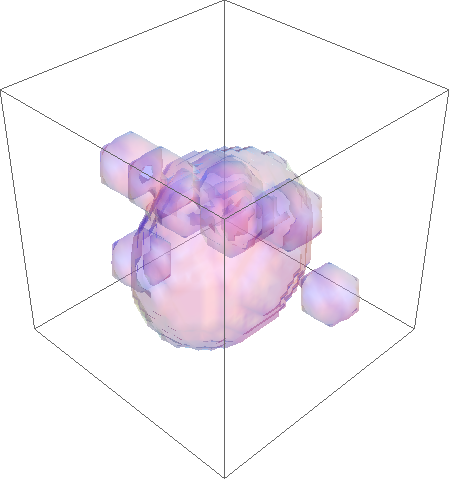}
\hfill
\includegraphics[width=0.15\textwidth]{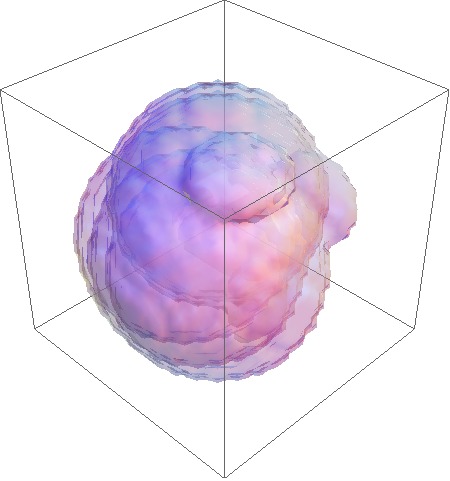}
\includegraphics[width=0.15\textwidth]{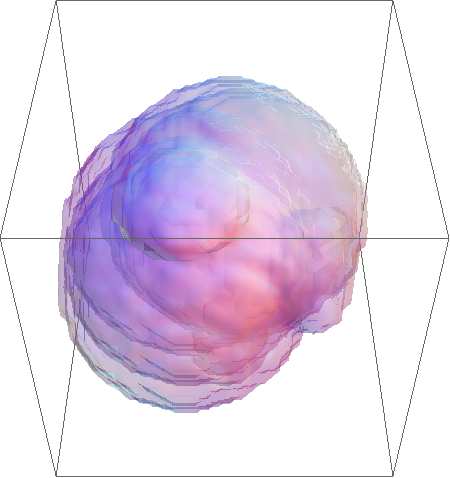}
\includegraphics[width=0.15\textwidth]{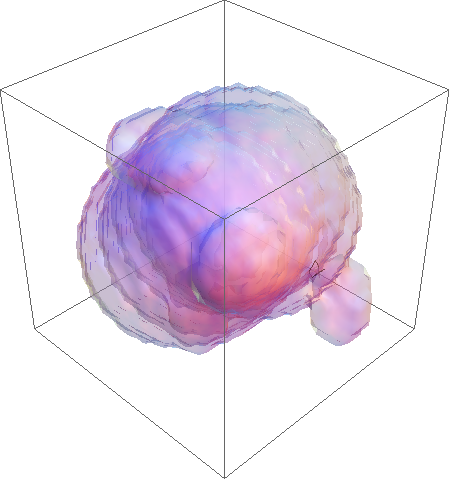}
}
\caption{
Random images with varying cosparsity $\ell=p-k$, parametrized by $\rho$ \eqref{eq:def-rho}. For dimension $d=31$, random images were generated for $\rho = 0.0032 \dotsb 1.01$. The figure shows a sample image for five different values of $\rho$, each plotted from three different viewpoints.
}
\label{fig:ellipsoids-3D}
\end{figure}

\subsubsection{Tomographic Projections}\label{Sensors}
Images in 2D are undersampled by the projection matrices from Section \ref{sec:setup_3D}, with parameters listed in Table \ref{tab:1}. In 3D
we consider the two projection matrices from
Section \ref{sec:setup_3D}, see Fig.~\ref{fig:3C3D} and Fig.~\ref{fig:4C3D}. We also consider a perturbation of each $A$.
Each perturbed matrix $\tilde A$ has the same sparsity structure as $A$, but random entries drawn from the standard
uniform distribution on the open interval $(0.9,1.1)$.


\subsection{Optimization}

To recover a $\ell$-cosparse test image $\ol{u}$, 
we solve the LP relaxation \eqref{eq:primal-LP} of
\eqref{eq:SNTV_unknown}, where we take into account the nonnegativity of $\ol{u}$.
The relaxation is obtained from \eqref{eq:TVmin_pos} by considering two additional variables ${v}^1$ and
$v^2$ which represent the positive and negative part of $Bu$.
In cases where we assume that $\Lambda$ is known, we add the constraint $B_\Lambda u=0$ and solve the LP with the same objective as
\eqref{eq:primal-LP}, but with the polyhedral feasible set defined by
\begin{equation}\label{eq:M_known_cosupport}
M := \bpm  
		B_{\Lambda^c} & -I_{\Lambda^c} & I_{\Lambda^c}\\
		B_{\Lambda} & 0 & 0 \\
		A & 0 & 0
	 \epm \quad {\rm and} \quad q := \bpm 0 \\0 \\ b \epm.
\end{equation}
The resulting LPs were solved with the help of a standard LP solver \footnote{MOSEK http://mosek.com/}.
The reconstruction is considered \emph{successful} if the solution $u$ of the above described LPs
is within a small distance from the original $\ell$-cosparse $\ol{u}$ generating the data, and
$\|u-\ol{u}\|_2\le \veps n$ holds, with $\veps=10^{-6}$ in 2D and $\veps=10^{-8}$ in 3D.

\subsection{Phase transitions}

Phase transitions display the empirical probability of exact recovery over the space of parameters that characterize the problem (cf.~\cite{DonohoT10}). Our parametrization relates to the design of the projection matrices $A \in \mathbb{R}^{m \times n}$.
Because both $m$ and $n$ depend on $d$, we choose $d$ as an \emph{oversampling parameter}, analogously to the \emph{undersampling parameter} $\rho = \frac{m}{n}$ used in \cite{DonohoT10}.

We analyze the influence of the image cosparsity, or equivalently of the image gradient sparsity, 
on the recovery via \eqref{eq:primal-LP} or \eqref{eq:M_known_cosupport}. We assess
empirical bounds in relation with the theoretically required sparsity that guarantees  exact recovery,
described as an empirical phase transition of $\rho$ depending on $d$. This \emph{phase transition} $\rho(d)$ indicates the necessary relative sparsity $\rho$ to recover a $\ell$-cosparse image with overwhelming probability by convex programming.

For each $d \in \{ 80, 90, \dots , 170, 180\}$ and for each relative sparsity $\rho$, we generated 70 images for the 2D case and 50 images for the 3D case, as illustrated in Section \ref{TestImages}, together with corresponding measurements
using the matrices from Section \ref{Sensors}.
This in turn gave us $d,n,m$ and $k$, defining a point $(d,\rho) \in \left[0,1 \right]^2$. This range was discretized into cells so as to accumulate in a $(d,\rho)$ cell a $1$ if the corresponding experiment was successful (exact recovery) and $0$ otherwise. In 2D, we performed 10 or 30 such runs for each $(d,\rho)$ pair, for unknown or known cosupport respectively. The success rate of image reconstruction is displayed by gray values: black $\leftrightarrow 0 \%$ recovery rate, white $\leftrightarrow 100\%$ recovery rate. 
In 3D, we analyzed the behavior for two image sizes, $d = 31$ and $d = 41$. The same reasoning as in the 2D case was applied, except that now instead of performing one test with 10 experiments, we ran 6 tests with 30 experiments each, in both cases of unknown and known cosupport. We show the mean value averaged over all 6 tests.

\subsubsection{Recovery of 2D Images}\label{sec:results2D}
The results are shown in Fig.~\ref{fig:res2D_unperturbed} and Fig.~\ref{fig:res2D_perturbed}. The empirical
transitions agree with the analytically derived thresholds up to a scaling factor
$\alpha$. The values of $\alpha$ are listed in Table
\ref{tab:2}. The accordingly rescaled curves are shown as dotted lines in the plots.

All plots display a phase transition and thus exhibit regions where exact image reconstruction has probability equal or close to one.

\begin{figure}[htbp]
\centerline{
\includegraphics[width=0.35\textwidth]{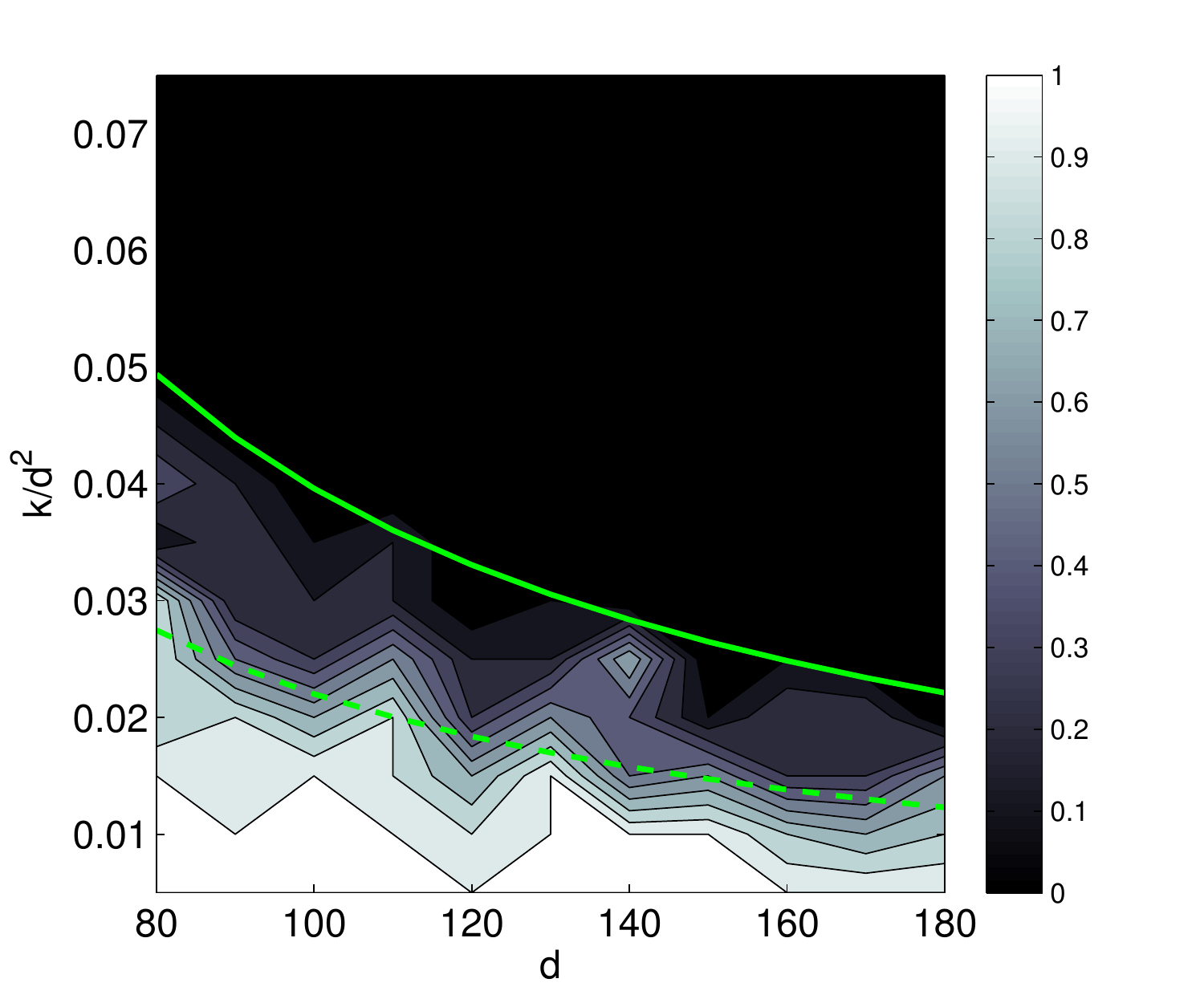} 
\includegraphics[width=0.35\textwidth]{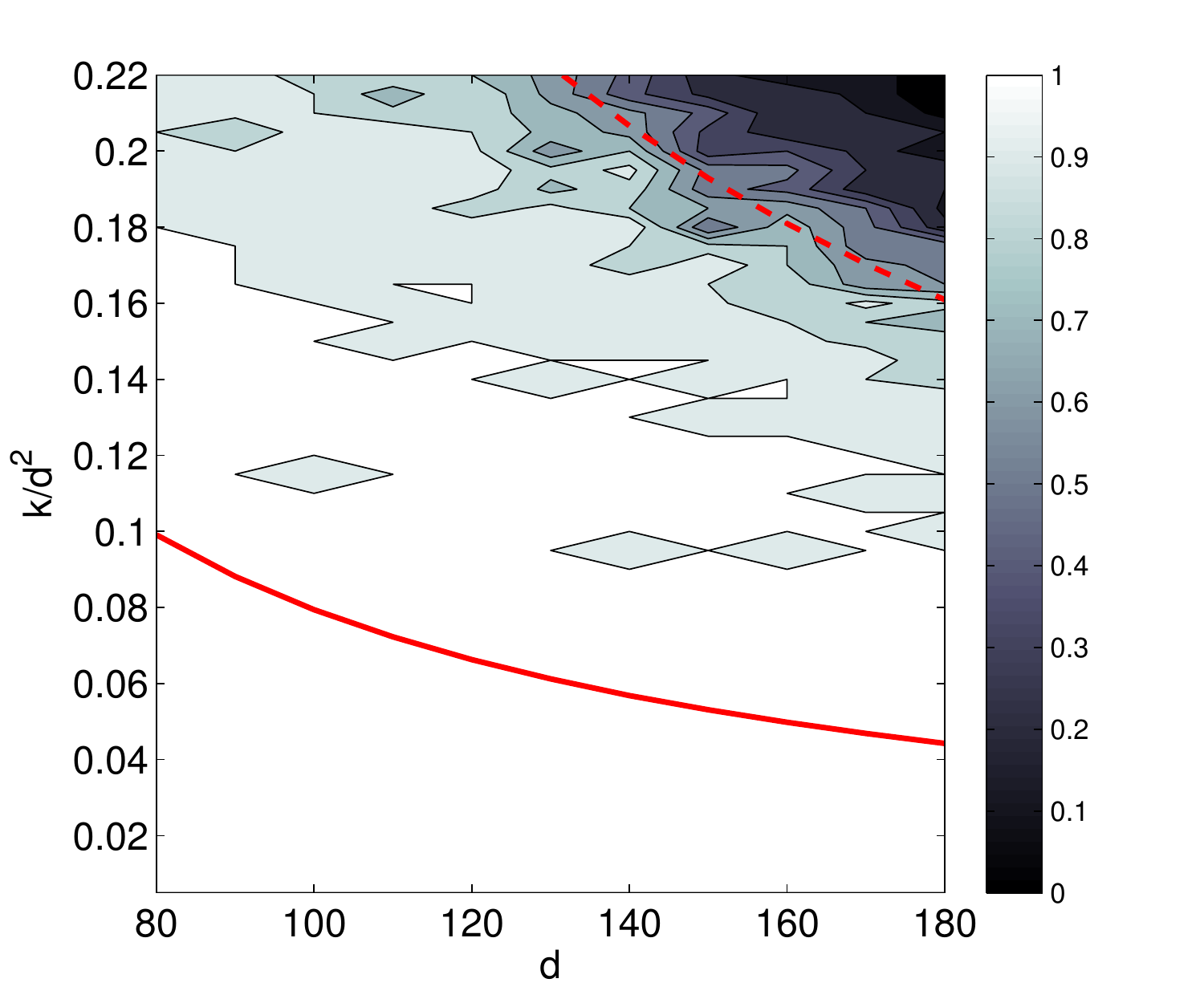} 
}
\vspace{-3mm}
\centerline{
\includegraphics[width=0.35\textwidth]{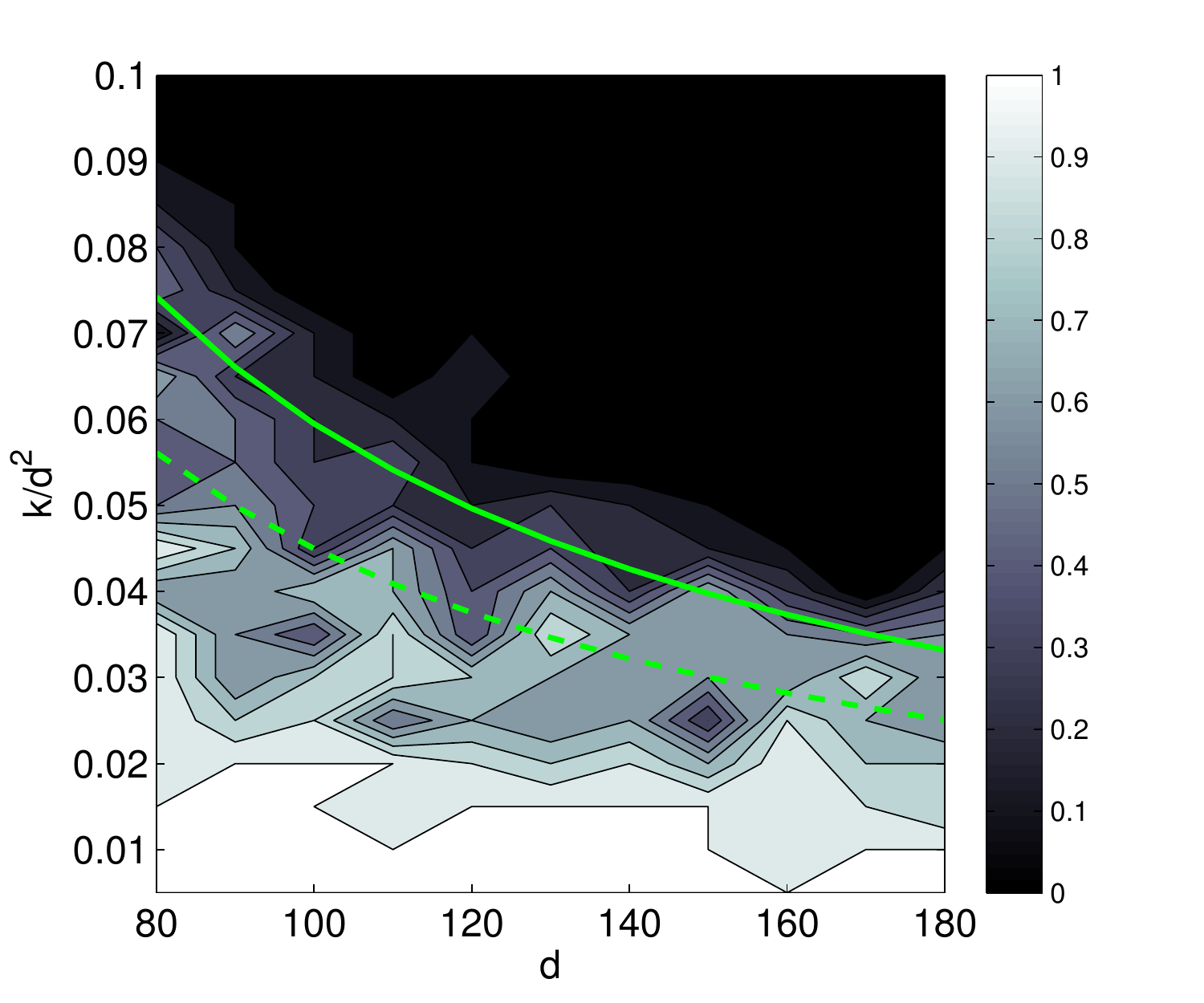} 
\includegraphics[width=0.35\textwidth]{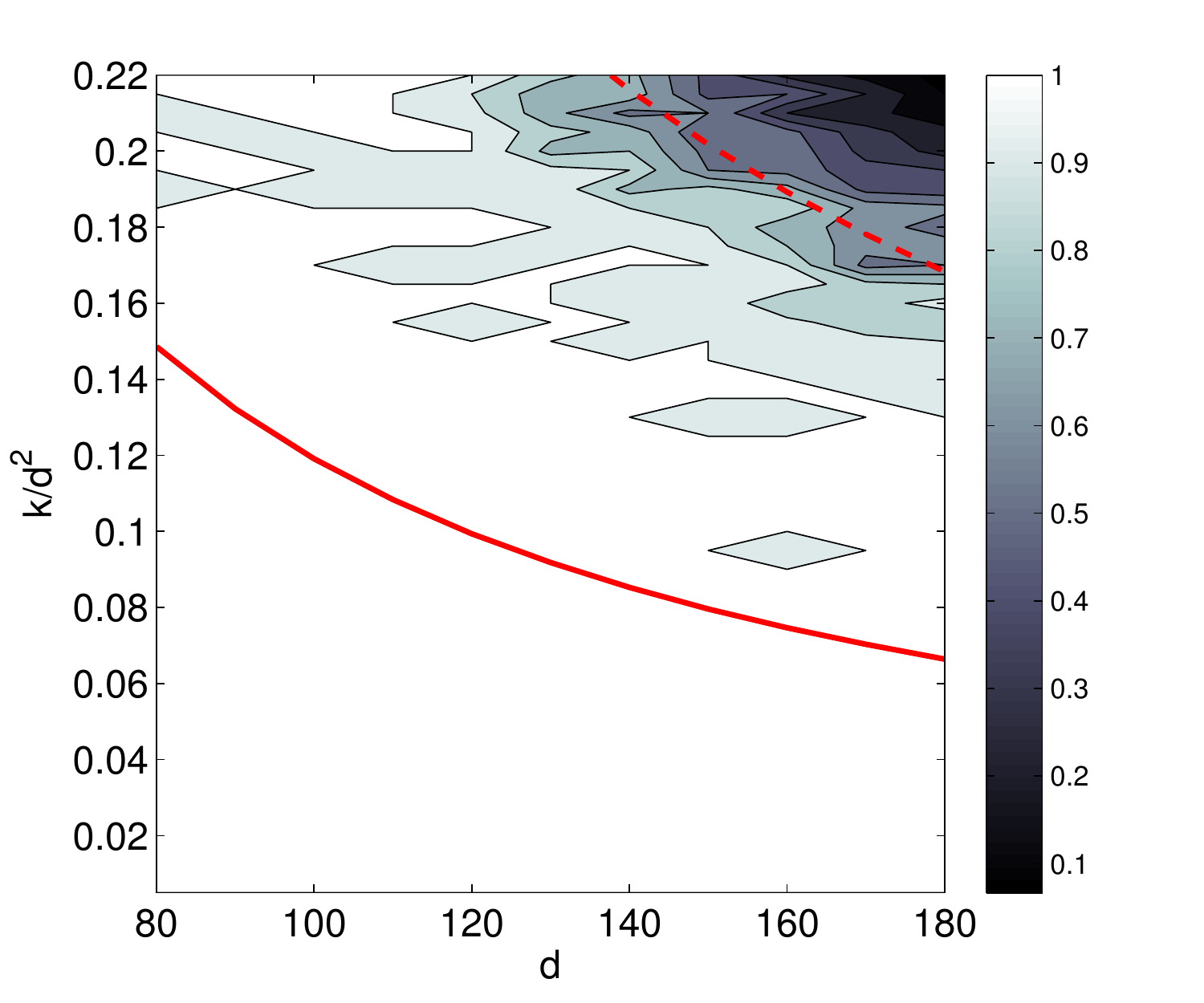} 
}
\vspace{-3mm}
\centerline{
\includegraphics[width=0.35\textwidth]{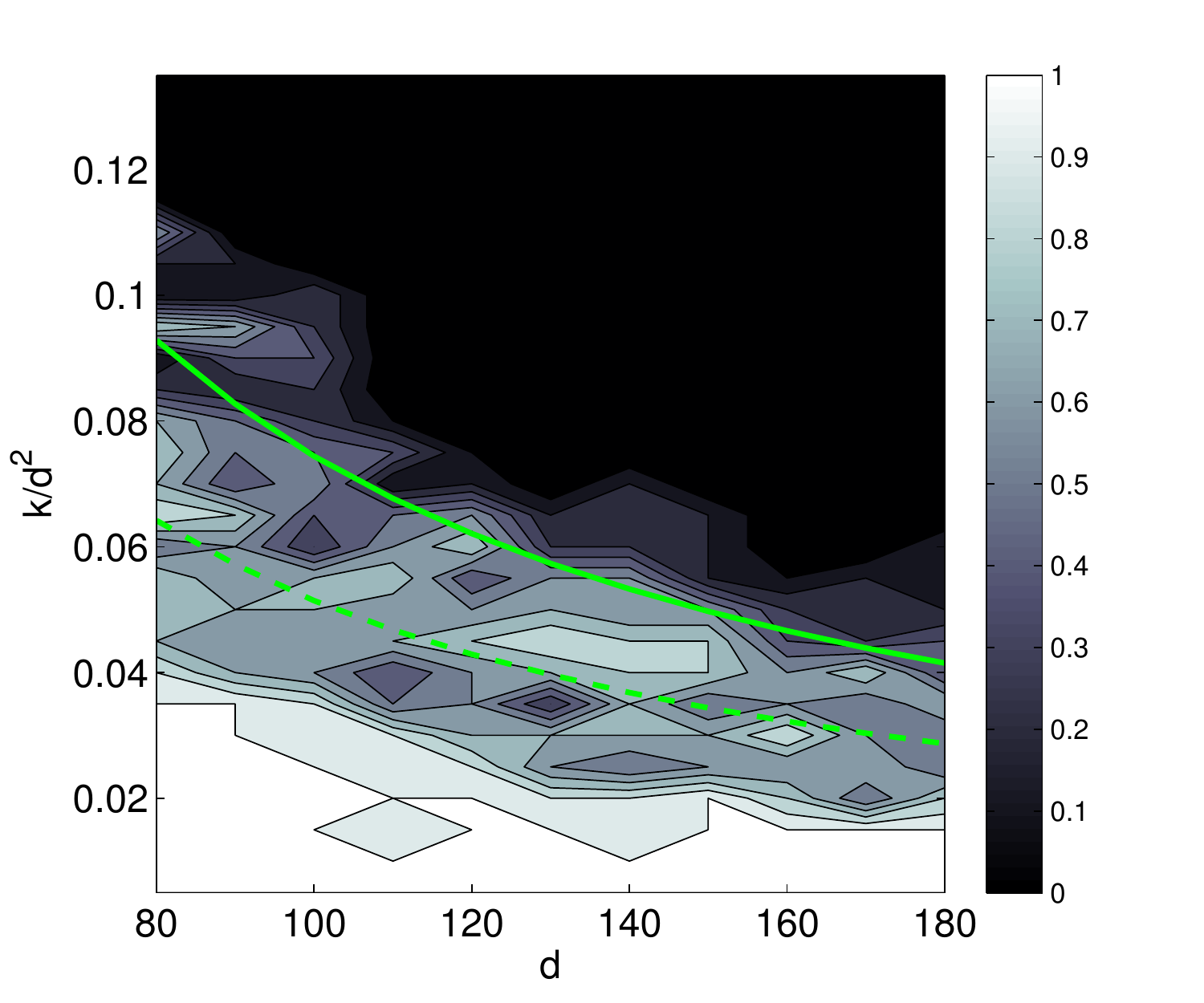} 
\includegraphics[width=0.35\textwidth]{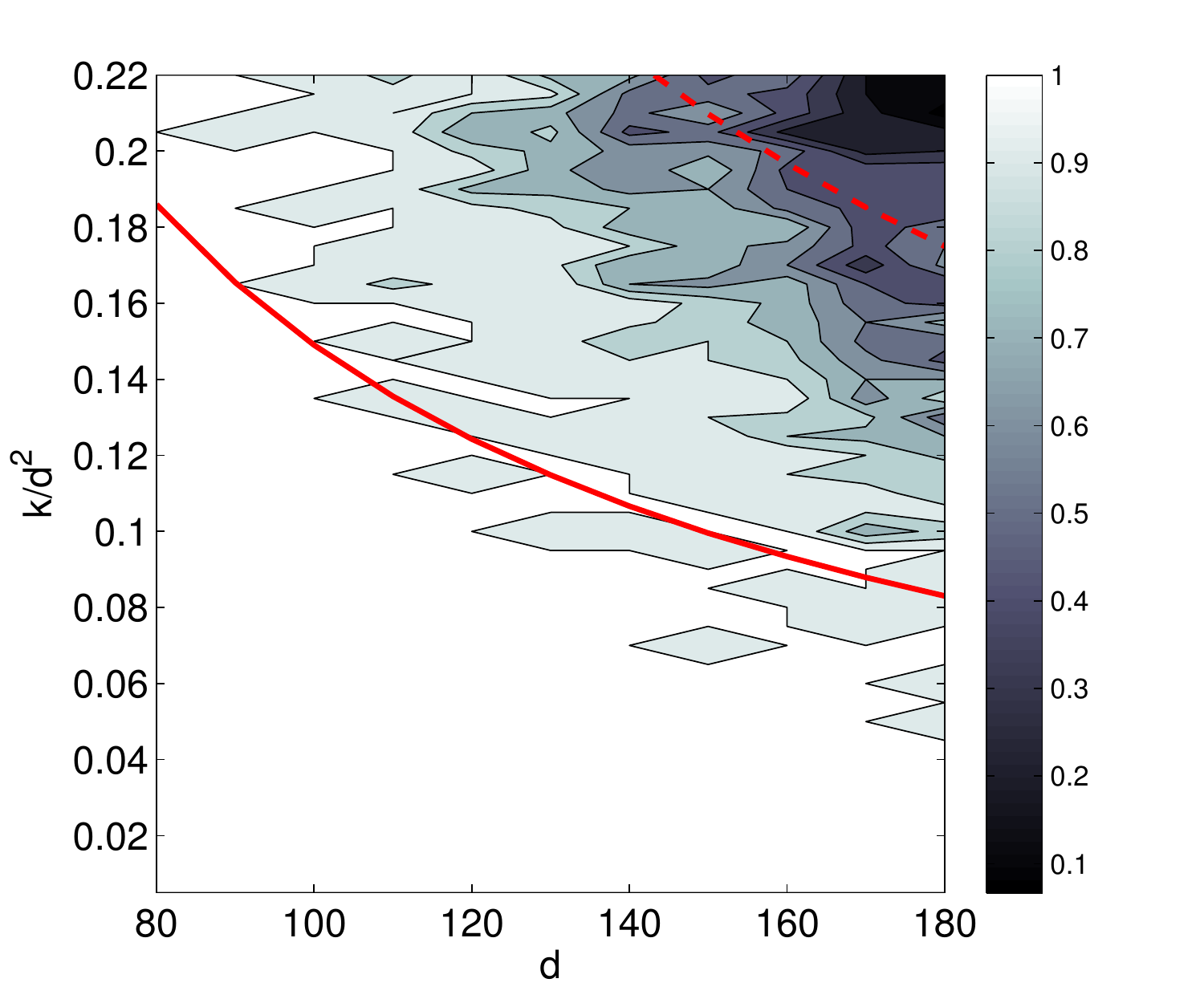} 
}
\vspace{-3mm}
\centerline{
\includegraphics[width=0.35\textwidth]{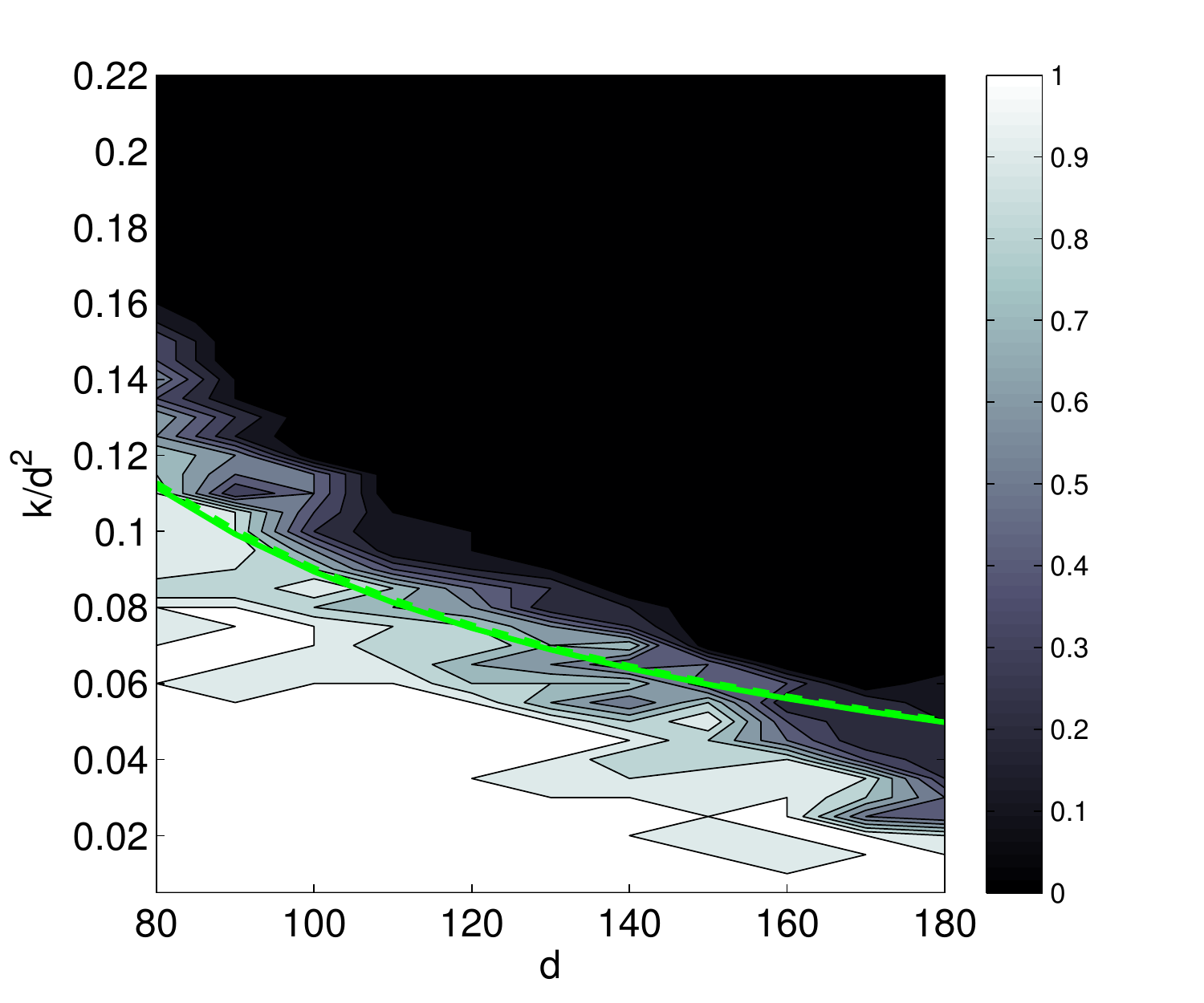} 
\includegraphics[width=0.35\textwidth]{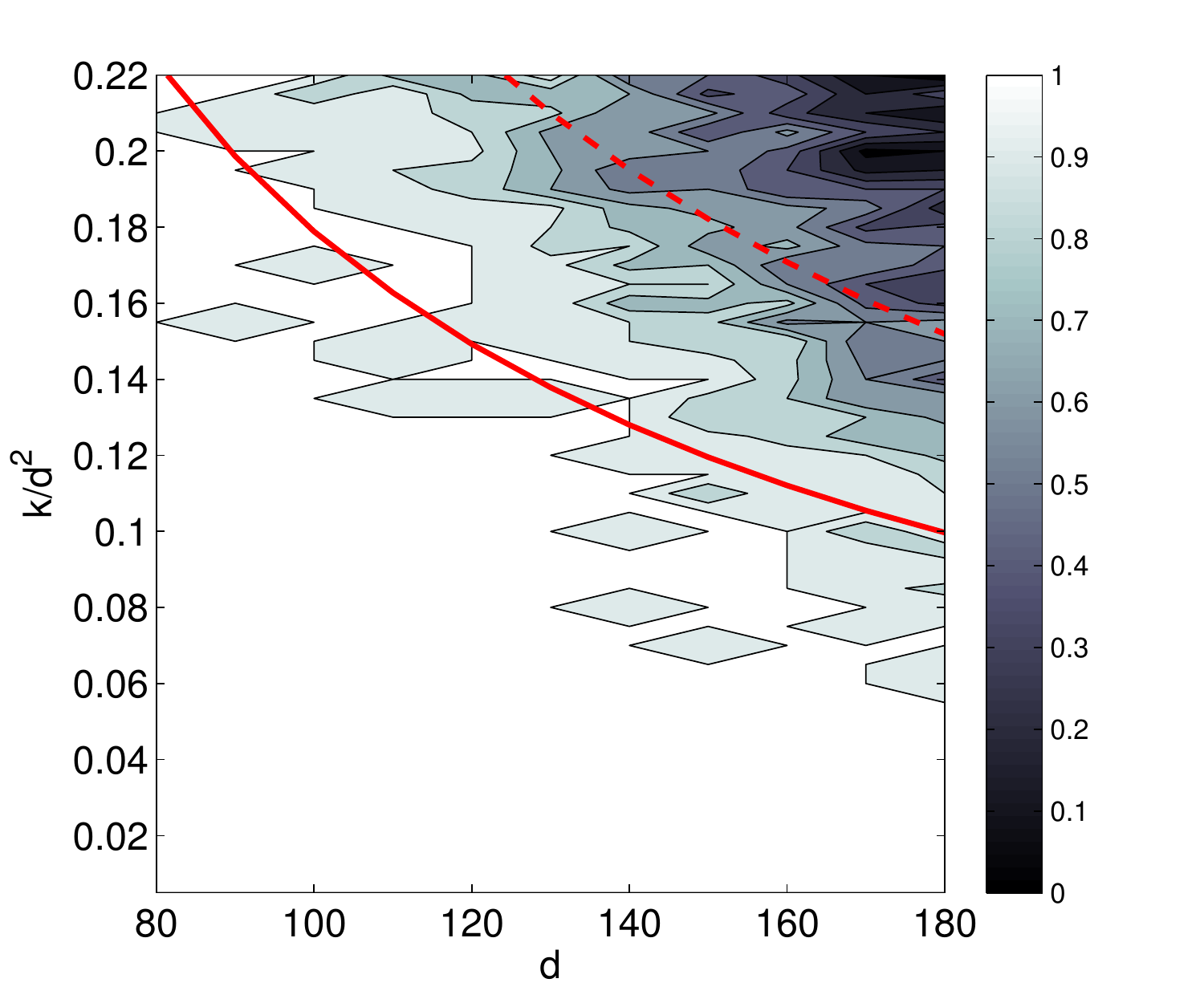} 
}
\caption{Phase transitions for the unperturbed matrix $A$, 2D case, 3, 4, 5 and 6 projecting directions (top to bottom), with unknown (left column) and known (right column) cosupport. The continuous green and red lines depict the theoretical curves \eqref{eq:m-lb-unknown-2D} and \eqref{eq:m-lb-known-2D} respectively. The dashed lines correspond to the empirical threshold, which are all scaled versions of \eqref{eq:m-lb-unknown-2D} (left column) or \eqref{eq:m-lb-known-2D} (right column)
with scaling factors $\alpha$ summarized in Table. \ref{tab:2}.}
\label{fig:res2D_unperturbed}
\end{figure}


\begin{figure}[htbp]
\centerline{
\includegraphics[width=0.35\textwidth]{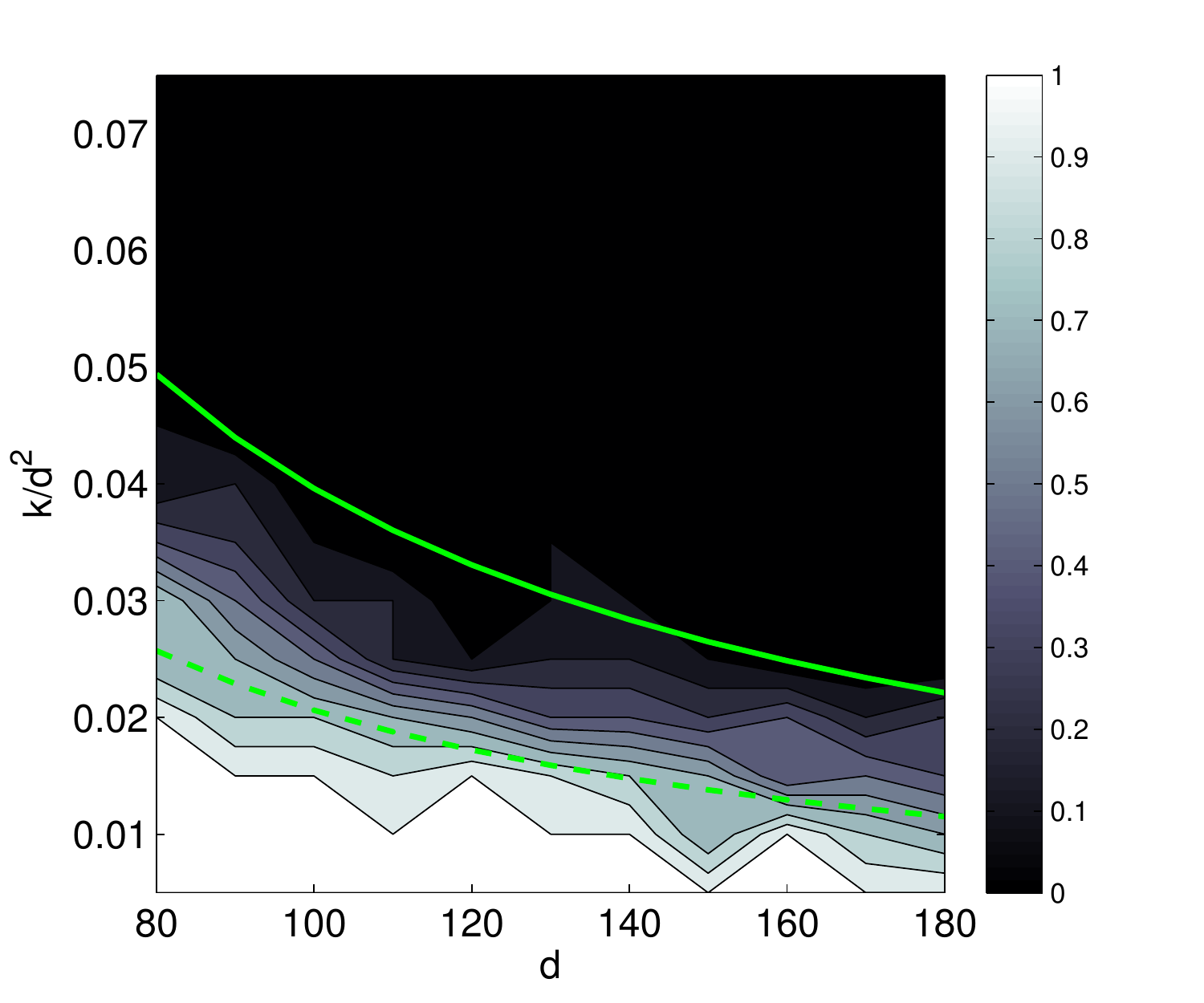} 
\includegraphics[width=0.35\textwidth]{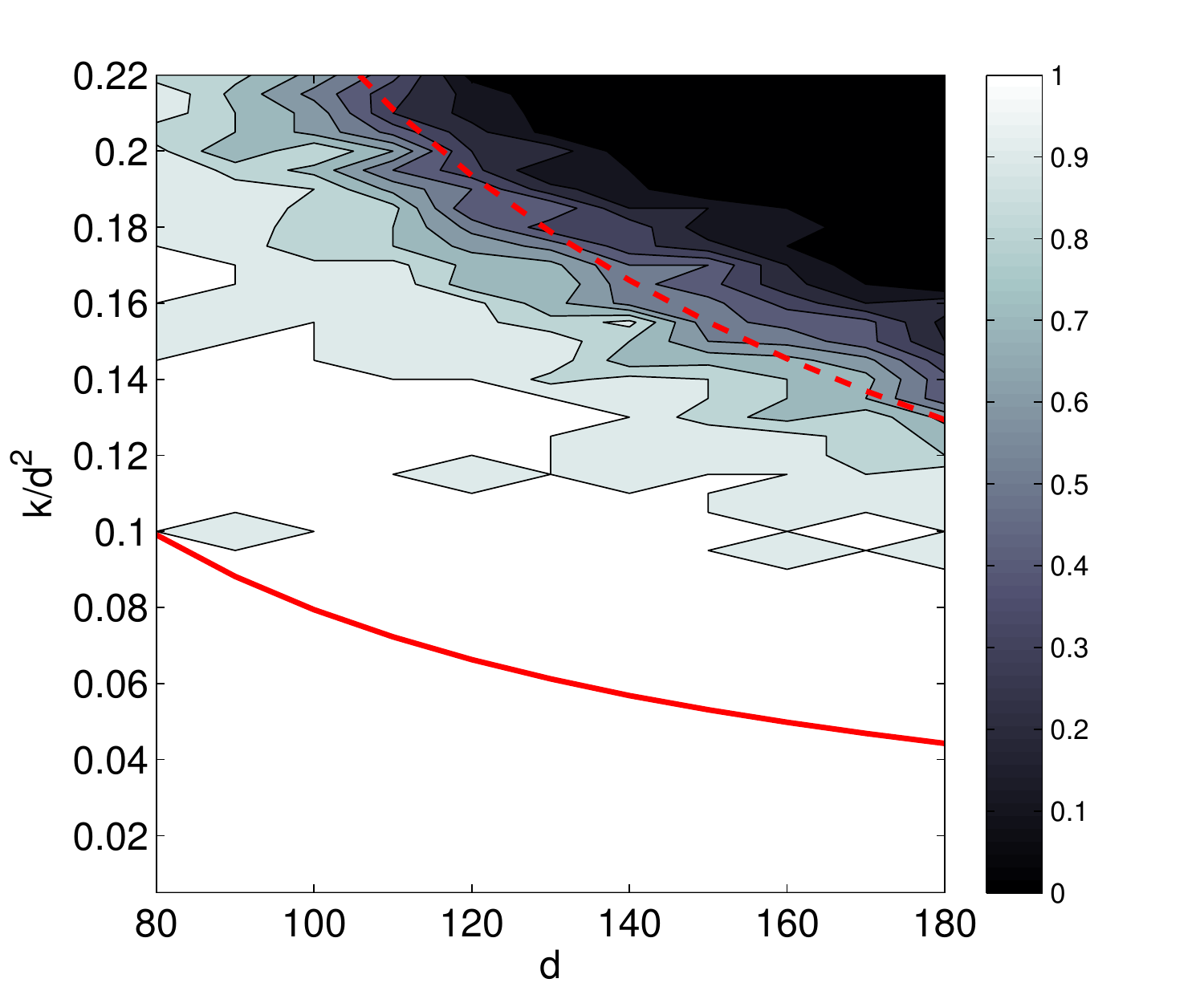} 
}
\vspace{-3mm}
\centerline{
\includegraphics[width=0.35\textwidth]{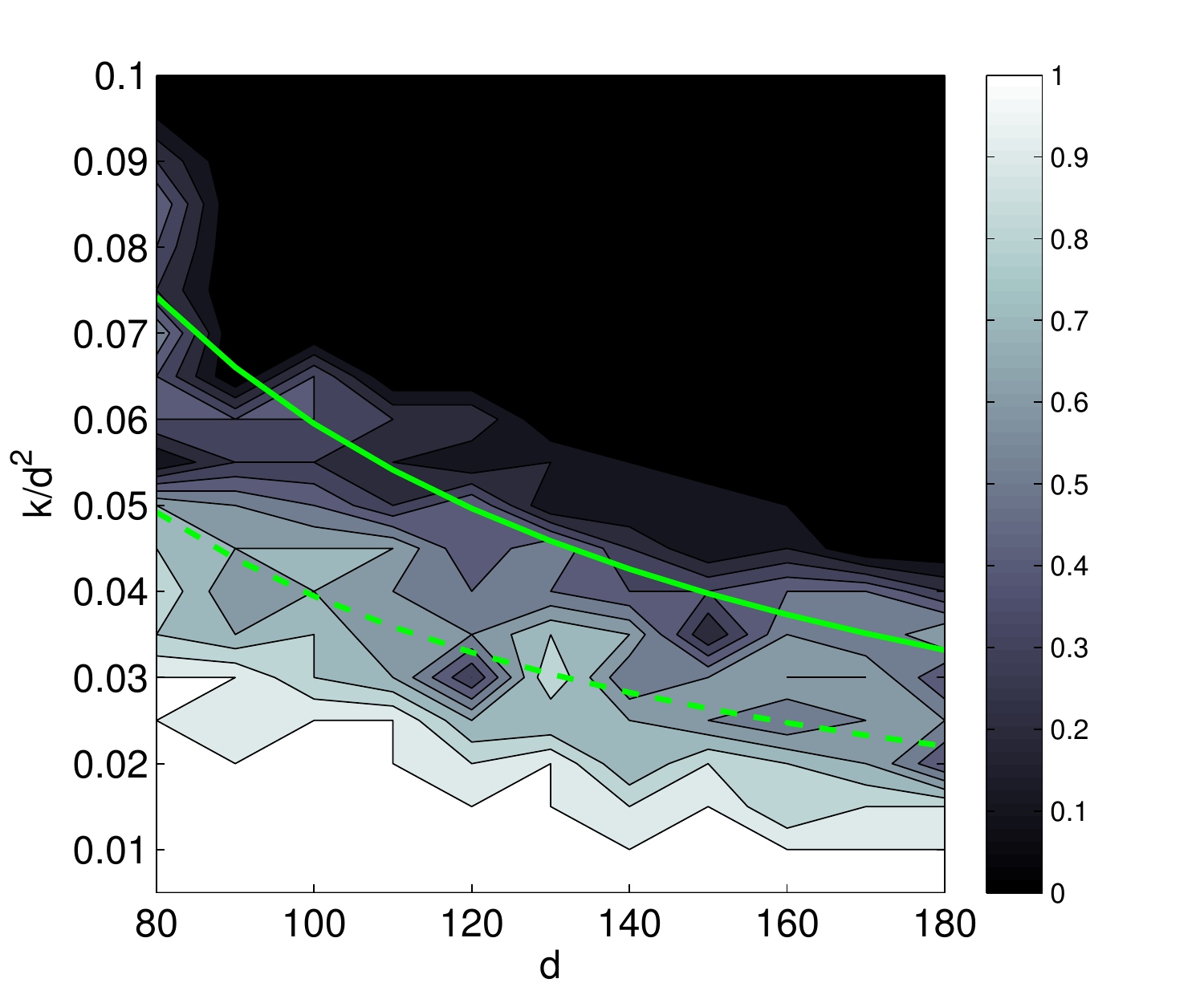} 
\includegraphics[width=0.35\textwidth]{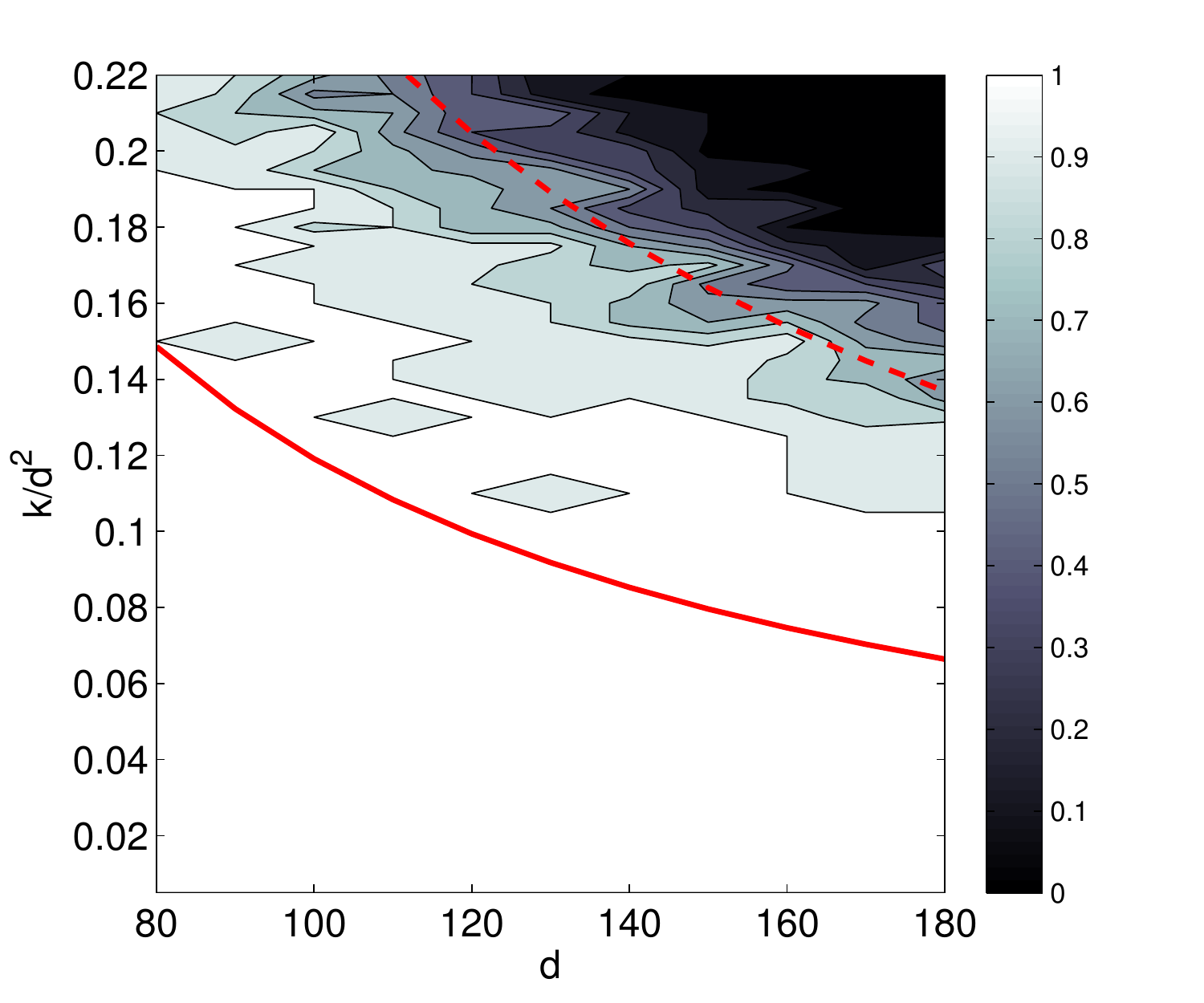} 
}
\vspace{-3mm}
\centerline{
\includegraphics[width=0.35\textwidth]{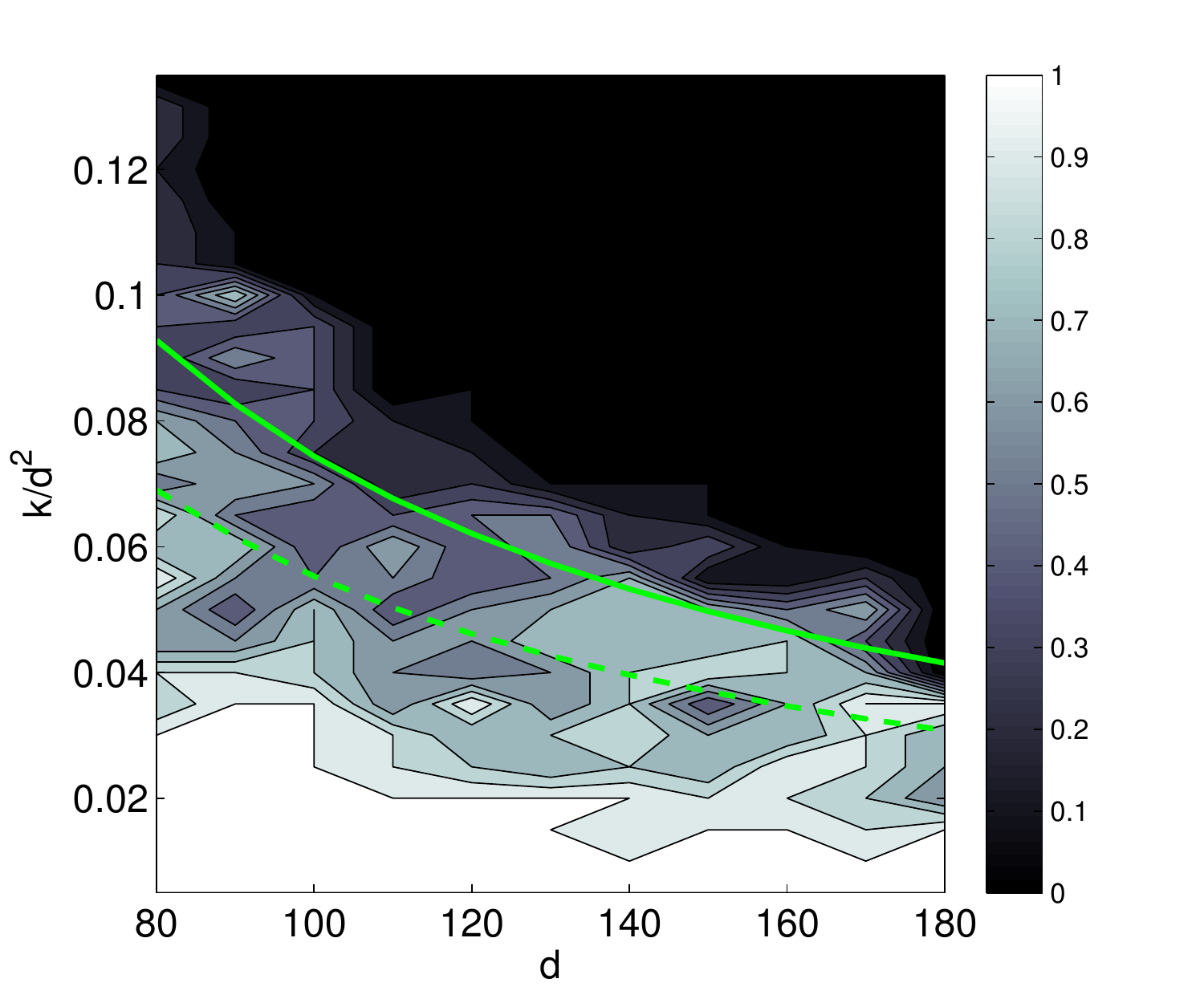} 
\includegraphics[width=0.35\textwidth]{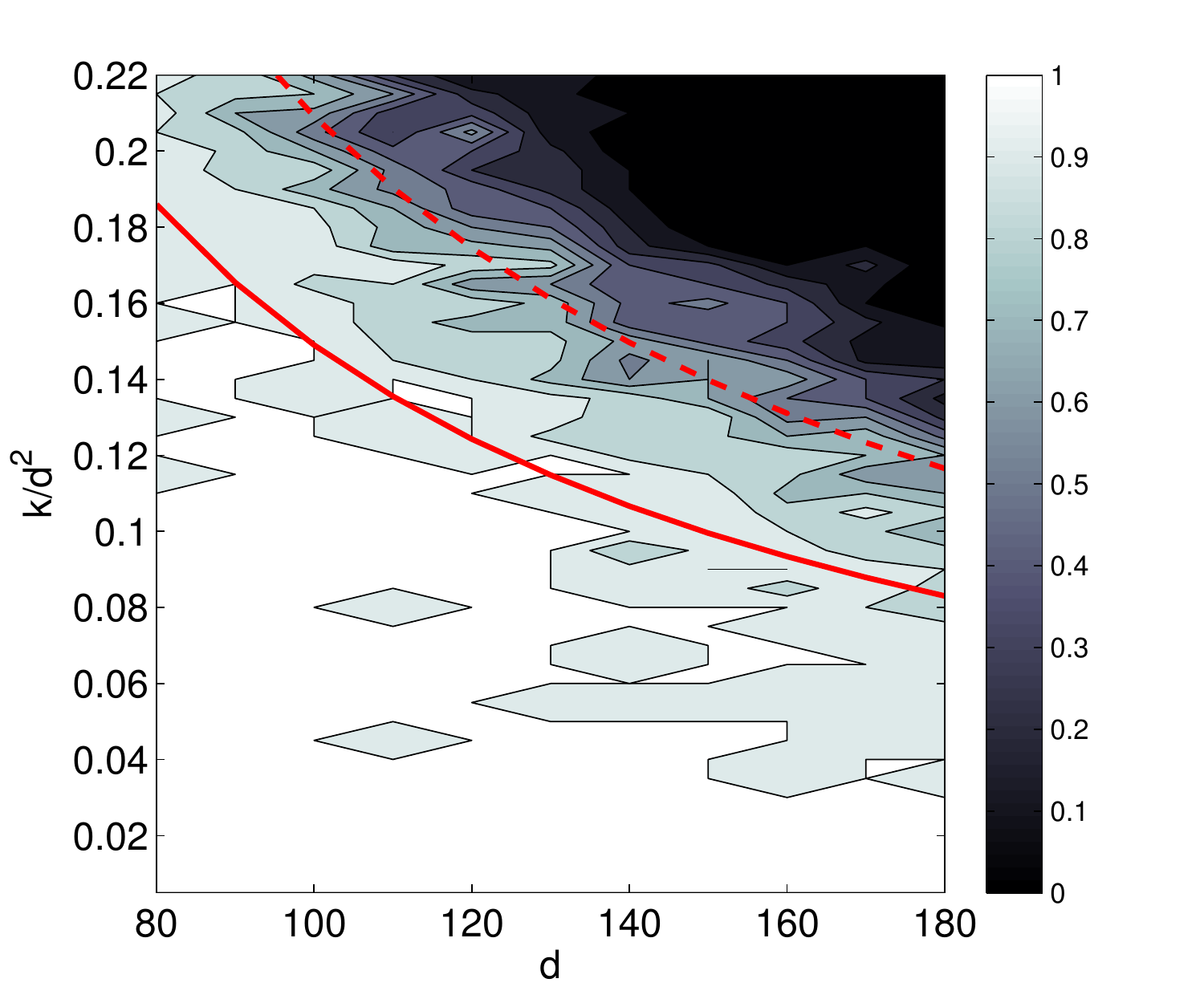} 
}
\vspace{-3mm}
\centerline{
\includegraphics[width=0.35\textwidth]{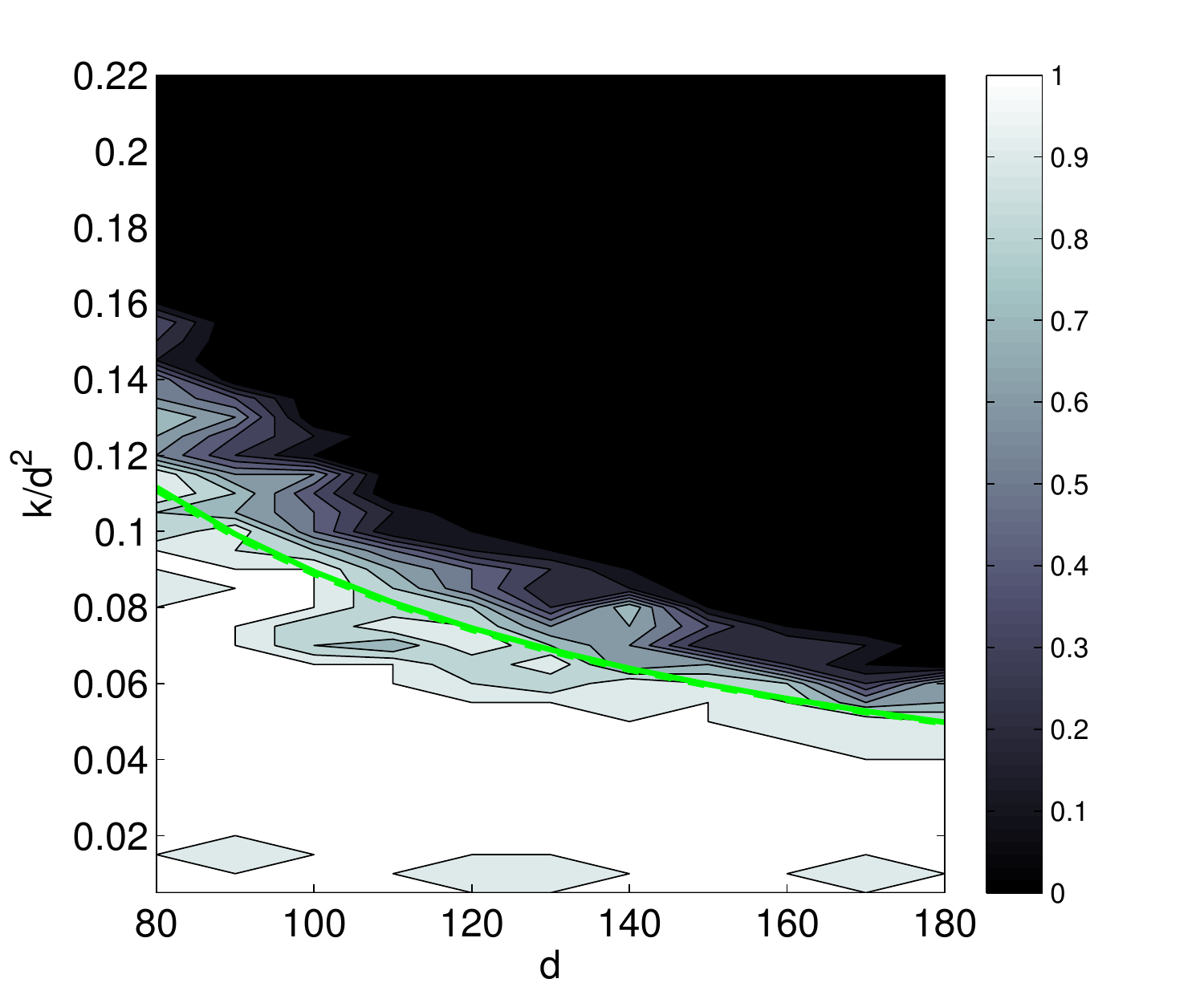} 
\includegraphics[width=0.35\textwidth]{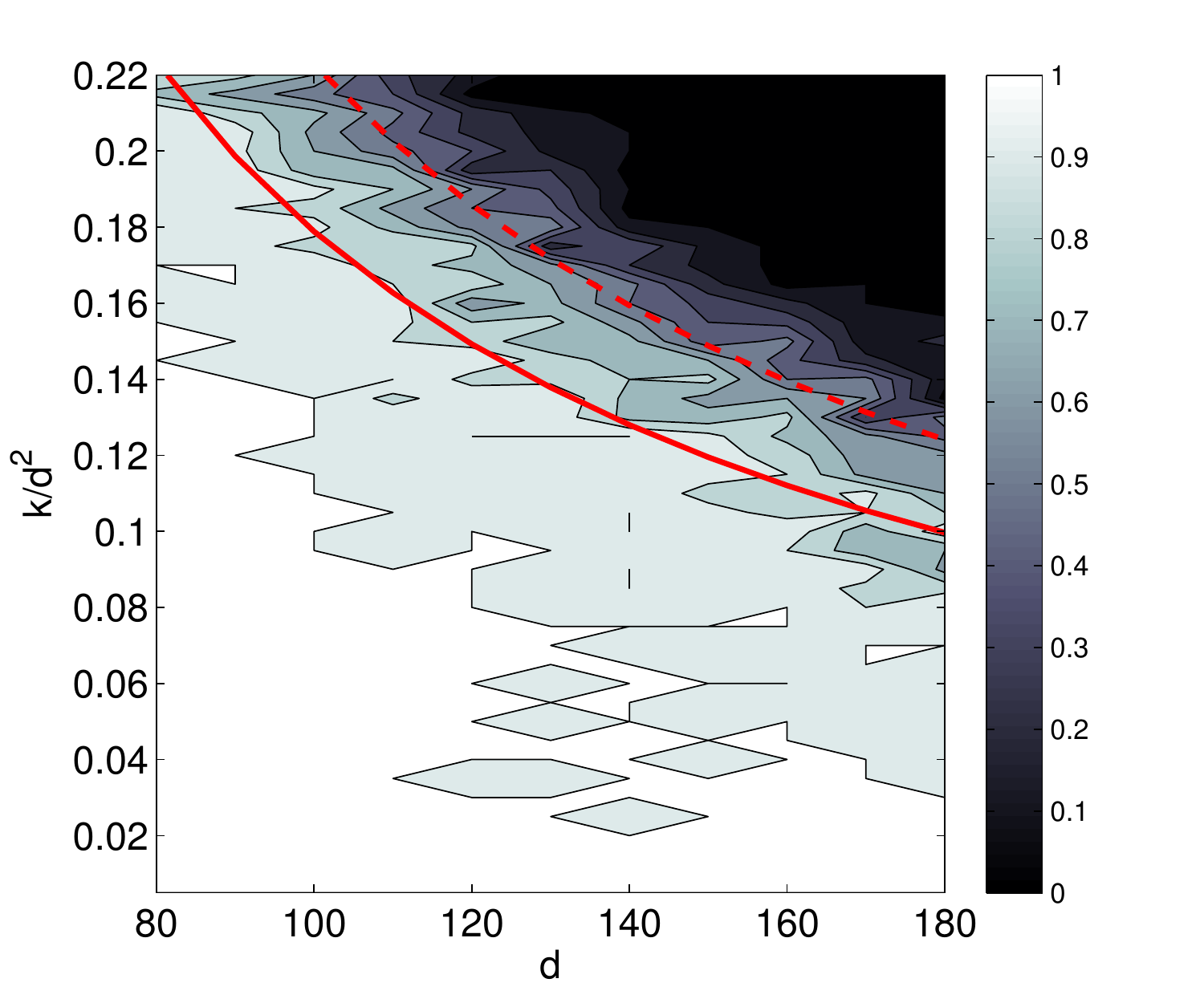} 
}
\caption{Phase transitions for the perturbed matrix $\tilde{A}$, 2D case, 3, 4, 5 and 6 projections (top to bottom), with unknown (left column) and known (right column) cosupport. The continuous green and red lines depict the theoretical curves \eqref{eq:m-lb-unknown-2D} and \eqref{eq:m-lb-known-2D} respectively. The dashed lines correspond to the empirical threshold, which are all scaled versions of
 \eqref{eq:m-lb-unknown-2D} (left column) or \eqref{eq:m-lb-known-2D} (right column)
with scaling factors $\alpha$ listed in Table. \ref{tab:2}.
}
\label{fig:res2D_perturbed}
\end{figure}


\begin{table} [hbtp]

	\begin{tabular}{|c|c|c|c|c|c|c|}
		\hline
		\multicolumn{7}{|c|}{$\alpha$-values in 2D} \\
		\hline
			$d$ & Cosupport & Measurements & $3P$ & $4P$ & $5P$ & $6P$  \\
	    \hline
 \multirow {4}{*}{80 \dots 180} & \multirow{2}{*}{Known}  & unperturbed  & 3.6358 & 2.5349 & 2.1073 & 1.5241\\ \cline{3-7}
       				  		    &    				      & perturbed  	 & 2.9220 & 2.0614 & 1.4039 & 1.2453 \\ \cline{2-7}
  	   					 		& \multirow{2}{*}{Unknown}& unperturbed  & 0.5560 & 0.6912 & 0.7556 & 1.0104\\ \cline{3-7}
       				     		& 					      & perturbed    & 0.5208 & 0.6630 & 0.7435 & 0.9926\\ \cline{1-7}                  
	    \hline
	\end{tabular}
	
\vspace{2mm}
\caption{The scaling factors of the theoretical curves \eqref{eq:m-lb-known-2D}
and \eqref{eq:m-lb-unknown-2D} for known and unknown cosupport respectively. }
\label{tab:2}
\end{table}


\subsubsection{Recovery of 3D Images}\label{sec:results3D}

The results are shown in Fig.~\ref{fig:res3D_d31} and Fig.~\ref{fig:res3D_d41} for $d=31$ and $d=41$, 
and summarized in Fig. \ref{fig:3D_phase_trans}.
The empirical phase transitions differ again from the analytically derived thresholds \ref{eq:m-lb-known-3D} and
\ref{eq:m-lb-unknown-3D} only by a scaling factor $\alpha$. These values are listed as Table
\ref{tab:22}. The rescaled curves are shown as dotted lines in the plots.
Fig. \ref{fig:3D_phase_trans} also relates the critical sparsity of the gradient to the critical sparsity estimated in
\cite{Petra-Schnoerr-LAA,Petra2013b},
which in turn implies exact recovery via \eqref{eq:l1min}.


\begin{figure}[htbp]
\centerline{
\includegraphics[width=0.4\textwidth]{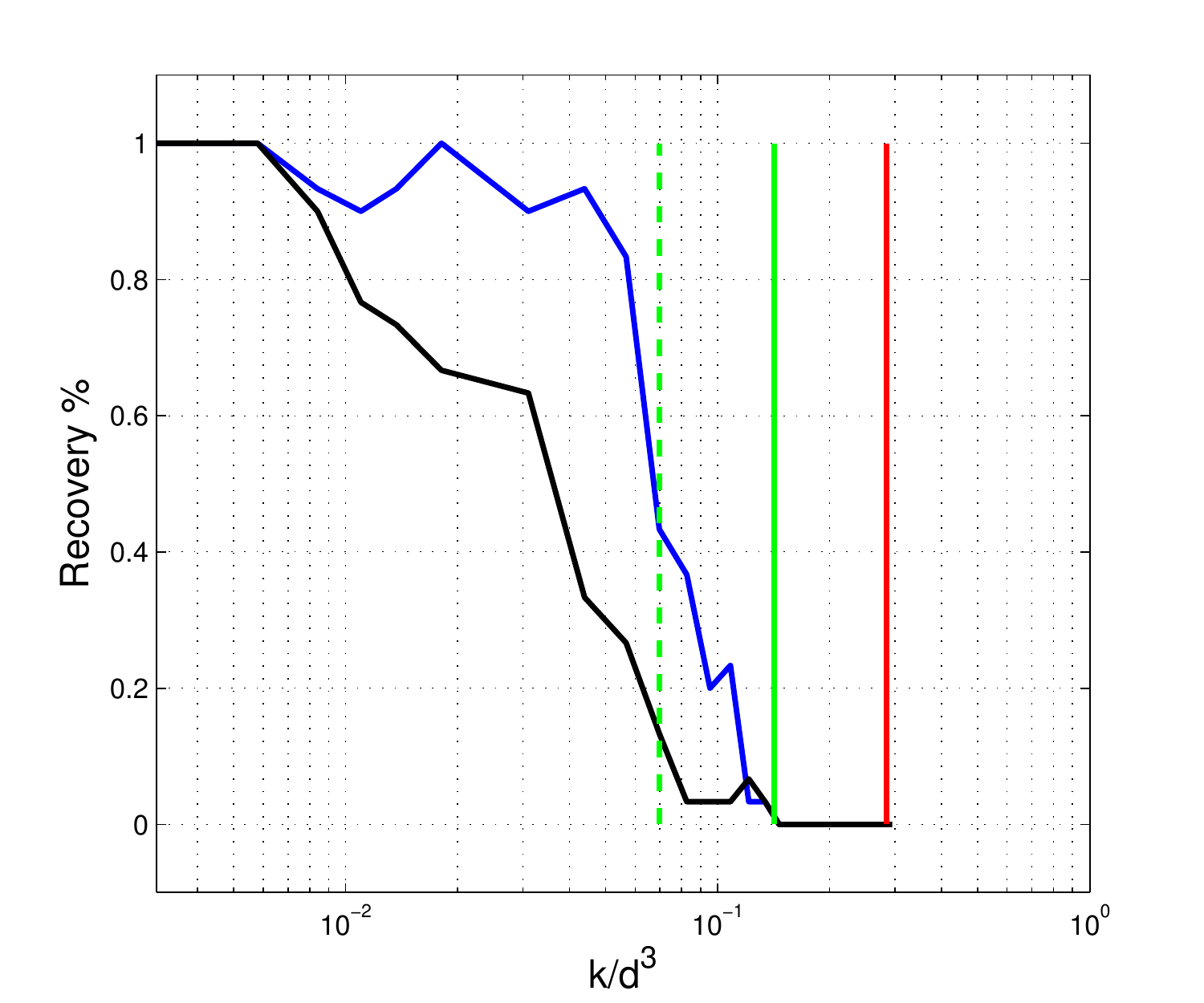} 
\includegraphics[width=0.4\textwidth]{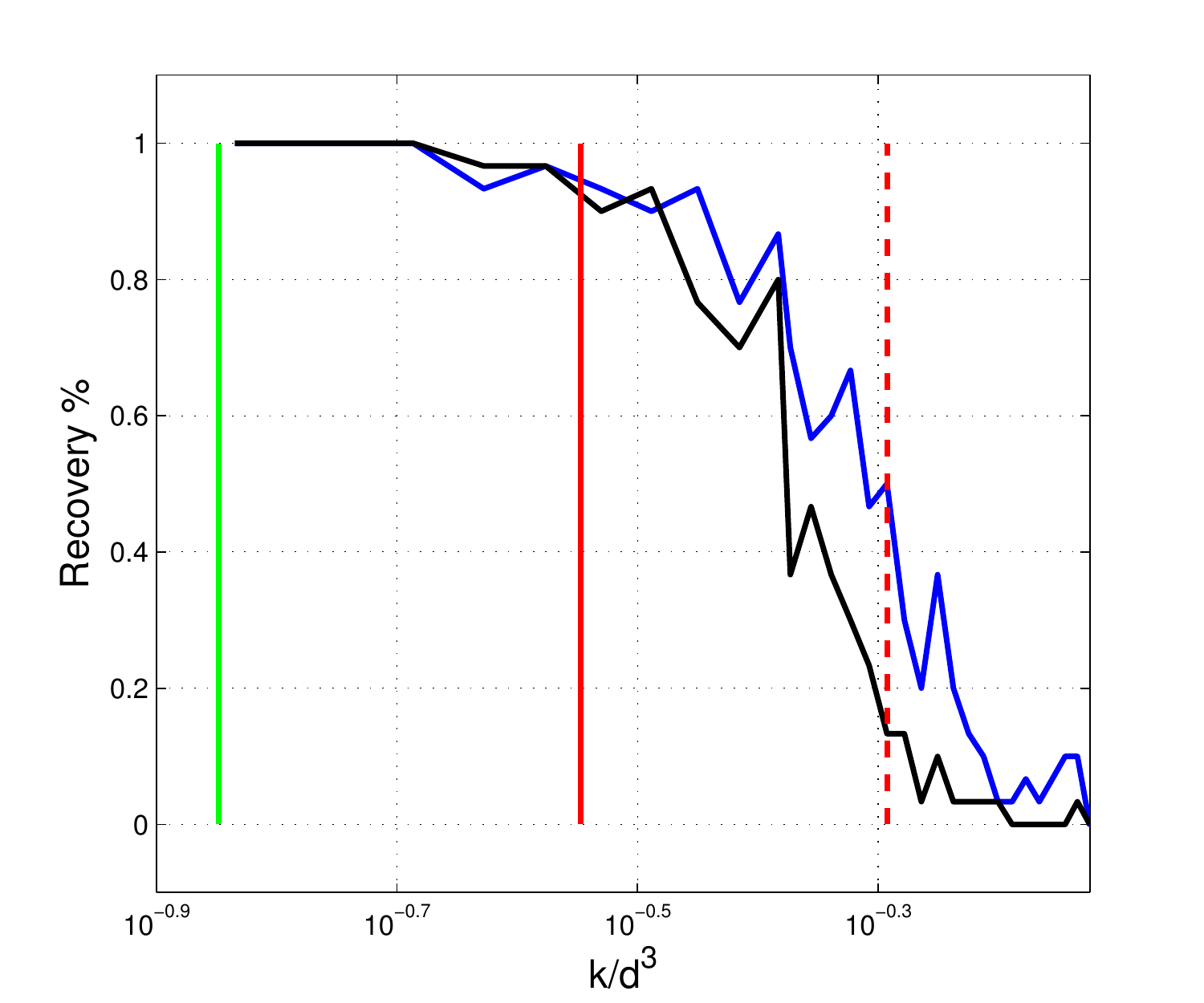}
}
\centerline{
\includegraphics[width=0.4\textwidth]{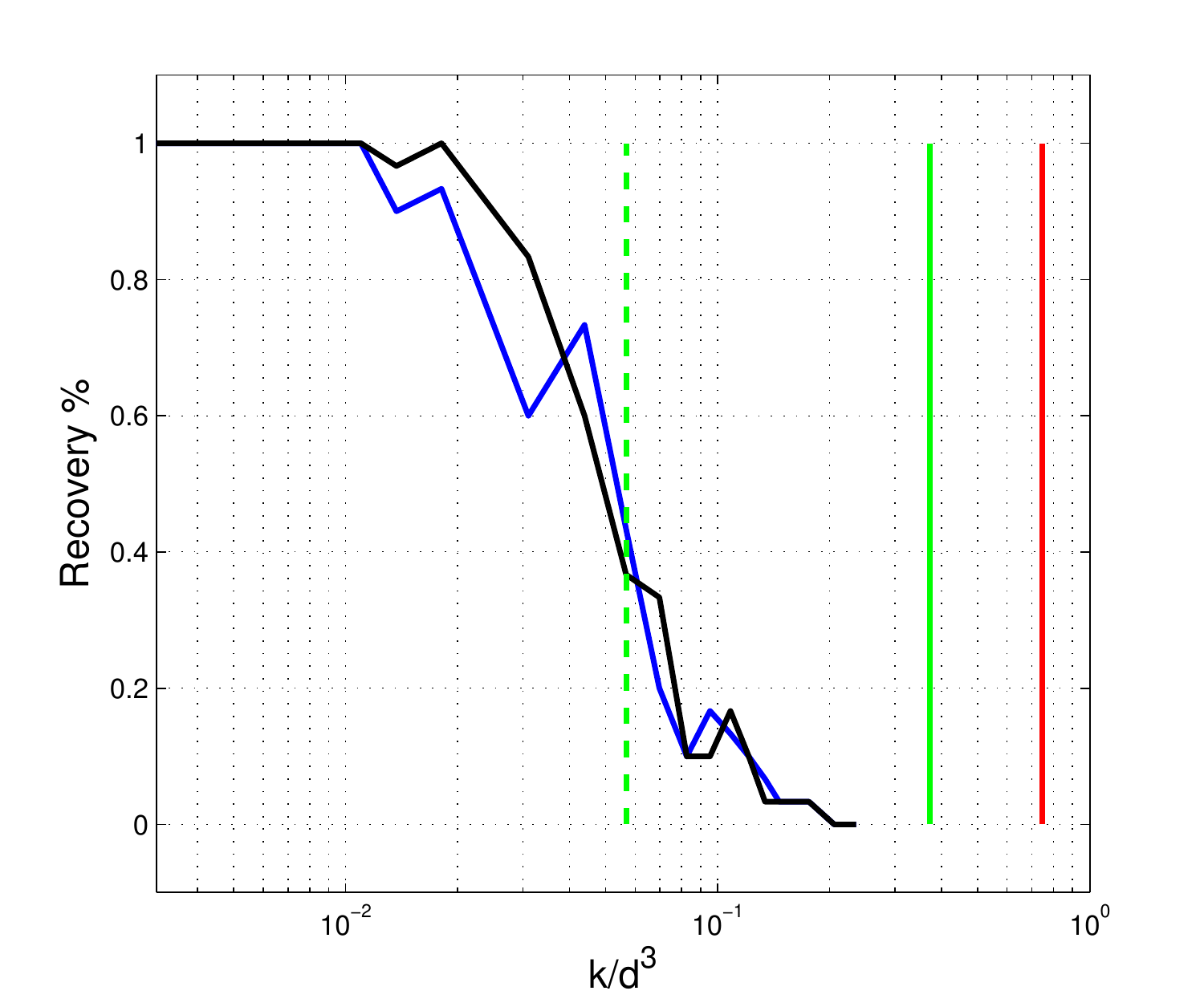}
\includegraphics[width=0.4\textwidth]{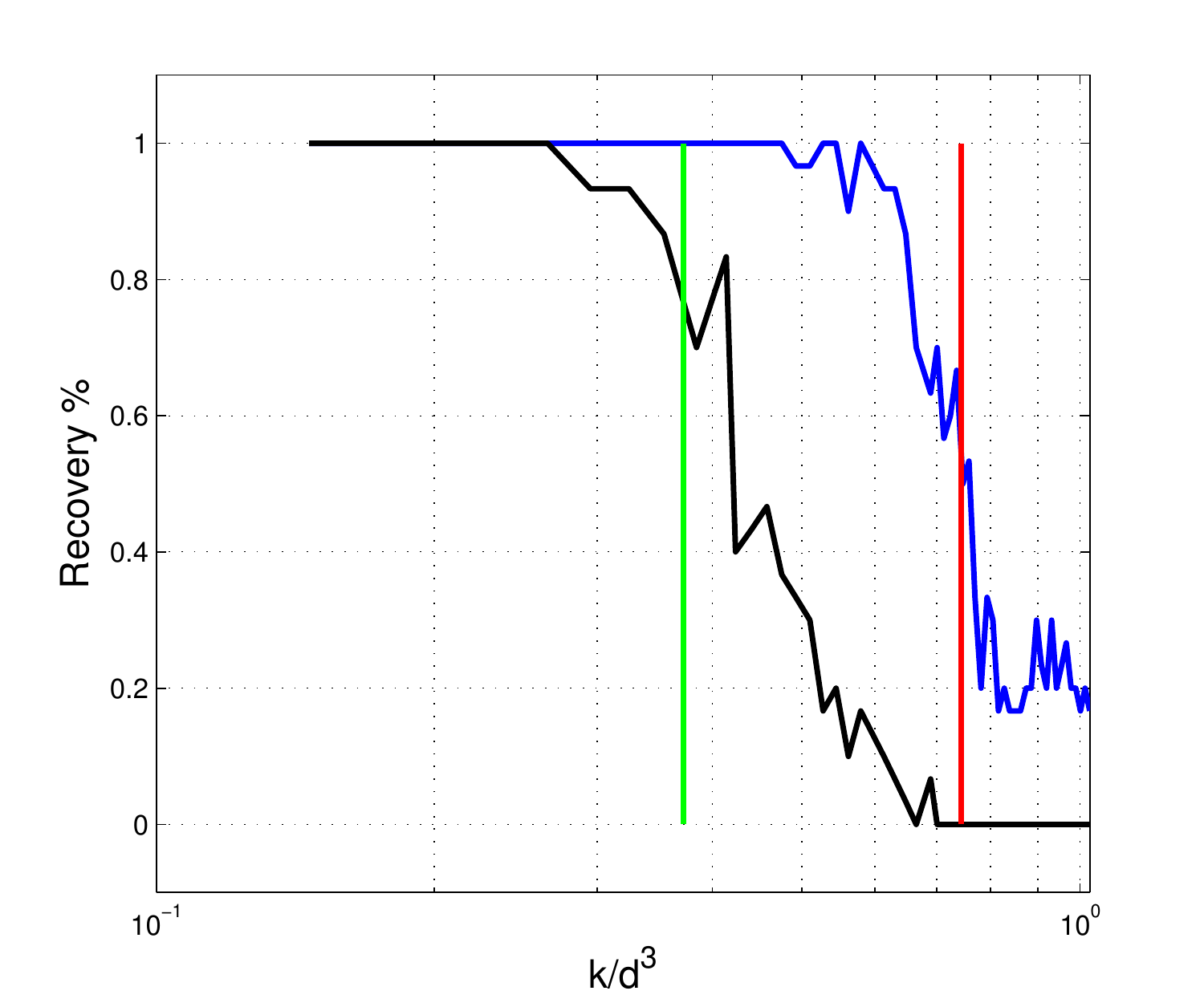}
}
\caption{Empirical probability ($6\times 30$ trials) of exact recovery by total variation minimization \eqref{eq:TVmin_pos}  via the unperturbed (blue line) and perturbed (black line) matrices from Section \ref{sec:setup_3D}, 
Fig. \ref{fig:3C3D} and Fig. \ref{fig:4C3D}, 
for 3 (top row) and 4 (bottom row) projecting directions, respectively, and $d = 31$. 
The left column shows the decay of the recovery probability when the cosupport is \emph{unknown} using both perturbed and unperturbed
projecting matrices, while the right one, shows results for \emph{known} cosupport.
The continuous vertical lines stand for the theoretical thresholds for known \eqref{eq:m-lb-known-3D} (red) and  
unknown \eqref{eq:m-lb-unknown-3D} (green) cosupport,
while the dotted red and green vertical lines stand for the empirically estimated threshold for known and unknown cosupport but for unperturbed matrices only. The deviation of the empirical thresholds from the theoretical curves for known cosupport \eqref{eq:m-lb-known-3D} and 
unknown cosupport \eqref{eq:m-lb-unknown-3D}
was estimated through least-squares fit and is summarized in Table \ref{tab:22}, along with results for the perturbed matrices.}
\label{fig:res3D_d31}
\end{figure}


\begin{figure} [htbp]
\centerline{
\includegraphics[width=0.4\textwidth]{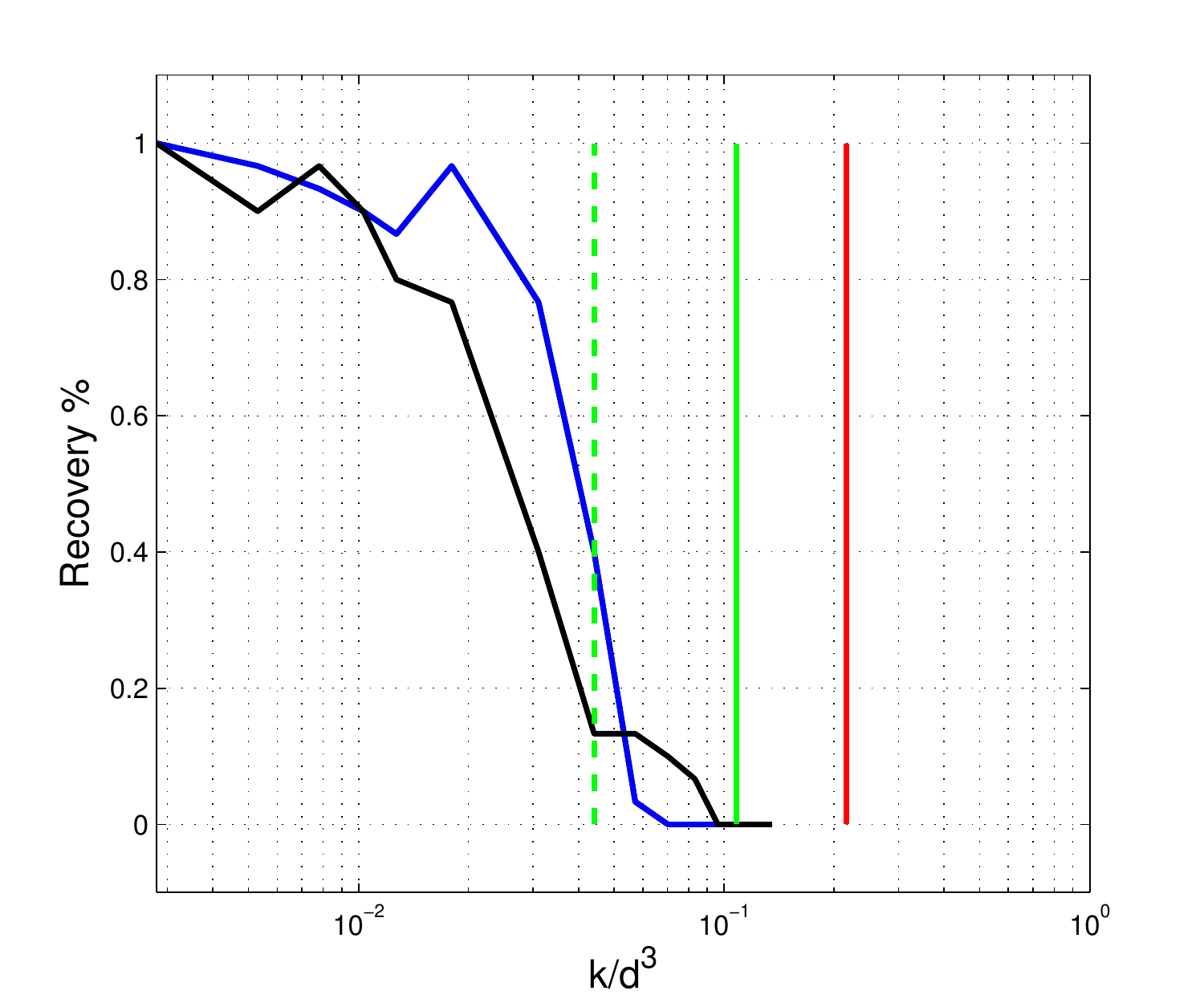} 
\includegraphics[width=0.4\textwidth]{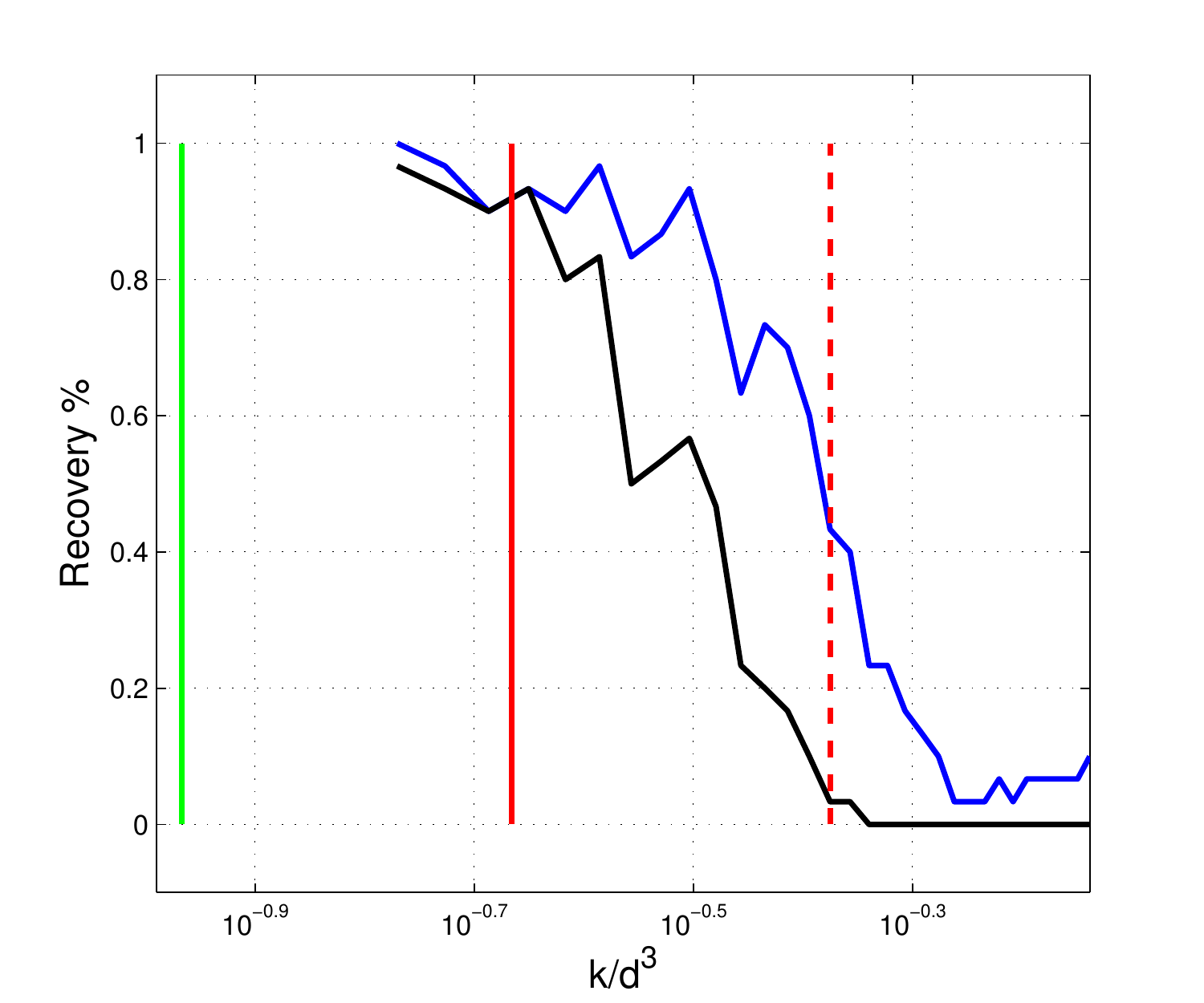}
}
\centerline{
\includegraphics[width=0.4\textwidth]{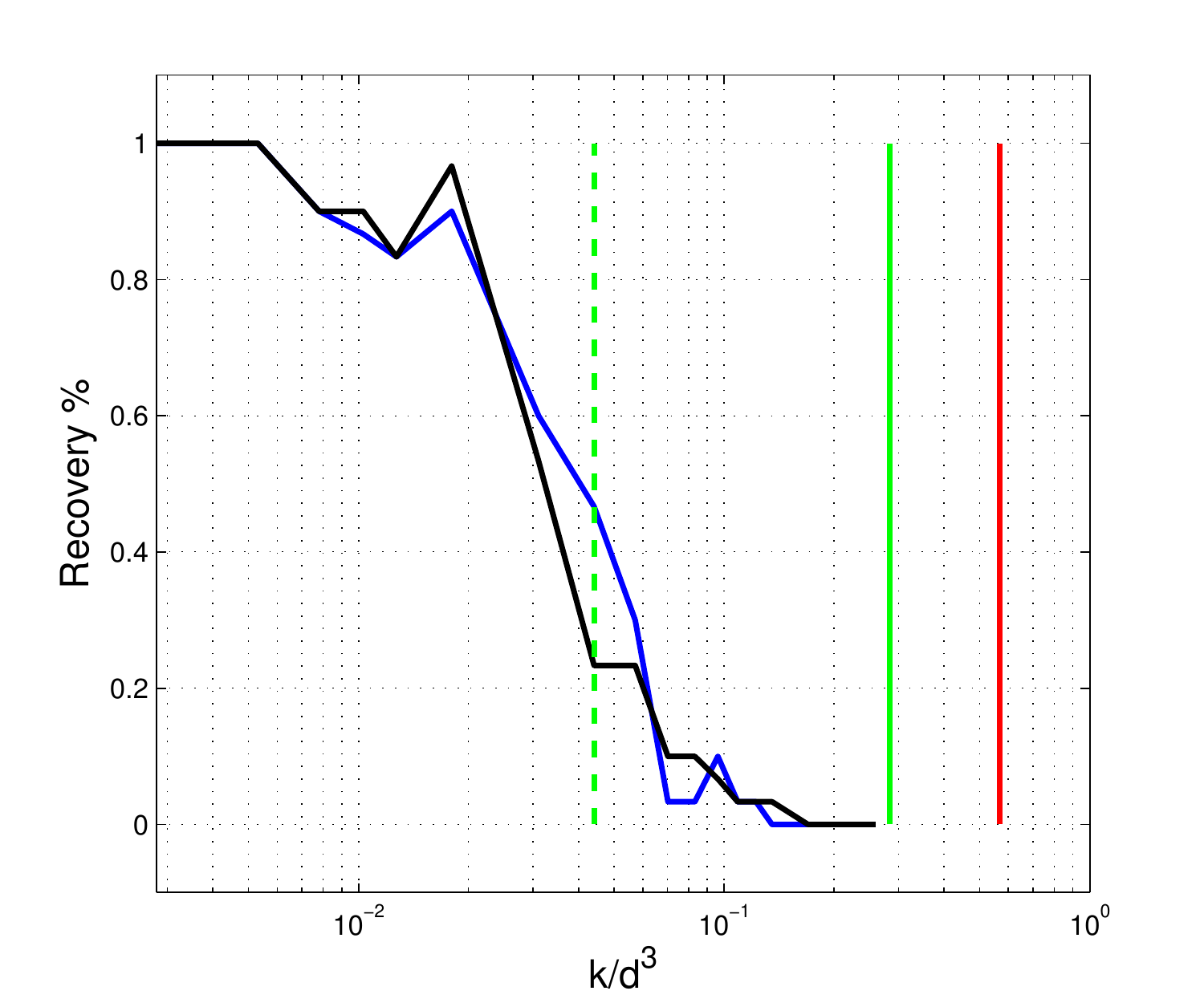}
\includegraphics[width=0.4\textwidth]{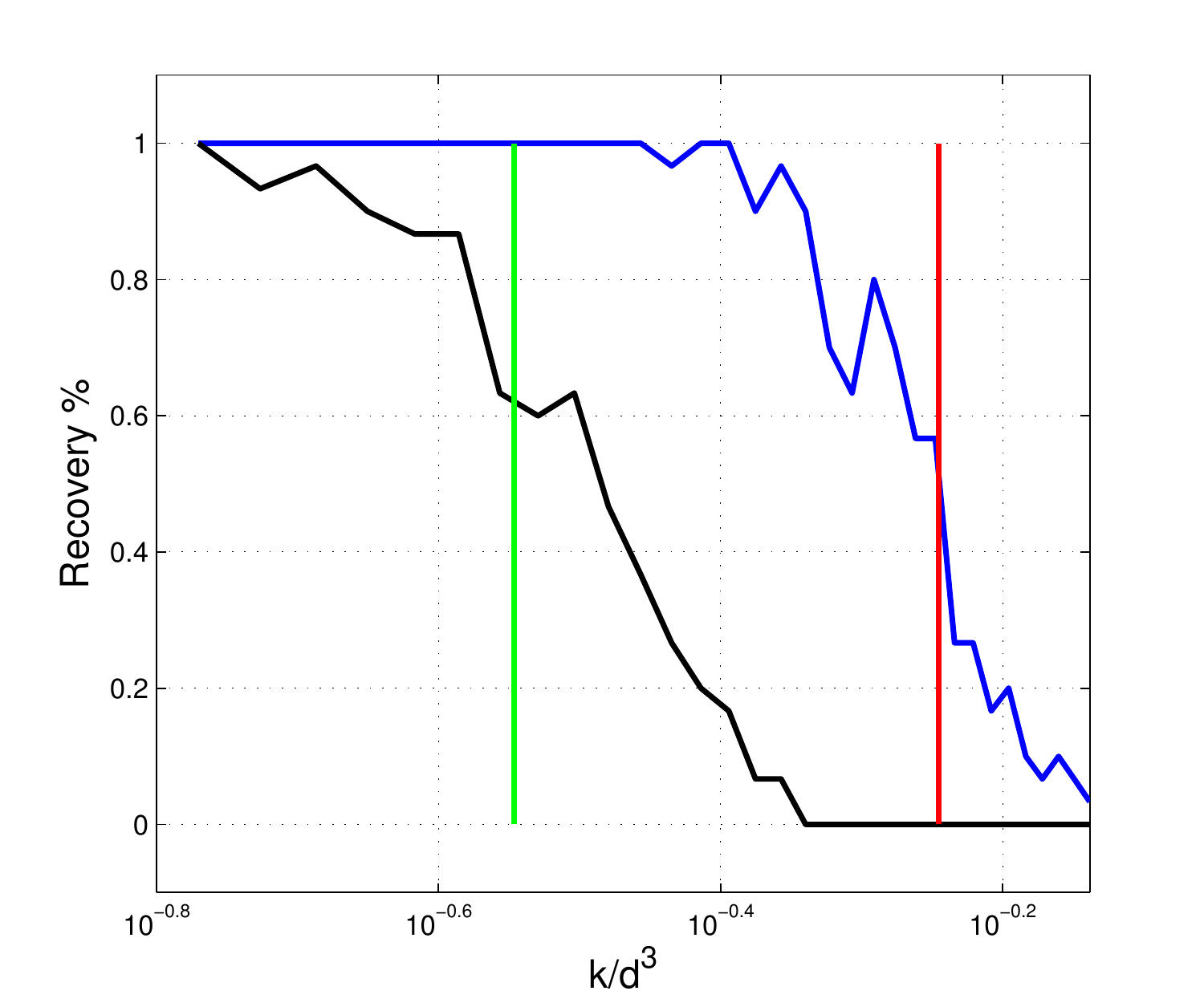}
}
\caption{Empirical probability ($6\times 30$ trials) of exact recovery by total variation minimization \eqref{eq:TVmin_pos}
via the unperturbed (blue line) and perturbed (black line) matrices from Section \ref{sec:setup_3D}, 
Fig. \ref{fig:3C3D} and Fig. \ref{fig:4C3D}, 
for 3 (top row) and 4 (bottom row) projecting directions respectively. Hereby $d = 41$. The significance of
each curve is identically to the one in Fig.~\ref{fig:res3D_d31}. Scaling factors are summarized in Table \ref{tab:22},
}
\label{fig:res3D_d41}
\end{figure}


\begin{figure}[htbp]
\centerline{
\includegraphics[width=0.8\textwidth]{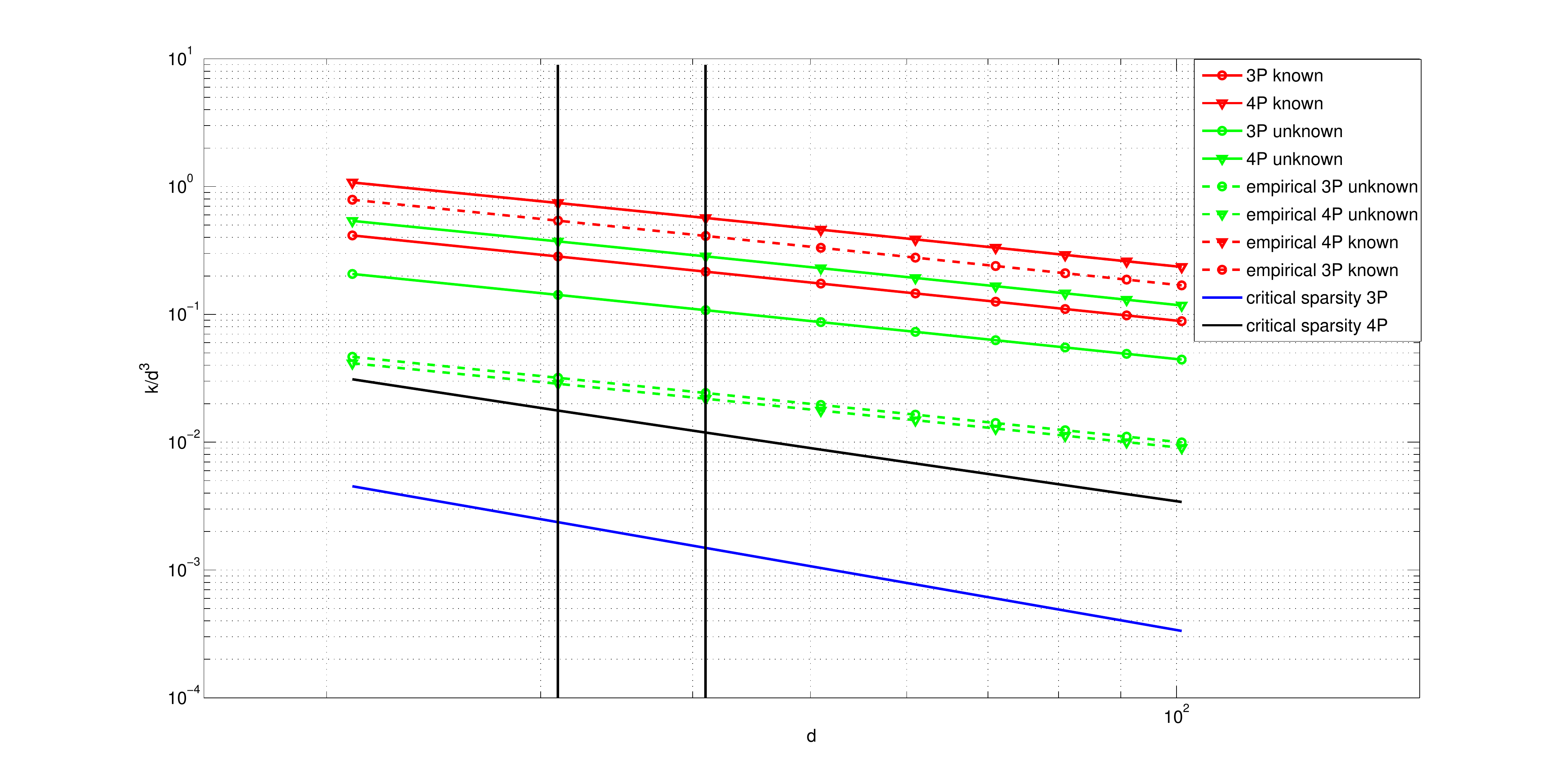}
}
\centerline{
\includegraphics[width=0.8\textwidth]{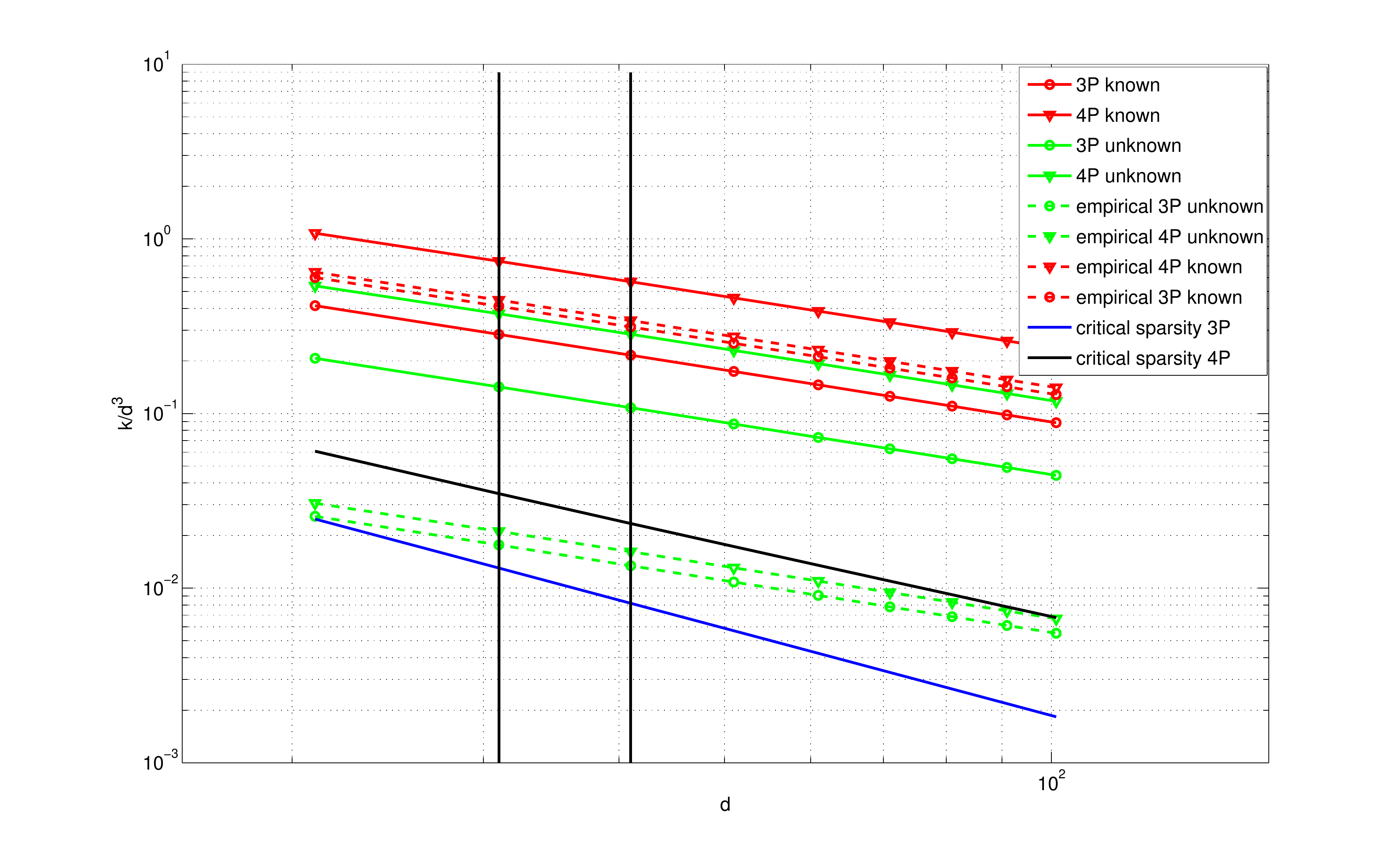}
}
\caption{Log-log plot of phase transitions in 3D for the unperturbed matrix $A$ (top), and perturbed matrix $\tilde A$ (bottom) 
for 3 ($\circ$-marked curves) and 4 projecting directions ($\triangledown$-marked curves). 
The continuous green and red lines depict the theoretical curves \eqref{eq:m-lb-unknown-3D} and \eqref{eq:m-lb-known-3D} respectively. The dashed lines correspond to the empirical thresholds, which are all scaled versions of \eqref{eq:m-lb-unknown-3D} or \eqref{eq:m-lb-known-3D}
with scaling factors summarized in Table. \ref{tab:22}. The blue (stands for 3 projecting directions) and black (stands for 4 projecting directions)
curves show the relative critical sparsity such that $k$ random points are recovered \emph{exactly} by \eqref{eq:l1min}.
These are the theoretical phase transition $\ell_1$-recovery from \cite{Petra-Schnoerr-LAA}, \cite{Petra2013b}. The vertical lines correspond to 
$d=31$ and $d=41$, compare with Fig.~\ref{fig:res3D_d31} and Fig.~\ref{fig:res3D_d41}.}
\label{fig:3D_phase_trans}
\end{figure}


\begin{table}[hbtp]

	\begin{tabular}{|c|c|c|c|c|}
		\hline
		\multicolumn{5}{|c|}{$\alpha$-values in 3D} \\
		\hline
			$\#$ proj. dir. & Cosupport & Measurements & $d=31$ & $d = 41$ \\
	    \hline
	 \multirow {4}{*}{3} & \multirow{2}{*}{Known}  & unperturbed & 1.8568 & 1.9527 \\ \cline{3-5}
       				     & 					       & perturbed   & 1.6137 & 1.2860 \\ \cline{2-5}
        					 & \multirow{2}{*}{Unknown}& unperturbed & 0.4910 & 0.4086 \\ \cline{3-5}
      				     & 					       & perturbed   & 0.2098 & 0.2882 \\ \cline{1-5}
     \multirow {4}{*}{4} & \multirow{2}{*}{Known}  & unperturbed & 1 	   & 1		\\ \cline{3-5}
       				     & 					       & perturbed   & 0.6153 & 0.5833 \\ \cline{2-5}
       				     & \multirow{2}{*}{Unknown}& unperturbed & 0.1526 & 0.1552	\\ \cline{3-5}
       				     & 					       & perturbed   & 0.1180 & 0.1094 \\       
     \hline
	\end{tabular}
	
\vspace{2mm}
\caption{The scaling factors of the theoretical curves \eqref{eq:m-lb-known-3D}
and \eqref{eq:m-lb-unknown-3D} for known and unknown cosupport respectively.
}
\label{tab:22}
\end{table}

\clearpage
\subsection{Discussion}
Several observations are in order.
\begin{itemize}
\item Perturbation of projection matrices brings no significant advantage
in the practically relevant case of unknown co-support.
The empirical transitions will remain the same for perturbed and unperturbed matrices.
This is very different to the $\ell_1$-minimization
problem \eqref{eq:l1min}, where perturbation boosts the recovery performance
significantly as shown in \cite{Petra-Schnoerr-LAA}.

\item In the case of known co-support, when $B_\Lambda u = 0$ is added as additional constraint,
unperturbed matrices perform better. We notice that the
empirical phase transition is \emph{above} the red curve, and deduce that linear 
dependencies might be beneficial when the co-support is known.

\item When increasing the number of projecting directions (4,5,6 or more) 
the differences between estimated (dashed) and theoretical (continuous line) phase transition
become smaller. This might be due to the fact that linear dependencies between the
columns (and rows) of $A$ become ``rare'', and the assumptions of Propositions \ref{prop:unique-cs-known} and \ref{prop:unique-cs-unknown} are more likely to be satisfied.

\item In 3D the difference between empirical phase transitions
for 3 and 4 projecting directions is very small, i.e.
relative phase transitions are almost equal. This is different
to the 2D case above. We currently do not have an explanation for this phenomenon.

\item The log-log plot in Figure \ref{fig:3D_phase_trans} shows
that phase transitions in 3D exhibit a power law behavior, similar to
the theoretical phase transitions for $\ell_1$-recovery from \cite{Petra-Schnoerr-LAA}, \cite{Petra2013b}.
Moreover, the plot also shows the scaling exponent of the green and red curves is higher, which results
in significantly higher sparsity levels of the image gradient then image sparsity
which allow exact recovery for big volumes and large $d$. 

\end{itemize}

\section{Conclusion}
\label{sec:Conclusions}
We studied the cosparsity model in order to theoretically investigate conditions for unique signal recovery from severely undersampled linear systems, that involve measurement matrices whose properties fall far short of the assumptions commonly made in the compressed sensing literature. Extensive numerical experiments revealed a high accuracy of the theoretical predictions, up to a scale factor caused by slight violations in practice of our mathematical assumptions. Unique recovery can be accomplished by linear programming that in principle copes with large problem sizes. The signal class covered by the cosparsity model seems broad enough to cover relevant industrial applications of non-standard tomography, like contactless quality inspection.

In our future work, we will aim at clarifying quantitatively the above-mentioned scale factor and its origin. In the same context, conducting a probabilistic analysis as in our recent work \cite{Petra-Schnoerr-LAA}), for the present scenarios, defines an open problem. We expect that the refinement of a probabilistic version of the cosparsity model, in connection with distributions of cosupports learned from relevant collections of signals, may have an impact both theoretically and practically beyond aspects of limited-angle tomography.

\vspace{1cm}
\noindent
\textbf{Acknowledgement.} SP gratefully acknowledges financial support from the Ministry of Science, Research and Arts, Baden-W\"{u}rttemberg, within the Margarete von Wrangell postdoctoral lecture qualification program. AD and the remaining authors appreciate financial support of this project by the Bayerische Forschungsstiftung. 

\bibliographystyle{amsalpha}

\end{document}